\documentclass{amsart}
\usepackage{amsmath}
\usepackage{amssymb}
\usepackage{amsbsy, amsthm,amstext, amsopn}
\usepackage[all]{xy}
\usepackage{amsfonts}
\usepackage{amscd}
\usepackage{stmaryrd}
\usepackage{color}
\usepackage[numbers, square, sort, comma]{natbib}

\usepackage{graphicx}

\newtheorem{thm}{Theorem} [section]
\newtheorem{lemma}[thm]{Lemma}

\newtheorem{corollary}[thm]{Corollary}
\newtheorem{prop}[thm]{Proposition}

\theoremstyle{definition}

\newtheorem{defn}[thm]{Definition}

\newtheorem{example}[thm]{Example}

\newtheorem{claim}[thm]{Claim}

\newtheorem{remark}[thm]{Remark}

\newtheorem{assumption}[thm]{Assumption}

\begin{document}

\numberwithin{equation}{section}

\newcommand{\hs}{\mbox{\hspace{.4em}}}
\newcommand{\ds}{\displaystyle}
\newcommand{\bd}{\begin{displaymath}}
\newcommand{\ed}{\end{displaymath}}
\newcommand{\bcd}{\begin{CD}}
\newcommand{\ecd}{\end{CD}}

\newcommand{\on}{\operatorname}
\newcommand{\proj}{\operatorname{Proj}}
\newcommand{\bproj}{\underline{\operatorname{Proj}}}

\newcommand{\spec}{\operatorname{Spec}}
\newcommand{\Spec}{\operatorname{Spec}}
\newcommand{\bspec}{\underline{\operatorname{Spec}}}
\newcommand{\pline}{{\mathbf P} ^1}
\newcommand{\aline}{{\mathbf A} ^1}
\newcommand{\pplane}{{\mathbf P}^2}
\newcommand{\aplane}{{\mathbf A}^2}
\newcommand{\coker}{{\operatorname{coker}}}
\newcommand{\ldb}{[[}
\newcommand{\rdb}{]]}

\newcommand{\Sym}{\operatorname{Sym}^{\bullet}}
\newcommand{\Symp}{\operatorname{Sym}}
\newcommand{\Pic}{\bf{Pic}}
\newcommand{\Aut}{\operatorname{Aut}}
\newcommand{\PAut}{\operatorname{PAut}}

\newcommand{\too}{\twoheadrightarrow}
\newcommand{\C}{{\mathbf C}}
\newcommand{\Z}{{\mathbf Z}}
\newcommand{\Q}{{\mathbf Q}}
\newcommand{\Cx}{{\mathbf C}^{\times}}
\newcommand{\Cbar}{\overline{\C}}
\newcommand{\Cxbar}{\overline{\Cx}}
\newcommand{\cA}{{\mathcal A}}
\newcommand{\cS}{{\mathcal S}}
\newcommand{\cV}{{\mathcal V}}
\newcommand{\cM}{{\mathcal M}}
\newcommand{\bA}{{\mathbf A}}
\newcommand{\cB}{{\mathcal B}}
\newcommand{\cC}{{\mathcal C}}
\newcommand{\cD}{{\mathcal D}}
\newcommand{\D}{{\mathcal D}}
\newcommand{\cs}{{\mathbf C} ^*}
\newcommand{\boldc}{{\mathbf C}}
\newcommand{\cE}{{\mathcal E}}
\newcommand{\cF}{{\mathcal F}}
\newcommand{\bF}{{\mathbf F}}
\newcommand{\cG}{{\mathcal G}}
\newcommand{\G}{{\mathbb G}}
\newcommand{\cH}{{\mathcal H}}
\newcommand{\CI}{{\mathcal I}}
\newcommand{\cJ}{{\mathcal J}}
\newcommand{\cK}{{\mathcal K}}
\newcommand{\cL}{{\mathcal L}}
\newcommand{\baL}{{\overline{\mathcal L}}}

\newcommand{\fF}{{\mathfrak F}}
\newcommand{\Mf}{{\mathfrak M}}
\newcommand{\bM}{{\mathbf M}}
\newcommand{\bm}{{\mathbf m}}
\newcommand{\cN}{{\mathcal N}}
\newcommand{\theo}{\mathcal{O}}
\newcommand{\cP}{{\mathcal P}}
\newcommand{\cR}{{\mathcal R}}
\newcommand{\Pp}{{\mathbb P}}
\newcommand{\boldp}{{\mathbf P}}
\newcommand{\boldq}{{\mathbf Q}}
\newcommand{\bbL}{{\mathbf L}}
\newcommand{\cQ}{{\mathcal Q}}
\newcommand{\cO}{{\mathcal O}}
\newcommand{\Oo}{{\mathcal O}}
\newcommand{\cY}{{\mathcal Y}}
\newcommand{\OX}{{\Oo_X}}
\newcommand{\OY}{{\Oo_Y}}
\newcommand{\otY}{{\underset{\OY}{\ot}}}
\newcommand{\otX}{{\underset{\OX}{\ot}}}
\newcommand{\cU}{{\mathcal U}}\newcommand{\cX}{{\mathcal X}}
\newcommand{\cW}{{\mathcal W}}
\newcommand{\boldz}{{\mathbf Z}}
\newcommand{\qgr}{\operatorname{q-gr}}
\newcommand{\gr}{\operatorname{gr}}
\newcommand{\rk}{\operatorname{rk}}
\newcommand{\SH}{{\underline{\operatorname{Sh}}}}
\newcommand{\End}{\operatorname{End}}
\newcommand{\uEnd}{\underline{\operatorname{End}}}
\newcommand{\Hom}{\operatorname{Hom}}
\newcommand{\uHom}{\underline{\operatorname{Hom}}}
\newcommand{\uHomY}{\uHom_{\OY}}
\newcommand{\uHomX}{\uHom_{\OX}}
\newcommand{\Ext}{\operatorname{Ext}}
\newcommand{\bExt}{\operatorname{\bf{Ext}}}
\newcommand{\Tor}{\operatorname{Tor}}

\newcommand{\inv}{^{-1}}
\newcommand{\airtilde}{\widetilde{\hspace{.5em}}}
\newcommand{\airhat}{\widehat{\hspace{.5em}}}
\newcommand{\nt}{^{\circ}}
\newcommand{\del}{\partial}

\newcommand{\supp}{\operatorname{supp}}
\newcommand{\GK}{\operatorname{GK-dim}}
\newcommand{\hd}{\operatorname{hd}}
\newcommand{\id}{\operatorname{id}}
\newcommand{\res}{\operatorname{res}}
\newcommand{\lrar}{\leadsto}
\newcommand{\im}{\operatorname{Im}}
\newcommand{\HH}{\operatorname{H}}
\newcommand{\TF}{\operatorname{TF}}
\newcommand{\Bun}{\operatorname{Bun}}

\newcommand{\F}{\mathcal{F}}
\newcommand{\Ff}{\mathbb{F}}
\newcommand{\nthord}{^{(n)}}
\newcommand{\Gr}{{\mathfrak{Gr}}}

\newcommand{\Fr}{\operatorname{Fr}}
\newcommand{\GL}{\operatorname{GL}}
\newcommand{\gl}{\mathfrak{gl}}
\newcommand{\SL}{\operatorname{SL}}
\newcommand{\ff}{\footnote}
\newcommand{\ot}{\otimes}
\def\Ext{\operatorname {Ext}}
\def\Hom{\operatorname {Hom}}
\def\Ind{\operatorname {Ind}}
\def\bbZ{{\mathbb Z}}

\newcommand{\nc}{\newcommand}
\nc{\ol}{\overline} \nc{\cont}{\on{cont}} \nc{\rmod}{\on{mod}}
\nc{\Mtil}{\widetilde{M}} \nc{\wb}{\overline} 
\nc{\wh}{\widehat}  \nc{\mc}{\mathcal}
\nc{\mbb}{\mathbb}  \nc{\K}{{\mc K}} \nc{\Kx}{{\mc K}^{\times}}
\nc{\Ox}{{\mc O}^{\times}} \nc{\unit}{{\bf \on{unit}}}
\nc{\boxt}{\boxtimes} \nc{\xarr}{\stackrel{\rightarrow}{x}}

\nc{\Ga}{\G_a}
 \nc{\PGL}{{\on{PGL}}}
 \nc{\PU}{{\on{PU}}}

\nc{\h}{{\mathfrak h}} \nc{\kk}{{\mathfrak k}}
 \nc{\Gm}{\G_m}
\nc{\Gabar}{\wb{\G}_a} \nc{\Gmbar}{\wb{\G}_m} \nc{\Gv}{G^\vee}
\nc{\Tv}{T^\vee} \nc{\Bv}{B^\vee} \nc{\g}{{\mathfrak g}}
\nc{\gv}{{\mathfrak g}^\vee} \nc{\BRGv}{\on{Rep}\Gv}
\nc{\BRTv}{\on{Rep}T^\vee}
 \nc{\Flv}{{\mathcal B}^\vee}
 \nc{\TFlv}{T^*\Flv}
 \nc{\Fl}{{\mathfrak Fl}}
\nc{\BRR}{{\mathcal R}} \nc{\Nv}{{\mathcal{N}}^\vee}
\nc{\St}{{\mathcal St}} \nc{\ST}{{\underline{\mathcal St}}}
\nc{\Hec}{{\bf{\mathcal H}}} \nc{\Hecblock}{{\bf{\mathcal
H_{\alpha,\beta}}}} \nc{\dualHec}{{\bf{\mathcal H^\vee}}}
\nc{\dualHecblock}{{\bf{\mathcal H^\vee_{\alpha,\beta}}}}
\newcommand{\ramBun}{{\bf{Bun}}}
\newcommand{\ramBuno}{\ramBun^{\circ}}

\nc{\Buntheta}{{\bf Bun}_{\theta}} \nc{\Bunthetao}{{\bf
Bun}_{\theta}^{\circ}} \nc{\BunGR}{{\bf Bun}_{G_\BR}}
\nc{\BunGRo}{{\bf Bun}_{G_\BR}^{\circ}}
\nc{\HC}{{\mathcal{HC}}}
\nc{\risom}{\stackrel{\sim}{\to}} \nc{\Hv}{{H^\vee}}
\nc{\bS}{{\mathbf S}}
\def\BRep{\operatorname {Rep}}
\def\Conn{\operatorname {Conn}}

\nc{\Vect}{{\operatorname{Vect}}}
\nc{\Hecke}{{\operatorname{Hecke}}}

\newcommand{\ZZ}{{Z_{\bullet}}}
\nc{\HZ}{{\mc H}\ZZ} \nc{\eps}{\epsilon}

\nc{\CN}{\mathcal N} \nc{\BA}{\mathbb A}

 \nc{\BB}{\mathbb B}

\nc{\ul}{\underline}

\nc{\bn}{\mathbf n} \nc{\Sets}{{\on{Sets}}} \nc{\Top}{{\on{Top}}}
\nc{\IntHom}{{\mathcal Hom}}

\nc{\Simp}{{\mathbf \Delta}} \nc{\Simpop}{{\mathbf\Delta^\circ}}

\nc{\Cyc}{{\mathbf \Lambda}} \nc{\Cycop}{{\mathbf\Lambda^\circ}}

\nc{\Mon}{{\mathbf \Lambda^{mon}}}
\nc{\Monop}{{(\mathbf\Lambda^{mon})\circ}}

\nc{\Aff}{{\on{Aff}}} \nc{\Sch}{{\on{Sch}}}

\nc{\bul}{\bullet}
\nc{\module}{{\operatorname{-mod}}}

\nc{\dstack}{{\mathcal D}}

\nc{\BL}{{\mathbb L}}

\nc{\BD}{{\mathbb D}}

\nc{\BR}{{\mathbb R}}

\nc{\BT}{{\mathbb T}}

\nc{\SCA}{{\mc{SCA}}}
\nc{\DGA}{{\mc DGA}}

\nc{\DSt}{{DSt}}

\nc{\lotimes}{{\otimes}^{\mathbf L}}

\nc{\bs}{\backslash}

\nc{\Lhat}{\widehat{\mc L}}

\newcommand{\Coh}{\on{Coh}}

\nc{\QCoh}{QC}
\nc{\QC}{QC}
\nc{\Perf}{\on{Perf}}
\nc{\Cat}{{\on{Cat}}}
\nc{\dgCat}{{\on{dgCat}}}
\nc{\bLa}{{\mathbf \Lambda}}

\nc{\BRHom}{\mathbf{R}\hspace{-0.15em}\on{Hom}}
\nc{\BREnd}{\mathbf{R}\hspace{-0.15em}\on{End}}
\nc{\colim}{\on{colim}}
\nc{\oo}{\infty}
\nc{\Mod}{\on{Mod} }

\nc\fh{\mathfrak h}
\nc\al{\alpha}
\nc\la{\alpha}
\nc\BGB{B\bs G/B}
\nc\QCb{QC^\flat}
\nc\qc{\on{QC}}

\def\w{\wedge}
\nc{\vareps}{\varepsilon}

\nc{\fg}{\mathfrak g}

\nc{\Map}{\on{Map}} \nc{\fX}{\mathfrak X}

\nc{\ch}{\check}
%\nc{\fg}{\mathfrak g}
\nc{\fb}{\mathfrak b} \nc{\fu}{\mathfrak u} \nc{\st}{{st}}
\nc{\fU}{\mathfrak U}
\nc{\fZ}{\mathfrak Z}
\nc{\fB}{\mathfrak B}

 \nc\fc{\mathfrak c}
 \nc\fs{\mathfrak s}

\nc\fk{\mathfrak k} \nc\fp{\mathfrak p}
\nc\fq{\mathfrak q}

\nc{\BRP}{\mathbf{RP}} \nc{\rigid}{\text{rigid}}
\nc{\glob}{\text{glob}}

\nc{\cI}{\mathcal I}

\nc{\La}{\mathcal L}

\nc{\quot}{/\hspace{-.25em}/}

\nc\aff{\it{aff}}
\nc\BS{\mathbb S}

\nc\Loc{{\mc Loc}}
\nc\Tr{{\on{Tr}}}
\nc\Ch{{\mc Ch}}

\nc\ftr{{\mathfrak {tr}}}
\nc\fM{\mathfrak M}

\nc\Id{\operatorname{Id}}

\nc\bimod{\on{-bimod}}

\nc\ev{\operatorname{ev}}
\nc\coev{\operatorname{coev}}

\nc\pair{\operatorname{pair}}
\nc\kernel{\operatorname{kernel}}

\nc\Alg{\operatorname{Alg}}

\nc\init{\emptyset_{\text{\em init}}}
\nc\term{\emptyset_{\text{\em term}}}

\nc\Ev{\on{Ev}}
\nc\Coev{\on{Coev}}

\nc\es{\emptyset}
\nc\m{\text{\it min}}
\nc\M{\text{\it max}}
\nc\cross{\text{\it cr}}
\nc\tr{\on{tr}}

\nc\perf{\on{-perf}}
\nc\inthom{\mathcal Hom}
\nc\intend{\mathcal End}

\newcommand{\Sh}{\mathit{Sh}}

\nc{\Comod}{\on{Comod}}
\nc{\cZ}{\mathcal Z}

\def\interiorsymbol {\on{int}}

\nc\frakf{\mathfrak f}
\nc\fraki{\mathfrak i}
\nc\frakj{\mathfrak j}
\nc\BP{\mathbb P}
\nc\stab{st}
\nc\Stab{St}

\nc\fN{\mathfrak N}
\nc\fT{\mathfrak T}
\nc\fV{\mathfrak V}

\nc\Ob{\on{Ob}}

\nc\fC{\mathfrak C}
\nc\Fun{\on{Fun}}

\nc\Null{\on{Null}}

\nc\BC{\mathbb C}

\nc\loc{\on{Loc}}

\nc\hra{\hookrightarrow}
\nc\fL{\mathfrak L}
\nc\R{\mathbb R}
\nc\CE{\mathcal E}

\nc\sK{\mathsf K}
\nc\sL{\mathsf L}
\nc\sC{\mathsf C}

\nc\Cone{\mathit Cone}

\nc\fY{\mathfrak Y}
\nc\fe{\mathfrak e}
\nc\ft{\mathfrak t}

\nc\wt{\widetilde}
\nc\inj{\mathit{inj}}
\nc\surj{\mathit{surj}}

\nc\Path{\mathit{Path}}
\nc\Set{\mathit{Set}}
\nc\Fin{\mathit{Fin}}

\nc\cyc{\mathit{cyc}}

\nc\per{\mathit{per}}

\nc\sym{\mathit{symp}}
\nc\con{\mathit{cont}}
\nc\gen{\mathit{gen}}
\nc\str{\mathit{str}}
\nc\rsdl{\mathit{res}}
\nc\impr{\mathit{impr}}
\nc\rel{\mathit{rel}}
\nc\pt{\mathit{pt}}
\nc\naive{\mathit{nv}}
\nc\forget{\mathit{For}}

\nc\sH{\mathsf H}
\nc\sW{\mathsf W}
\nc\sE{\mathsf E}
\nc\sP{\mathsf P}
\nc\sB{\mathsf B}
\nc\sS{\mathsf S}
\nc\fH{\mathfrak H}
\nc\fP{\mathfrak P}
\nc\fW{\mathfrak W}
\nc\fE{\mathfrak E}
\nc\sx{\mathsf x}
\nc\sy{\mathsf y}

\nc\ord{\mathit{ord}}

\nc\sm{\mathit{sm}}

\nc\rhu{\rightharpoonup}
\nc\dirT{\mathcal T}
\nc\dirF{\mathcal F}
\nc\link{\mathit{link}}
\nc\cT{\mathcal T}

\newcommand{\ssupp}{\mathit{ss}}
\newcommand{\cyl}{\mathit{Cyl}}
\newcommand{\ball}{\mathit{B(x)}}

 \nc\ssf{\mathsf f}
 \nc\ssg{\mathsf g}
\nc\sq{\mathsf q}
 \nc\sQ{\mathsf Q}
 \nc\sR{\mathsf R}

\nc\fa{\mathfrak a}
\nc\fA{\mathfrak A}

\nc\trunc{\mathit{tr}}
\nc\pre{\mathit{pre}}
\nc\expand{\mathit{exp}}

\nc\Sol{\mathit{Sol}}
\nc\direct{\mathit{dir}}

\nc\out{\mathit{out}}
\nc\Morse{\mathit{Morse}}
\nc\arb{\mathit{arb}}
\nc\prearb{\mathit{pre}}

%%%%%%%%%%%%%%%%%%%%%%%%%%%%%%%%%%%%%%%%%%%%%%%%%%%%%%%
%%%%%%%%%%%%%%%%%%%%%%%%%%%%%%%%%%%%%%%%%%%%%%%%%%%%%%%

\title[Non-characteristic expansions of Legendrian singularities]{Non-characteristic expansions\\ of Legendrian singularities}

\author{David Nadler}
\address{Department of Mathematics\\University of California, Berkeley\\Berkeley, CA  94720-3840}
\email{nadler@math.berkeley.edu}

\begin{abstract}

This paper presents an algorithm to deform any Legendrian singularity to a nearby Legendrian subvariety with  singularities of a simple combinatorial nature. Furthermore, the category of microlocal sheaves on the original Legendrian singularity is  equivalent to that on the nearby Legendrian subvariety. 
This yields a concrete  combinatorial model for microlocal sheaves,
as well as an elementary method for calculating them.
%%This refines and answers a question of Kontsevich inspired by homological mirror symmetry. It also contributes to  a parallel program  in microlocal sheaf theory pioneered   by MacPherson.
\end{abstract}

\maketitle

%%%%%%%%%%%%%%%%%%%%%%%%%%%%%%%%%%%%%%%%%%%%%%%%%%%%%%%
%%%%%%%%%%%%%%%%%%%%%%%%%%%%%%%%%%%%%%%%%%%%%%%%%%%%%%%

\tableofcontents

%%%%%%%%%%%%%%%%%%%%%%%%%%%%%%%%%%%%%%%%%%%%%%%%%%%%%%%
%%%%%%%%%%%%%%%%%%%%%%%%%%%%%%%%%%%%%%%%%%%%%%%%%%%%%%%
%%%%%%%%%%%%%%%%%%%%%%%%%%%%%%%%%%%%%%%%%%%%%%%%%%%%%%%

\section{Introduction}

This paper presents an algorithm to deform any Legendrian singularity to a nearby Legendrian subvariety with  singularities of a simple combinatorial nature. Furthermore, the category of microlocal sheaves, as developed by Kashiwara-Schapira~\cite{KS}, on the original Legendrian singularity is  equivalent to that on the nearby Legendrian subvariety. 
%This yields a concrete  combinatorial model for microlocal sheaves,
%as well as an elementary method for calculating them.
This yields a concrete  combinatorial  model for microlocal sheaves, in terms of modules over quivers, as well as an elementary method for calculating them.

Among other applications, it allows one to establish new homological mirror symmetry equivalences for Landau-Ginzburg models with singular thimbles~\cite{Nlg3d}.
% via  where the $A$-model is taken to be microlocal sheaves. 
It also provides a key tool  in the foundations of microlocal sheaves, in particular in the recent development of its wrapped variant~\cite{Nwms}. Furtheremore, it underlies ongoing work on
 the structure of Weinstein manifolds~\cite{ENS} and their Fukaya categories~\cite{GPS}.

In the rest of the introduction, we first state the main results of the paper, then sketch some of the arguments involved in their proof. Finally, we elaborate further on their place in  the subject, in particular their original motivation in conjectures of Kontsevich~\cite{kont} and a program of MacPherson and collaborators (see for example~\cite{BG, GMV}) devoted  to combinatorial models of respectively Fukaya categories and microlocal sheaves.

%%%%%%%%%%%%%%%%%%%%%%%%%%%%%%%%%%%%%%%%%%%%%%%%%%%%%%%

\subsection{Main results}

A natural setting for the paper is local contact geometry. (See Sects.~\ref{s prelim} and \ref{s dir} below for detailed geometric preliminaries.) % (or equivalently local conic symplectic geometry).

Let $M$ be a smooth manifold with cotangent bundle  $T^*M$ with its canonical  exact symplectic structure. 
 Introduce the  cosphere bundle
$$
\xymatrix{
\pi:S^*M = (T^*M \setminus M)/\BR_{>0} \ar[r] & M
}
$$
with its canonical cooriented contact structure. 
By the contact Darboux theorem, any contact manifold is locally equivalent to $S^*M$. 

By a Legendrian subvariety $\Lambda \subset S^*M$, we will mean a closed Whitney stratified subspace of pure dimension 
$\dim\Lambda = \dim M - 1$ whose strata are isotropic for the contact structure.
By a Legendrian singularity  $\Lambda \subset S^*M$ , we will mean the germ of a Legendrian subvariety at a point $ z\in\Lambda$  which we refer to as its center.

We will assume that any Legendrian singularity $\Lambda \subset S^*M$ we encounter can be placed in generic position in the sense that the front projection $H = \pi(\Lambda) \subset M$ is the germ at the image of the center $x = \pi(z)$ of a  Whitney stratified hypersurface, and the restriction 
$$
\xymatrix{
\pi|_{\Lambda}:\Lambda\ar[r] &  H
}
$$ 
is a finite map.
(To simplify the exposition, we will also assume the Whitney stratification of $H$ satisfies some modest local connectivity which can always be arranged by refining the stratification, see Sect.~\ref{ss: strat} for details).

Fix a field $k$, and following Kashiwara-Schapira~\cite{KS} (and reviewed in Sect.~\ref{s inv} below) given a Legendrian subvariety $\Lambda\subset S^*M$,  introduce the dg category $\mu\Sh_\Lambda(S^*M)$ of constructible microlocal complexes of $k$-modules on $S^*M$  supported along $\Lambda$. (The results and arguments of the paper work equally well without the constructible hypothesis, but we include it to keep within a traditional setting.)
In particular, for a Legendrian singularity $\Lambda\subset S^*M$ in generic position, and a small open ball $B \subset M$ around the image of the center $x = \pi(z)$, 
there is a canonical quotient presentation 
$$
\xymatrix{
\mu\Sh_{\Lambda}(S^*M) & \ar[l]_-\sim \Sh_\Lambda(B)/\Loc(B)
}
$$ in terms of the more concrete dg category $\Sh_\Lambda(B)$ of constructible  complexes on  $B$ with singular support in $\Lambda$, and its full dg subcategory $\Loc(B)$ of finite-rank derived local systems.

\medskip

 In the paper~\cite{Narb}, we introduced a natural class of Legendrian singularities $\Lambda_\cT\subset S^* \BR^{|\cT|}$, called arboreal singularities, indexed by rooted trees $\cT$ (finite connected acyclic graphs with a choice of root vertex), where $|\cT|$ denotes the number of vertices of the underlying tree.
 Each is naturally stratified by strata $\Lambda_\cT(\fp) \subset \Lambda_\cT$  
 indexed by correspondences of trees
$$
\xymatrix{
\fp = (\cR &  \ar@{->>}[l]_-q  \cS \ar@{^(->}[r]^-i & \cT)
}
$$
with $i$ an inclusion of a full subtree, and $q$ a quotient by collapsing edges. 
Moreover, the normal geometry to the stratum $\Lambda_\cT(\fp) \subset \Lambda_\cT$  is equivalent via Hamiltonian reduction to the arboreal singularity $\Lambda_\cR\subset S^* \BR^{|\cR|}$.

Passing to microlocal sheaves, we constructed in~\cite{Narb}  a canonical equivalence
$$
\xymatrix{
\mu\Sh_{\Lambda_{\cT}}(S^* \BR^{|\cT|}) \simeq \Perf(\cT)
}
$$
where  $ \Perf(\cT)$ denotes the dg category of perfect complexes over the quiver associated to the rooted tree $\cT$ where all edges are directed away from the root vertex.
We also showed that for each stratum $\Lambda_\cT(\fp) \subset \Lambda_\cT$, the corresponding   microlocal restriction functor
$$
\xymatrix{
\mu\Sh_{\Lambda_{\cT}}(S^* \BR^{|\cT|}) \ar[r] & \mu\Sh_{\Lambda_{\cR}}(S^* \BR^{|\cR|})
}
$$
 is given by the integral transform
$$
\xymatrix{
\Perf(\dirT) \ar@{->>}[r]^-{i^*} & \Perf(\mathcal S) \ar@{->>}[r]^-{q_!} &  \Perf(\mathcal R)
}$$
where $i^*$ kills the projective object $P_\alpha\in\Perf(\dirT)$  attached to  $\alpha \in \cT \setminus \cS$,
and
and $q_!$ identifies  the projective objects $P_\alpha, P_\beta\in\Perf(\cS)$  attached to $\alpha, \beta\in \cS$ such that
$q(\alpha) = q(\beta) \in \cR$.

 In Sect.~\ref{s arb} below, we review the basic properties of arboreal singularities, and introduce natural generalizations 
$\Lambda_{\cT^*}\subset S^* \BR^{|\cT^*|}$  indexed by leafy rooted trees  $\cT^*$ (finite connected acyclic graphs with a choice of root vertex
and a subset of leaf vertices), where $|\cT^*|$ denotes the sum of the number of vertices of the underlying tree  and the number of marked leaves. Similar statements to those recalled above hold for generalized arboreal singularities.

By an arboreal  Legendrian subvariety  $\Lambda \subset S^*M$, we will mean a Legendrian subvariety  such that  its normal geometry along each of its strata  is  equivalent via Hamiltonian reduction   to 
 an arboreal singularity $\Lambda_{\cT^*}\subset S^* \BR^{|\cT^*|}$.
 Thus for an arboreal Legendrian subvariety $\Lambda\subset S^*M$,  the dg category  
of microlocal sheaves $\mu\Sh_\Lambda(S^*M)$   admits an elementary description: it is the global sections of a sheaf of dg categories over $\Lambda \subset S^* M$  that assigns perfect modules over trees to small open sets,     integral transforms 
for correspondences of trees to inclusions of small open sets, along with natural higher coherences.

\medskip

Now to state  the main result of this paper, we need one further key concept.

By a deformation 
of a Legendrian singularity $\Lambda\subset S^*M$ to a Legendrian subvariety $\Lambda' \subset S^*M$, we will mean a Legendrian subvariety $\tilde \Lambda\subset S^*(M\times \BR)$ supported over the parameters $\BR_{\geq 0}$ with special Hamiltonian reduction $\tilde \Lambda_0 = \Lambda \subset S^*M$ the original Legendrian singularity, and generic Hamiltonian reductions $\tilde \Lambda_t\simeq \Lambda' \subset S^*M$ all equivalent for $t> 0$.

It is not true in general that microlocal sheaves will be constant under such deformations, 
even if we insist on strong topological properties. (See Example~\ref{ex non-char} below.)
By a non-characteristic  deformation, we will mean one for which Hamiltonian reduction at any of the parameters $\BR_{\geq 0}$ induces an equivalence on microlocal sheaves.

\begin{thm}\label{thm: intro} %[Theorems~\ref{thm dir exp is arb}, \ref{thm main result}]
Let $\Lambda\subset S^*M$ be a Legendrian singularity. 

The expansion algorithm  of Sect.~\ref{s exp} provides a non-characteristic deformation of $\Lambda\subset S^*M$ to an arboreal Legendrian subvariety $\Lambda_\arb \subset S^*M$. \end{thm}

%\begin{remark}
Our use of the term  expansion reflects the idea that we perform a kind of ``spherical real blowup" to expand complicated singularities into irreducible components which then interact in a combinatorial way. One could  compare this with resolutions of singularities in algebraic geometry where complicated singularities become divisors with normal crossings. 
%Proving the expansion algorithm  is non-characteristic  in the sense the that the dg category of microlocal sheaves is invariant is the most technically involved part of the paper. 
Proving the expansion algorithm is non-characteristic occupies Sect.~\ref{s inv} below
and is the most technically involved part of the paper. 
%\end{remark}

\begin{remark}
One can   easily refine the expansion algorithm  so that $\Lambda_\arb \subset S^*M$ is as $C^0$-close to $\Lambda\subset S^*M$
as one wishes.
\end{remark}

\begin{corollary}\label{intro: cor}
Let $\Lambda\subset S^*M$ be a Legendrian singularity.

There is a canonical equivalence between the dg category 
$\mu\Sh_{\Lambda}(S^* M)$ of microlocal sheaves  and  the global sections of a sheaf of dg categories over $\Lambda_\arb\subset S^*M$ 
that assigns perfect modules over trees to small open sets,     integral transforms 
for correspondences of trees to inclusions of small open sets, along with natural higher coherences.

\end{corollary}

One can view  the theorem and corollary as providing analogies with basic results of Morse theory. 

On the one hand, any germ of a smooth function $f:M\to \R$  can be deformed to a nearby function %$f_{\Morse}:M\to \R$
with Morse singularities. Moreover, Morse singularities are of a simple combinatorial form 
%$\sum_{i = 0}^k x_i^2 - \sum_{i = k+1}^{\dim M} x_i^2$ 
enumerated by their Morse index 
$0 \leq k \leq \dim M$.
What results is  a  combinatorial description of the cohomology of $M$  in terms of the  the Morse-Witten complex.

On the other hand, any Legendrian singularity  $\Lambda\subset S^*M $ can be deformed to a  Legendrian subvariety with arboreal singularities. Moreover, arboreal singularities are of a simple combinatorial form 
%$\sum_{i = 0}^k x_i^2 - \sum_{i = k+1}^{\dim M} x_i^2$ 
enumerated by leafy rooted trees.
What results is  a  combinatorial description of microlocal sheaves along $\Lambda\subset S^*M $ in terms of a diagram of perfect modules over trees.

We have not attempted to formulate here the sense in which  arboreal singularities are the ``stable" Legendrian singularities, though
we expect the analogy to extend in this direction.

\begin{remark}
We have used the term non-characteristic  to mean that the dg category  of microlocal sheaves  is invariant under the deformation. We expect this to be a consequence of the following more basic geometric notion. 

For any choice of Reeb vector field and resulting Reeb flow $\varphi_t$,
a reasonable Legendrian singularity $\Lambda \subset S^*M$ will admit a small $\epsilon>0$ so that the Reeb flow displaces
   $$
   \xymatrix{
   \Lambda\cap \varphi_t(\Lambda) = \emptyset, & \mbox{ for all } 0<t<\epsilon.
   }
   $$
In other words, there will be no positive  Reeb trajectories from $\Lambda$ to itself of length less than~$\epsilon$.

We expect a deformation $\tilde \Lambda \subset S^*(M\times \BR)$ will be non-characteristic if there is a small $\epsilon>0$ so 
 that the Reeb flow displaces  
 $$
 \xymatrix{
 \tilde \Lambda_s\cap \varphi_t(\tilde \Lambda_s) = \emptyset, & 
\mbox{ for all } 0<t<\epsilon, \mbox{ uniformly in } s.
 }
 $$
In other words,  there will be no positive Reeb trajectories from $\tilde \Lambda_s$ to itself of length less than~$\epsilon$
for all parameters $s\in \BR_{\geq 0}$.
\end{remark}

\begin{example}\label{ex non-char}
Take $M=\R^2$ with coordinates $x, y$. % and $S^*M$  its cosphere bundle. 
 Introduce the hypersurface 
 $$
 \xymatrix{
 H = \{ y(y-x^2)(y+x^2) = 0\} \subset \R^2
} $$ 
as pictured in Fig.~\ref{fig: init curve}.
It is the homeomorphic wavefront projection of a Legendrian subvariety $\Lambda \subset S^*\R^2$ given by  the closure of the restriction of the $dy$ codirection
to the conormal of the smooth locus of $H$. As a topological space, the curve $H$, and hence 
  the Legendrian subvariety $\Lambda$, is the union of three real lines all glued to each other at  zero to form a six-valent node.

  \begin{figure}[h]
  \begin{center}  
\includegraphics[trim = 0mm 70mm 0mm 70mm, clip, scale = .25]{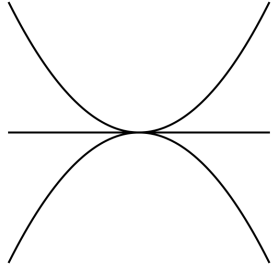}
   \caption{Front projection of initial Legendrian $\Lambda\subset S^*\BR^2$.}
   \label{fig: init curve}
\end{center}
\end{figure}

 We will describe two deformations of $\Lambda \subset S^*M$ to nearby Legendrian subvarieties with simpler singularities, but only the first will be a  non-characteristic deformation.
 
 \medskip
 1) For $s\geq 0$, consider the family of hypersurfaces
 $$
 \xymatrix{
 H_s = \{ y= 0\} \cup  \{x\geq s,  (y-(x-s)^2)(y+(x-s)^2) = 0\} \cup  \{x\leq 0,  (y-x^2)(y+x^2) = 0\}   \subset \R^2
 }
 $$ 
 pictured in Fig.~\ref{fig: nonchar def}.
  It is the homeomorphic wavefront projection of a  non-characteristic family of Legendrian subvarieties $ \Lambda_s \subset S^*\R^2$ given  by  the closure of the restriction of the $dy$ codirection
to the conormal of the smooth locus of $H_s$. When $s>0$,  as a topological space,
 the curve $H_s$, and hence the Legendrian subvariety $ \Lambda_s$, has two singularities which are four-valent nodes.

  \begin{figure}[h]
  \begin{center}  
\includegraphics[trim = 0mm 70mm 0mm 70mm, clip, scale = .25]{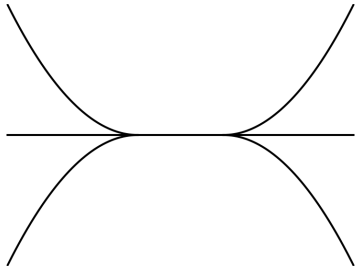}
   \caption{Front projection of non-characteristic deformation.}
   \label{fig: nonchar def}
\end{center}
\end{figure}

 \medskip
 2) For $s\geq 0$, consider the family of hypersurfaces
 $$
 \xymatrix{
 H_s = \{ y= 0\} \cup  \{(y-(x-s)^2)=0\} \cup  \{(y+(x+s)^2) = 0\}   \subset \R^2
 }
 $$ 
 pictured in Fig.~\ref{fig: char def}.
  It is the homeomorphic wavefront projection of a   family of Legendrian subvarieties $ \Lambda_s \subset S^*\R^2$
  given by  the closure of the restriction of the $dy$ codirection
to the conormal of the smooth locus of $H_s$.  When $s>0$,  as a topological space,
 the curve $H_s$, and hence the Legendrian subvariety $ \Lambda_s$, has two singularities which are four-valent nodes.
But the family is not non-characteristic: for any small $\epsilon>0$, there is  a small $s>0$ so that there is a geodesic in $\BR^2$, positive with respect to $dy$,  of length less than $\epsilon$, from a point of  $\{y+(x+s)^2=0\}$ to a point of $\{y-(x-s)^2 = 0\}$ and orthogonal to each.

  \begin{figure}[h]
  \begin{center}  
\includegraphics[trim = 0mm 70mm 0mm 70mm, clip, scale = .25]{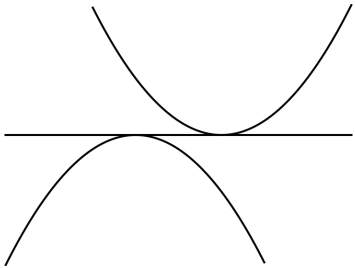}
   \caption{Front projection of characteristic deformation.}
   \label{fig: char def}
\end{center}
\end{figure}

  \end{example}

\subsection{Sketch of arguments}

We briefly sketch here the idea behind the expansion algorithm of Theorem~\ref{thm: intro} in the case of one-dimensional Legendrian singularities. As topological spaces, one-dimensional Legendrian subvarieties are nothing more than graphs, and it is not difficult to understand their deformations. But let us use this special case to give a hint about the proof of the general case which is a somewhat intricate inductive pattern built out of similar arguments. 

%
%Given a Legendrian singularity $\Lambda \subset S^*M$ in generic position,
%we focus on the front projection $H = \pi(\Lambda) \subset M$.
% It is the germ of a Whitney stratrified hypersurface  equipped with distinguished codirections recovering the Legendrian singularity.
%
%
%From here, rather than try to explain the somewhat complicated inductive pattern of the algorithm in general, let us focus the current discussion in the introduction on the form it takes for  one-dimensional Legendrian singularities. 

First, it suffices to take $M=\R^2$ with coordinates $x, y$. Any Legendrian singularity $\Lambda \subset S^* \BR^2$ in generic position has wavefront projection a singular plane curve $C = \pi(\Lambda) \subset \BR^2$
as pictured in Fig.~\ref{fig: init front proj}. We may assume $C$ passes through the origin $0 \in \BR^2$, is smooth away from $0$, so that $\Lambda$ defines a coorientation of $C\setminus \{0\} \subset \BR^2$,
and the fiber at the origin $\Lambda|_0 \subset S^*_0\BR^2$ is the single codirection~$dy$.  With this setup, the wavefront projection from the Legendrian $\Lambda\subset S^*\BR^2$ to the curve $C \subset \BR^2$ is a homeomorphism. (A significant complication in higher dimensions is the fact that it is only possible  to arrange for the wavefront projection to be a finite map.)

  \begin{figure}[h]
  \begin{center}  
\includegraphics[trim = 0mm 0mm 40mm 0mm, clip, scale = .25]{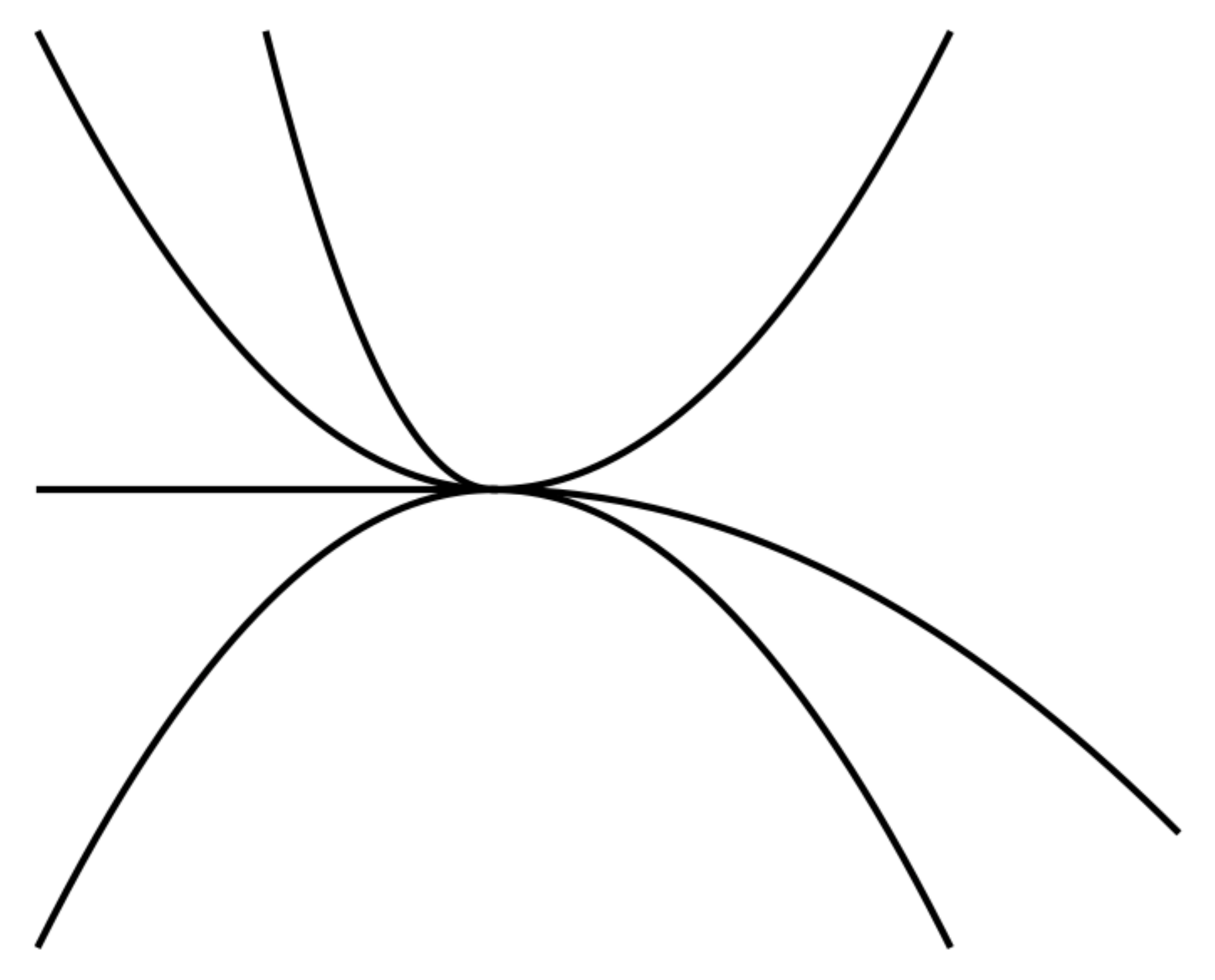}
   \caption{Initial front projection $C\subset \BR^2$.}
   \label{fig: init front proj}
\end{center}
\end{figure}

Next, consider the circle $S(r)\subset \BR^2$ of a small radius $r>0$ around the origin $0\in \BR^2$. Let us assume  $S(r')$ is transverse to $C$, for all radii $0<r'\leq r$. For a very small constant $d>0$, 
introduce the closed subarc of the circle
$$
\xymatrix{
E=S(r) \setminus  \{ (x, y) \in S(r) \, |\,  y<0, |x| <d\} \subset   \BR^2
}
$$
Consider the closed ball $B(r)\subset \BR^2$ of radius $r$ around the origin $0\in \BR^2$, 
and form a new curve given by the union
$$
\xymatrix{
C_{\prearb}= (C \setminus (C\cap B(r)) \cup E \subset \BR^2
}
$$
as pictured in Fig.~\ref{fig: new curve}.

  \begin{figure}[h]
  \begin{center}  
\includegraphics[trim = 0mm 0mm 40mm 0mm, clip, scale = .25]{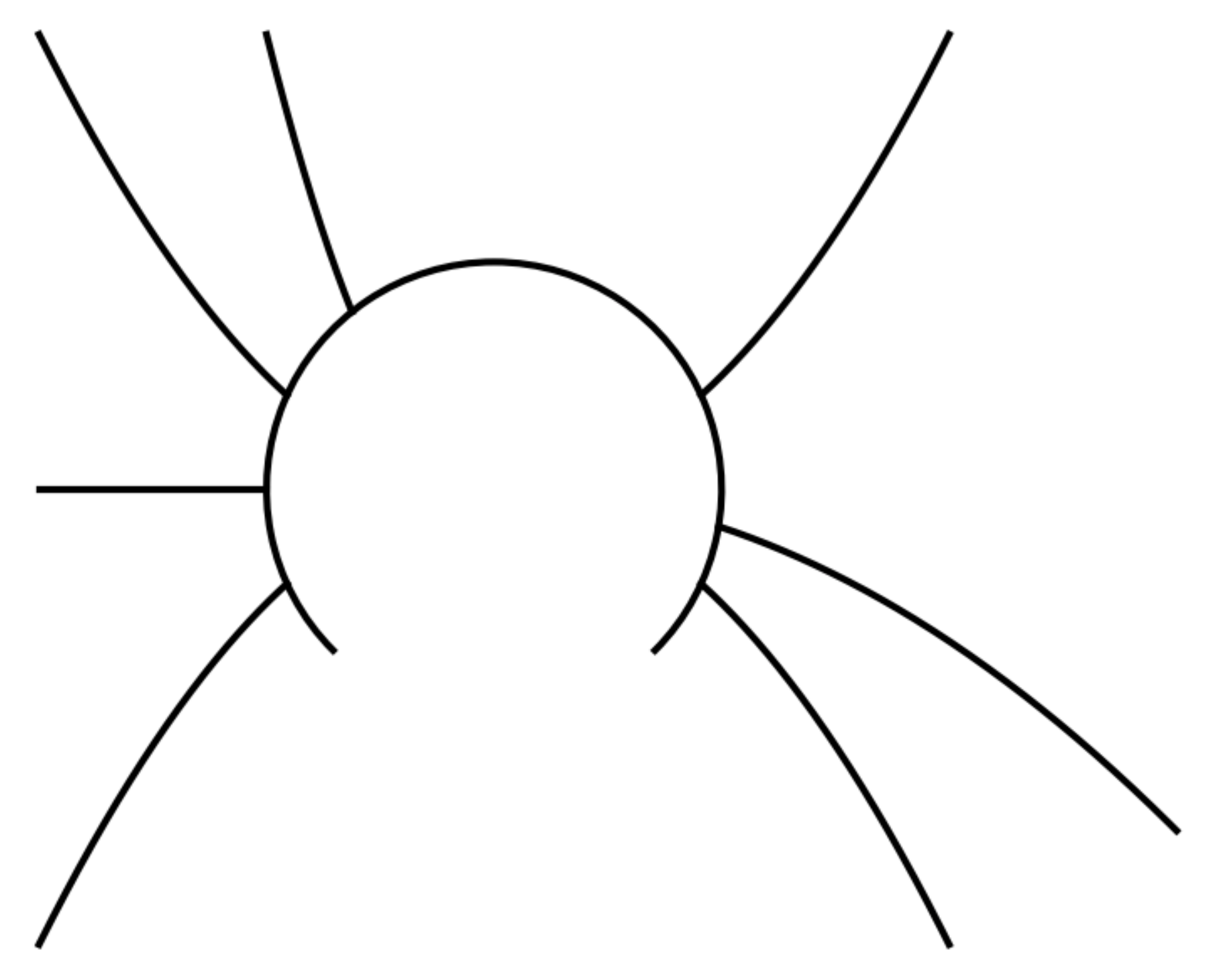}
   \caption{Intermediate curve $C_{\prearb}\subset \BR^2$.}
   \label{fig: new curve}
\end{center}
\end{figure}

Observe that $C_\prearb$ is smooth away from the finitely many points of the  intersection $C\cap S(r)$.
Moreover, away from these points, it has a canonical coorientation given by $\Lambda$ along $C \setminus (C\cap B(r))$, 
and by the outward radial differential $dr$ along $E \setminus (C\cap S(r))$.
Working locally near each of the points of $C\cap S(r)$, we can smooth $C_\prearb \subset \BR^2$ to  a new homeomorphic  curve $C_{\arb}\subset \BR^2$, as
 pictured in Fig.~\ref{fig: smoothed new curve},
 that has an unambiguous  coorientation defined everywhere. Thus there is a corresponding Legendrian $\Lambda_{\arb}\subset S^*\BR$ with homeomorphic wavefront projection to $C_\arb \subset \BR^2$.

  \begin{figure}[h]
  \begin{center}  
\includegraphics[trim = 0mm 0mm 40mm 0mm, clip, scale = .25]{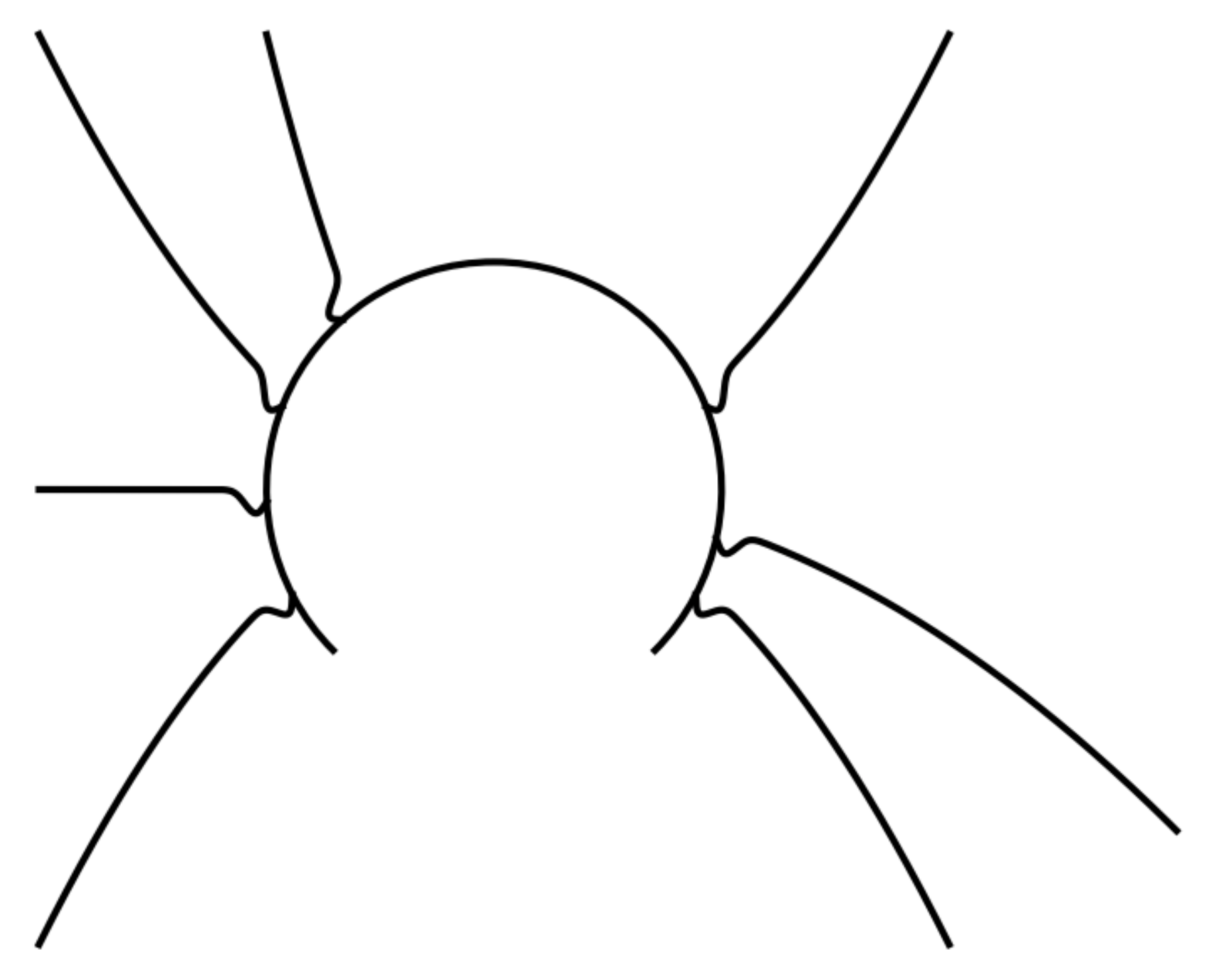}
   \caption{Final curve $C_{\arb}\subset \BR^2$.}
   \label{fig: smoothed new curve}
\end{center}
\end{figure}

Finally, the singularities of $C_\arb$, and hence  those of $\Lambda_{\arb}$, are of two combinatorial types. First, in a neighborhood of the points of the intersection  $C\cap S(r)$, the singularities are trivalent nodes. With the exception of smooth points, these are the simplest example of arboreal singularities. In a neighborhood of the boundary ends of the subarc
$$
\xymatrix{
\partial E=S(r) \setminus  \{ (x, y) \in S(r) \, |\,  y<0, |x| =d\} \subset   C_\arb
}
$$
we find univalent nodes. These are the simplest example of the generalized arboreal singularities discussed in Sect.~\ref{s arb} below.
One might hope to  find only trivalent nodes and not univalent nodes, but if the original  Legendrian $\Lambda\subset S^*\BR^2$, and hence curve $C\subset \BR^2$,  itself had a univalent node, it would be awkward to try to deform away a univalent node  rather than accept it as a reasonable singularity.
%
%is the zero-locus of a differentiable function $f:\BR^2\to \BR$ which is a submersion away from the origin so that $\Lambda$ is the closure of the positive codirection of the differential $df|_{C\setminus 0}$.
%
%
%We use the phrase   expansion algorithm to refer to the general procedure that formalizes the above constructions in all dimensions.

%%%%%%%%%%%%%%%%%%%%%%%%%%%%%%%%%%%%%%%%%%%%%%%%%%%%%%%

\subsection{Motivations}
We summarize here some of the prior and ongoing work motivating the results this paper. 

To start, 
there is a well established expectation that the Fukaya category of an exact symplectic manifold,
specifically a Weinstein manifold, admits a description in terms of microlocal sheaves supported on a Lagrangian skeleton.
(Alternatively, to more closely approach the setting of this paper, one can formulate both theories on the contactification of the exact symplectic manifold along with the Legendrian lift of the skeleton.)
There are two primary variants of this expectation: the infinitesimal Fukaya category~\cite{FOOO, Seidel} of compact branes running along the skeleton corresponds to constructible microlocal sheaves, and the wrapped Fukaya category~\cite{AS} of non-compact branes transverse to the skeleton corresponds to wrapped microlocal sheaves (see~\cite{Nwms} for further discussion). 
It is also possible to formulate a unified version involving non-compact skeleta and the partially wrapped Fukaya category~\cite{aurouxpw}.
The  expectation can be realized in many settings, notably for cotangent bundles~\cite{NZ, Nequiv}.

The Fukaya category and microlocal sheaves each have their advantages and challenges.
On the one hand, the Fukaya category is evidently a symplectic invariant, and in particular independent of the choice of a skeleton. Its objects have clear geometric appeal, though its morphisms involves challenging global analysis.
On the other hand, microlocal sheaves form  a sheaf along the skeleton, and enjoy powerful functoriality
developed by Kashiwara-Schapira~\cite{KS}.   Its objects are often quite abstract, and in particular involve choices of local polarizations. 

Without a comprehensive theory that the Fukaya category and microlocal sheaves coincide,
one is naturally led to the following questions. 

First, given a skeleton, is there a sheaf of dg categories whose global sections is the Fukaya category? This was conjectured by  Kontsevich~\cite{kont} and establishing instances of the resulting local-to-global gluing is an active theme in the subject. 

Second, is there an elementary way to define and calculate microlocal sheaves along a skeleton?
 This has been the focus of a longstanding program  pioneered  by MacPherson (notably for microlocal perverse sheaves in the holomorphic symplectic setting), and  advanced
for example by Gelfand-MacPherson-Vilonen~\cite{GMV}. 

The current paper provides an answer to this second question in terms of the linear algebra of modules over quivers. It can be applied to reduce any calculation about microlocal sheaves to a finite amount of linear algebra. This was implemented in~\cite{Nlg3d} to establish a central instance of mirror symmetry where the skeleton is a three-dimensional singular thimble. Namely, a simple application of the expansion algorithm provides the main step in the proof of the following.

\begin{thm}[\cite{Nlg3d}]
Set $M= \BC^3$ with its conic exact symplectic structure, and $\Lambda \subset M$  the Lagrangian cone over the
standard Legendrian
torus $T^2  \subset S^5$.

There are are equivalences of dg categories 
$$
\xymatrix{
\mu\Sh_{\Lambda}(M) \simeq \Coh_{\mathit{tors}}(\BP^1 \setminus\{0, 1, \infty\})
&
\mu\Sh^w_{\Lambda}(M) \simeq \Coh(\BP^1 \setminus\{0, 1, \infty\})
}
$$
where $\mu\Sh_{\Lambda}(M)$ and $\mu\Sh^w_{\Lambda}(M)$ denote respectively traditional and wrapped microlocal sheaves on $M$ supported along $\Lambda$, and $\Coh(\BP^1 \setminus\{0, 1, \infty\})$ and $\Coh_{\mathit{tors}}(\BP^1 \setminus\{0, 1, \infty\})$ denote the dg category of respectively coherent complexes and coherent complexes with proper support on $\BP^1 \setminus\{0, 1, \infty\}$.
 \end{thm}

  In another direction, the results of this paper were invoked in~\cite{Nwms} to establish that traditional microlocal sheaves and wrapped microlocal sheaves are dual in a sense parallel to the duality of perfect complexes and coherent complexes~\cite{BNP}, or more basically, functions/cohomology and distributions/homology. More precisely, it provided the central tool in the proof of the following, reducing the assertion to a simple observation about modules over trees.

\begin{thm}[\cite{Nwms}]
Let $\Omega\subset T^*X$ be a conic open subset, and $\Lambda\subset T^*X$ a closed conic Lagrangian subvariety.

The natural hom-pairing provides an equivalence
$$
\xymatrix{
\mu\Sh_\Lambda(\Omega)  \ar[r]^-\sim & \Fun^{ex}(\mu\Sh_\Lambda^w(\Omega)^{op}, \Perf_k) 
}
$$
where $\mu\Sh_{\Lambda}(\Omega)$ and $\mu\Sh^w_{\Lambda}(\Omega)$ denote respectively traditional and wrapped microlocal sheaves on $\Omega$ supported along $\Lambda$,
 $ \Fun^{ex}$ denotes the dg category of exact functors, and $\Perf_k$ denote the dg category of perfect $k$-modules.
\end{thm}

At a more fundamental level, the current paper  also provides the basis for answering the first question, about sheafifying the Fukaya category along a skeleton, via showing the Fukaya category and microlocal sheaves coincide. 
Namely, in ongoing work with Eliashberg and Starkston~\cite{ENS}, its techniques are applied to show that
any Weinstein manifold admits an arboreal skeleton. From here, ongoing work of Ganatra-Pardon-Shende~\cite{GPS} shows that 
given an arboreal skeleton,  the Fukaya category is equivalent to microlocal sheaves along it. 
Going further, we expect that arboreal skeleta  will provide a higher dimensional notion of ribbon graphs from which one can recover the ambient symplectic manifold itself. These developments rest upon the  fact 
that the singularities of skeleta can be deformed  to arboreal singularities.

%%%%%%%%%%%%%%%%%%%%%%%%%%%%%%%%%%%%%%%%%%%%%%%%%%%%%%%

\subsection{Summary of sections}
Here is a brief summary of the specific contents of the sections of the paper.
Sect.~\ref{s prelim} summarizes standard material from singularity theory, in particular Whitney stratifications and their control data which provide the language for our geometric constructions. Sect.~\ref{s dir} summarizes the basic structure of wavefront projections in the form of directed hypersurfaces and positive coray bundles. Sect.~\ref{s arb} reviews the notion of arboreal singularities from~\cite{Narb},
then extends their  exposition to generalized arboreal singularities.
Sect.~\ref{s exp} contains  our main geometric constructions: it presents the expansion algorithm that takes a Legendrian singularity to an arboreal Legendrian subvariety. Sect.~\ref{s inv} contains our main technical arguments: it proves that the expansion algorithm is non-characteristic in the sense that the dg category of microlocal sheaves is invariant under it.
Finally, a brief appendix summarizes the data that goes into  the expansion algorithm, in particular the hierarchy of the chosen constants.

%%%%%%%%%%%%%%%%%%%%%%%%%%%%%%%%%%%%%%%%%%%%%%%%%%%%%%%

\subsection{Acknowledgements}
I thank  D. Auroux, D. Ben-Zvi, M. Goresky, J. Lurie, I. Mirkovi\'c, N. Rozenblyum, D. Treumann, H. Williams, L. Williams, and E. Zaslow for their interest, encouragement, and valuable comments.
% I  also thank  for many  discussions on a broad range of related and unrelated topics.
I am additionally grateful to D. Treumann for generously creating the pictures appearing in the figures.
Finally, I am grateful to the NSF for the support of grant DMS-1502178.

%%%%%%%%%%%%%%%%%%%%%%%%%%%%%%%%%%%%%%%%%%%%%%%%%%%%%%%
%%%%%%%%%%%%%%%%%%%%%%%%%%%%%%%%%%%%%%%%%%%%%%%%%%%%%%%
%%%%%%%%%%%%%%%%%%%%%%%%%%%%%%%%%%%%%%%%%%%%%%%%%%%%%%%

\section{Preliminaries}\label{s prelim}

This section collects  standard material on stratification theory following Mather~\cite{Mather}.

We write $\R$ for the real numbers, $\R_{>0}$ for the positive real numbers, and $\R_{\geq 0}$ for the non-negative real numbers. All manifolds will be smooth and equidimensional and all maps will be smooth unless otherwise stated.

%%%%%%%%%%%%%%%%%%%%%%%%%%%%%%%%%%%%%%%%%%%%%%%%%%%%%%%

\subsection{Whitney stratifications}
%
%To ensure reasonable topology, we will work with Whitney stratified spaces following~\cite{}.
%Let us recall some highlights of the general theory that we will use.

Let $M$ be an manifold and $X\subset M$ a closed subspace.  A {\em Whitney stratification} of $X$ is
a disjoint decomposition 
$$
X = \cup_{\alpha \in A} X_\alpha
$$ 
into submanifolds $X_\alpha\subset M$ satisfying:

\medskip

{\em (Axiom of the frontier)}  If $X_\alpha \cap \ol X_\beta \not = \emptyset $, then $X_\alpha \subset \ol X_\beta$.

\medskip
 {\em (Local finiteness)}
Each point $x\in M$ has an open neighborhood $U\subset M$ such that $U \cap X_\alpha  = \emptyset$  for all but finitely many $\alpha\in A$.

\medskip
 {\em (Whitney's condition $B$)}  If sequences $a_k\in X_\alpha$ and $b_k \in X_\beta$ converge to some $a\in X_\alpha$, and the sequence of secant lines $ [a_k b_k] $ (with respect to a local coordinate system)
and tangent planes $T_{b_k} X_\beta$ both converge, then $\lim_k [a_k b_k]  \subset \lim_k T X_\beta$.
%(Here one can construct the secant lines and their limit with respect to any local  coordinate system.)

\medskip

The index set $A$ is naturally a poset with $\alpha<\beta$ when $X_\alpha \not = X_\beta$ and $X_\alpha \subset  \ol X_\beta$.

We will say that $X$ has {\em dimension} $k$ if it is the closure of its strata of dimension $k$.

\begin{remark}
Note that we can trivially extend any Whitney stratification of $X\subset M$ to a Whitney stratification of all of $M$
by including the open complement $M\setminus X$ itself as  a stratum.
\end{remark}

Given a Whitney stratification of $X\subset M$, by a {\em small open ball} $B\subset M$ around a point $x\in M$, we will always mean an open  ball of   radius $R>0$ with respect to the Euclidean metric of some local coordinate system. Moreover, we will assume  that for any  $0<r\leq R$,
 the sphere around $x$ of  radius   $r$  is transverse to the strata.  Whitney's condition $B$ guarantees
 this holds for any local coordinate system and small enough $R>0$.

%
%\begin{terminology}
%By  a {\em subvariety} $X\subset M$, we will mean a closed subset $X\subset M$ that admits a Whitney stratification.
%We say that a subvariety $ X\subset M$ has {\em dimension} $k$ if for some (and hence any) Whitney stratification of $X$, its $k$-dimensional strata are nonempty and dense.
%\end{terminology}

%We will work in situations where there are in fact only finitely many strata with finitely many connected components.
%
%
%a filtration of $X$
%by closed subsets
%$$
%\xymatrix{
%X_0 \subset X_1 \subset \cdots X_n = X
%}
%$$
%such that each complement $S_i = X_i \setminus X_{i-1}$ is an $i$-submanifold,
% and each pair $S_i, S_j$ satisfies Whitney's condition $B$.
% 
%It follows that each pair  $S_i, S_j$ satisfies the axiom of the frontier in the sense that $S_i \cap \ol S_j$ implies $S_i \subset \ol S_j$.
%It also follows that each $S_i$ is locally finite in the sense that each point of $M$ has an open neighborhood whose intersection with  $S_i$ has finitely many connected components.

%%%%%%%%%%%%%%%%%%%%%%%%%%%%%%%%%%%%%%%%%%%%%%%%%%%%%%%

\subsection{Control data}

Let $M$ be an manifold.

A {\em tubular neighborhood} of a submanifold $Y \subset M$ consists of an  inner product on the normal bundle 
$p:E\to Y$,
and a smooth embedding 
$$
\xymatrix{
\varphi: E[{<\epsilon}] = \{ v\in E \, |\, \langle v, v\rangle < \epsilon \} \ar@{^(->}[r] & M
}$$
of the open ball bundle determined by some $\epsilon>0$. The image $T =\varphi(E[{<\epsilon}])$ is required to be an open neighborhood of $Y \subset M$, and the restriction of $\varphi$ to the zero section $ Y \subset E$ is required to be the identity map  to $Y\subset M$. By rescaling the inner product, we can assume
that $\epsilon = 1$. % or in fact any positive number we like.

By transport of structure, the neighborhood $T$ comes equipped with the tubular distance function
$\rho:T\to \R_{\geq 0}$
and tubular projection $\pi:T \to Y$ defined by
$$
\xymatrix{
 \rho(x) = \langle \varphi^{-1}(x),  \varphi^{-1}(x) \rangle
& 
 \pi(x) = p( \varphi^{-1}(x))
 }
 $$
We will  write  $(T, \rho, \pi)$ to   denote  the tubular neighborhood and  remember that $\pi: T\to Y$ is the open unit ball   bundle in 
a vector bundle with inner product  inducing $\rho:T\to \R_{\geq 0}$.

Given small $\epsilon\geq 0$, we have the inclusions
$$
\xymatrix{
j[\epsilon]:S[\epsilon] = \{x\in T \, |\, \rho(x) = \epsilon\}\ar@{^(->}[r] & T
&
j[<\epsilon]:T[<\epsilon] =  \{x\in T \, |\, \rho(x) < \epsilon\}\ar@{^(->}[r] & T
}
$$
 and similarly with $<$ replaced by 
 $\leq, >,$ or $\geq$.
Of course when $\epsilon = 0$, we have $T[<\epsilon] = \emptyset$, $T[\leq \epsilon] =  S[\epsilon] = Y$, 
$T[> \epsilon] = T\setminus Y$,
$T[\geq \epsilon] = T$.

\medskip

Any Whitney stratified subspace $X\subset M$ admits a compatible {\em system of control data} consisting of a  tubular neighborhood $(T_\alpha, \rho_\alpha, \pi_\alpha)$ of each stratum $X_\alpha \subset X$. Whenever $\alpha <\beta$, the tubular distance functions and tubular projections are required to satisfy
$$
\xymatrix{
\pi_\alpha(\pi_\beta(x)) = \pi_\beta(x)
&
\rho_\alpha(\pi_\beta(x)) = \rho_\alpha(x)
}
$$
on the common domain of points  $x\in T_\alpha\cap T_\beta$ such that $\pi_\beta(x) \in T_\alpha$.

A key property of a system of control data is the fact that 
the product  map 
$$
\xymatrix{
\rho_\alpha \times \pi_\alpha  : T_\alpha  \ar[r] &  \R_{>0} \times X_\alpha 
}
$$ 
has surjective differential when restricted to any stratum $X_\beta\subset X$ with $\beta> \alpha$.

%%%%%%%%%%%%%%%%%%%%%%%%%%%%%%%%%%%%%%%%%%%%%%%%%%%%%%%

\subsection{Almost retraction}\label{s almost ret}

Let $M$ be a manifold. 

Let $X \subset M$ be a closed subspace with Whitney stratification
 $\{X_\alpha\}_{\alpha \in A}$. 
% $X_\alpha \subset X$, indexed by $\alpha\in A$.

Suppose given a compatible  system of control data $\{(T_\alpha, \rho_\alpha, \pi_\alpha)\}_{\alpha\in A}$.

%Recall for each 
%$\alpha\in A$,  this includes a tubular neighborhood $T_{\alpha} \subset M$ 
% of each stratum $X_\alpha \subset X$ along with a tubular distance function  $\rho_\alpha:T_\alpha\to \R$ and
% tubular  projection $\pi_\alpha:T_\alpha\to X_\alpha$.
%
%Recall that for $\epsilon>0$, we write $T_\alpha[<\epsilon] \subset T_\alpha$ for the open subspace where $\rho_\alpha < \epsilon$, and $S_\alpha[\epsilon] \subset T_\alpha$ for the closed subspace where $\rho_\alpha = \epsilon$.
%
Following Goresky~\cite{Goresky1, Goresky2}, we review some further fundamental constructions.

Fix once and for all a small $\epsilon>0$.

Choose a {\em family of lines} subordinate to the system of control data. This consists of a retraction 
$$
\xymatrix{
r_\alpha: T_\alpha[<2\epsilon] \setminus X_\alpha \ar[r] &  S_\alpha[2\epsilon]
}
$$ for each  $\alpha\in A$  satisfying the following. For $\alpha<\beta$, one requires  $r_{\alpha}|_{X_\beta}$ is smooth
and the compatibilities
$$
\xymatrix{
r_{\alpha} r_\beta = r_\beta r_\alpha &
\rho_{\alpha} r_\beta = \rho_\alpha
&
\rho_\beta r_\alpha = \rho_\beta
&
\pi_\alpha r_\alpha = \pi_\alpha
&
\pi_\alpha r_\beta = \pi_\alpha
}$$
on their natural domains.
The retractions provide homeomorphisms 
$$
\xymatrix{
h_\alpha:T_\alpha[<2\epsilon] \setminus X_\alpha \ar[r]^-\sim &  S_\alpha[2\epsilon] \times (0, 2\epsilon)
&
h_\alpha = r_\alpha \times \rho_\alpha
}$$
and for $B\subset A$, more general collaring homeomorphisms
$$
\xymatrix{
h_B: \bigcap_{\alpha\in B} (T_\alpha[<2\epsilon] \setminus X_\alpha) \ar[r]^-\sim &  
(\bigcap_{\alpha\in B} S_\alpha[2\epsilon] )\times \prod_{\alpha \in B} (0, 2\epsilon)
}
$$
$$
\xymatrix{
h_B = (r_{\alpha_1}r_{\alpha_2} \cdots r_{\alpha_k}) \times \prod_{\alpha\in B} \rho_\alpha
}$$
where $k=|B|$ and the indices $\alpha_i \in B$ can be arbitrarily ordered thanks to 
 $r_{\alpha_i} r_{\alpha_j} =r_{\alpha_j} r_{\alpha_i}$.

Fix a smooth nondecreasing function $q:\R\to \R$ such that $q(t) = 0$, for $t\leq \epsilon$, and $q(t) = t$, for $t\geq 2\epsilon$. For each stratum $X_\alpha \subset X$, introduce the mapping
$$
\Pi_\alpha:M\longrightarrow M
\hspace{3em}
\Pi_\alpha(x) = 
\left\{
\begin{array}{cl}
x & \mbox{ when } x\not\in T_\alpha[<2\epsilon]\\
h^{-1}_\alpha(r_\alpha(x), q(\rho_\alpha(x))) &  \mbox{ when }  x\in T_\alpha[<2\epsilon]
\end{array}
\right.
$$
It is continuous, homotopic to the identity,  and satisfies $\Pi_\alpha(x) =\pi_\alpha(x)$ when $x\in T_\alpha[\leq \epsilon]$.
Moreover, the mappings commute $\Pi_{\alpha} \Pi_{\beta} = \Pi_{\beta} \Pi_{\alpha}$, for  $\alpha, \beta \in A$. To confirm this, if $x\not \in T_\alpha[<2\epsilon]$, then $\Pi_\alpha(x) = x$, and $\Pi_\beta(x) \not \in T_\alpha[<2\epsilon]$
since $\rho_\alpha r_\beta(x) = \rho_\alpha(x)$, so $\Pi_{\alpha} \Pi_{\beta} (x) = \Pi_\beta(x) = \Pi_{\beta} \Pi_{\alpha}(x)$. 
If $x\in T_\alpha[<2\epsilon] \cap T_\beta[<2\epsilon]$, then 
$$
\xymatrix{
\Pi_\alpha \Pi_\beta (x) = h_{\{\alpha, \beta\}}^{-1}(r_\alpha r_\beta(x), q(\rho_\alpha(x)), q(\rho_\beta(x))) 
=\Pi_\beta \Pi_\alpha (x)
}
$$

Now introduce the composition
$$
\xymatrix{
r:M \ar[r] & M &
r = \Pi_{\alpha_0} \Pi_{\alpha_1} \cdots  \Pi_{\alpha_N}
}
$$
where $N+1=|A|$ and the indices $\alpha_i \in A$ can be arbitrarily ordered thanks to $\Pi_{\alpha_i} \Pi_{\alpha_j} =\Pi_{\alpha_j} \Pi_{\alpha_i}$.
 It is continuous, homotopic to the identity,
 and satisfies 
 $$
 \xymatrix{
r(x) =\pi_\alpha(x) &  \mbox{ when }  x\in T_\alpha[\leq \epsilon] \setminus  \bigcup_{\beta<\alpha}  (T_\alpha[\leq\epsilon] \cap T_\beta[< 2\epsilon])
}
$$

The restriction of $r$ to  the open subspace 
$$
U[< \epsilon] = \bigcup_{\alpha  \in A} T_{\alpha}[<\epsilon] \subset M
$$
 is almost a retraction of $U[< \epsilon]$ to $X$  in that it maps $U[< \epsilon]$ to $X$ and the restriction of $r$ to $X$ is homotopic to the identity.
In fact,  the restriction of $r$ to $X$ is the identity on any $x\in X$ whenever
$x \in X_\alpha \setminus \cup_{\beta<\alpha} (X_\alpha \cap T_\beta[<2\epsilon])$ for some $\alpha \in A$.

\begin{remark}\label{rem induction}
The above constructions are well suited to inductive arguments. Fix a closed stratum $X_{0}\subset X$,
and set $M' = M\setminus X_{0}$, $X'= X\setminus X_{0}$.
The system of control data and family of lines for $X\subset M$ immediately provide the same for
$X'\subset M'$ by deleting the data for $X_{0}\subset X$. 
The resulting almost retraction $r':M'\to M'$
satisfies the following evident compatibility with the almost retraction $r: M\to M$. Introduce the composition
$$
\xymatrix{
\hat r_{0}:M \ar[r] & M &
\hat r_{0} = \Pi_{\alpha_1} \cdots  \Pi_{\alpha_N}
}
$$
so that $r = \Pi_{0}  \hat r_{ 0}$. Then $\hat r_{0}|_{X_{0}} = \id_{X_{0}}$ and $ \hat r_{0} |_{M'} = r'$.
\end{remark}

\begin{remark}
By convention, when $X=\emptyset$,  we set $r = \id_M:M\to M$ to be the identity.  This would  naturally result from invoking the above constructions with the complement $M\setminus X \subset M$ itself as a stratum.
This is a trivial modification since the tubular neighborhood of such an open stratum is simply itself.
\end{remark}

%%%%%%%%%%%%%%%%%%%%%%%%%%%%%%%%%%%%%%%%%%%%%%%%%%%%%%%

\subsection{Multi-transversality}

Let $M$ be an manifold.

We say that a finite set $\fF = \{f_i\}_{i\in I}$ of functions $f_i: M \to \R$ is {\em multi-transverse} at a value $s_I = (s_i)\in  \R^I$ if for any subset $J\subset I$, the product map
$$
\xymatrix{
F_J  = \prod_{j\in J} f_j : M   \ar[r] &  \R^J
}
$$ 
is a submersion along $F_J^{-1}(s_J) \subset M$ where $s_J =(s_j)\in \R^J$ is the image of  $s_I= (s_i)\in  \R^I$ under the natural projection  $\pi_J:\R^I\to \R^J$. 

If $\fF = \{f_i\}_{i\in I}$ is multi-transverse
at $s_I= (s_i)\in \R^I$, then the level-sets 
$$
\xymatrix{
H_i(s_i) = f_i^{-1}(s_i)\subset M
}
$$ 
are smooth hypersurfaces (if non-empty) and multi-transverse in the following sense.
For any subset $J\subset I$, the intersection 
$$
\xymatrix{
H_J(s_J) = \bigcap_{j\in J} H_j(s_j) \subset M
}
$$ 
is a smooth submanifold
of codimension $|J|$ (if non-empty), and 
 transverse to $H_i(s_i) \subset M$, for each $i\in I\setminus J$.

\medskip

Suppose  a finite set 
 $\fF= \{f_i\}_{i\in I}$ of functions  $f_i: M \to \R$ is multi-transverse at a value  $s_I = (s_i)\in \R^I$. Then given 
 another function $f:M\to \R$ and a value $s\in \R$, we may find a nearby value $s'\in \R$ so that the extended set
  $ \{f\}\coprod  \{f_i\}_{i\in I}$ of functions  is multi-transverse at the extended value $ (s', s_i)\in \R\times \R^I$. 
  This follows from Sard's Theorem: there is a nearby regular value $s'\in \R$ for the function
  $$
  \xymatrix{
  \coprod_{J \subset I} f|_{H_J(s_J)}: \coprod_{J \subset I} H_J(s_J)\ar[r] & \R
  }
  $$
Thus  given any finite set  $\fF = \{f_i\}_{i\in I}$ of functions   $f_i: M \to \R$ and a value  $s= (s_i)\in \R^I$,
by induction on any order of $I$,
 there is a nearby value $s' = (s'_i)\in \R^I$ such that
$\fF = \{f_i\}_{i\in I}$ is multi-transverse at $s' = (s'_i)\in \R^I$.

\begin{example}
(1) Let $M= \R$ with coordinate $x$. Set $f_1 = f_2 = x$.
Then $\fF = \{f_1, f_2\}$ is multi-transverse at $(s_1, s_2)\in \R^2$ if and only if $s_1 \not =  s_2$.

(2)
Let $M= \R^2$ with coordinates $x_1, x_2$. Set $f_1 = x_1$, $f_2 = x_2$, and $f_3 = x_1 + x_2$. Then $\fF = \{f_1, f_2, f_3\}$ is  multi-transverse at  $(s_1, s_2, s_3)\in \R^3$ if and only if $s_3 \not = s_1 + s_2$. 
\end{example}

More generally, suppose given a set
$\fF = \{f_i\}_{i\in I}$ of functions $f_i: U_i \to \R$ defined on a locally finite set $\fU = \{U_i\}_{i\in U}$ of open subsets $U_i \subset M$.
We will say that  such a set 
$\fF = \{f_i\}_{i\in I}$ is  {\em multi-transverse} at a value $s_I = (s_i)\in  \R^I$ if for any finite subset $J\subset I$ the product map
$$
\xymatrix{
F_J  = \prod_{j\in J} f_i : \bigcap_{j\in J} U_j   \ar[r] &  \R^J
}
$$ 
is a submersion along $F_J^{-1}(s_J) \subset M$.
Note that if $I$ is finite, and $U_i = M$, for all $i\in I$, then we recover the previous notion.

For a key example of such a multi-transverse set of functions,
consider a closed subspace  $X\subset M$  with Whitney stratification $\{X_\alpha\}_{\alpha \in A}$. 
For any compatible system of control data $\{ (T_\alpha, \rho_\alpha, \pi_\alpha)\}_{\alpha\in A}$,
the tubular distance functions 
$\fF = \{ \rho_\alpha\}_{\alpha\in A}$  defined on the tubular neighborhoods $\fU= \{T_\alpha\}_{\alpha\in A}$ 
are  multi-transverse at any completely nonzero value $s_A = (s_\alpha)\in \R_{>0}^A$.
Moreover, for any stratum $X_\beta\subset X$, the restrictions 
$\fF_\beta = \{ \rho_\alpha|_{X_\beta}\}_{\alpha<\beta}$  defined on the intersections $\fU_\beta= \{T_\alpha\cap X_{\beta}\}_{\alpha<\beta}$ 
are  multi-transverse at any completely nonzero value. % $s_A = (s_\alpha)\in \R_{>0}^A$.

%%%%%%%%%%%%%%%%%%%%%%%%%%%%%%%%%%%%%%%%%%%%%%%%%%%%%%%
%%%%%%%%%%%%%%%%%%%%%%%%%%%%%%%%%%%%%%%%%%%%%%%%%%%%%%%
%%%%%%%%%%%%%%%%%%%%%%%%%%%%%%%%%%%%%%%%%%%%%%%%%%%%%%%

\section{Directed hypersurfaces}\label{s dir}

\subsection{Notation}

Let $M$ be a  manifold.

Let $T^*M$ denote its  cotangent bundle, and $\theta \in \Omega^1(T^*M)$  the canonical one-form.
We will identify $M$
with the zero-section of $T^*M$.

 Introduce  the spherical projectivization 
$$
\xymatrix{
S^*M = (T^*M\setminus M)/\R_{>0}
}$$ 

If we choose a Riemannian metric on $M$, we can canonically identify  $S^*M$ with the unit cosphere bundle 
$$
\xymatrix{
U^*M =\{v\in T^*M \, |\, \|v\| = 1\}
}$$
The canonical one-form $\theta \in \Omega^1(T^*M)$  restricts to equip $U^*M$ and hence $S^*M$ with a contact form $\alpha \in \Omega^1 (S^*M)$ (depending on the  metric) and a canonical contact structure $\xi = \ker (\alpha) \subset TS^*M$ (independent of the choice of metric).

Introduce  the projectivization 
$$
\xymatrix{
P^*M = (T^*M \setminus M)/\R^\times
}$$ 
%with position and projectivized momentum  coordinates
%$$(x_0, x_1, \ldots, x_n, [ y_0,  y_1, \ldots,  y_n])
%$$ 
%where the  projectivized momentum coordinates $[ y_0,  y_1, \ldots,  y_n]$  are not all zero and are collectively only defined up to multiplication by a nonzero scalar.
 
 We have the natural two-fold cover $S^*M \to P^*M$ which in particular equips $P^*M$ with a
 compatible canonical contact structure (though not a contact form).

Given a submanifold $Y \subset M$, 
we have  its conormal bundle, its spherical projectivization, and its projectivization respectively  
$$
\xymatrix{
T^*_{Y}M \subset T^*M
&
S^*_{Y }M \subset S^*M
&
P^*_{Y }M \subset P^*M
}
$$
The first is a conical Lagrangian submanifold 
%(it is the dimension of $M$ and the canonical one-form restricts to zero along it) 
and the latter two are Legendrian submanifolds.
%(they are one less than the dimension of $M$ and tangent to the contact structure).

\subsection{Good position}

\begin{defn}
By a {\em hypersurface} $H\subset M$, we will mean a  subspace admitting a
Whitney stratification with $\dim H = \dim M -1$. 
\end{defn}

Given a hypersurface $H \subset M$, and any open, dense smooth locus $H^{sm} \subset H$, we have a natural diagram of maps
$$
\xymatrix{
S^*_{H^{sm}} M \ar[r] & P^*_{H^{sm}} M \ar[r] & H^{sm}
}
$$
where the first  is a two-fold cover and the second is a diffeomorphism.

\begin{defn}
A hypersurface $H\subset M$ is said to be in {\em good position} if 
for some (or equivalently any)  open, dense smooth locus $H^{sm}\subset H$,
the closure 
$$
\cL= \ol{P^*_{H^{\mathit{sm}}} M} \subset P^* M,
$$ 
is finite over $H$.
If this holds, we refer to $\cL$ as 
the {\em coline bundle} of $H$.

\end{defn}

\begin{remark}
Equivalently,  $H\subset M$ is in  good position if   the closure 
$$
\cR= \ol{S^*_{H^{\mathit{sm}}} M} \subset S^*M
$$ 
 is finite over $H$.
If so, we refer to $\cR$ as 
the {\em coray bundle} of $H$. 
% The number of points in a fiber of the coray bundle $\cR_H\to H$ is twice that of the coline bundle.

\end{remark}

\begin{remark}
If $H\subset M$ is in good position, we have a natural diagram of finite maps  
$$
\xymatrix{
\cR \ar[r] &  \cL \ar[r] & H
}
$$
where the first is a two-fold cover and the second is a diffeomorphism over $H^{sm}\subset H$.

\end{remark}

\begin{example}
(1) All  real algebraic or subanalytic plane curves are in good position.

(2) Nondegenerate quadratic singularities (singular Morse level-sets) of dimension strictly greater than one are not in good position.
\end{example}

\begin{remark}
If $H\subset M$ is in good position, then Whitney's condition $B$ (in fact Whitney's condition $A$) implies  its coline bundle $\cL$ and coray bundle $\cR$ are conormal to  each stratum $H_\alpha\subset H$ in the sense  that
$$
\xymatrix{
\cL|_{H_\alpha} \subset P^*_{H_\alpha} M
&
\cR|_{H_\alpha} \subset S^*_{H_\alpha} M
}
$$
\end{remark}

\subsection{Coorientation}

\begin{defn}

By a {\em coorientation}  of   a hypersurface $H\subset M$  in good position,
we will  mean   a section 
\vspace{-0.5em}
$$
\xymatrix{
\cR  \ar[r] &  \cL \ar@/_0.5pc/[l]_-\sigma
}$$
 of the natural two-fold cover from the coray to coline bundle.
\end{defn}

\begin{defn}

(1)
By a  {\em directed hypersurface} inside of $M$, we will mean a hypersurface $H\subset M$ in good position equipped with a coorientation $\sigma$.

(2) 
By the {\em positive coray bundle} of a directed hypersurface,
we will mean the image of the coline bundle under the coorientation  
$$
\xymatrix{
\Lambda = \sigma(\cL) \subset S^*M
}$$
\end{defn}

\section{Arboreal singularities}\label{s arb}

We recall and expand upon the local notion of arboreal singularity from~\cite{Narb}.

%%%%%%%%%%%%%%%%%%%%%%%%%%%%%%%%%%%%%%%%%%%%%%%%%%%%%%%

\subsection{Terminology}

We gather here for easy reference some  language used below.

By a {\em graph} $G$, we will mean a set of {\em vertices} $V(G)$ and a set of {\em edges} $E(G)$ satisfying the simplest convention  that $E(G)$ is a
subset of the set of two-element subsets of  $V(G)$. Thus  $E(G)$  records whether pairs of distinct elements  of $V(G)$ are connected by an edge or not. 
We will write $\{\alpha, \beta\} \in E(G)$ and say that $\alpha, \beta \in V(T)$ are {\em adjacent}  if an edge connects them.

By a  {\em tree} $T$, we will mean a nonempty, finite, connected, acyclic graph. Thus for any pair of  vertices $\alpha, \beta\in V(T)$,
there is a unique   minimal path (nonrepeating sequence of  edges) connecting them. We call the number of edges in the sequence the {\em distance} between the
vertices.
%(Since we will only consider trees, it  suffices to  adopt  the simplest convention and assume the edge set  is a subset of the set of two-element subsets of the vertex set.)

Given a graph $G$, by a {\em subgraph} $S\subset G$, we will mean a full subgraph (or vertex-induced subgraph) in the sense that its vertices are a subset $V(S) \subset V(G)$ and
its edges are the subset $E(S) \subset E(G)$ such that $\{\alpha,\beta\} \in E(S)$ if and only if $\{\alpha,\beta\}\in E(G)$ and $\alpha, \beta\in V(S)$.
By the {\em complementary subgraph}
$G\setminus S \subset G$, we will mean the full subgraph on the complementary vertices $V(T\setminus S) = V(T) \setminus V(S)$.

Given a tree $T$, any subgraph $S\subset T$ is a disjoint union of trees. By a {\em subtree} $S \subset T$, we will mean a subgraph that is a tree.
The complementary subgraph
$T\setminus S \subset T$ is not necessarily a tree but in general a disjoint union of subtrees.
%Given a subtree $S\subset T$, and a vertex $\alpha\in V(T \setminus S)$, there is a unique  vertex $\gamma\in V(S)$
%{\em nearest} to $\alpha$.

Given a tree $T$, by a {\em quotient tree} $T\twoheadrightarrow Q$, we will mean a tree $Q$ with a surjection $V(T)\twoheadrightarrow V(Q)$ such that each fiber comprises the vertices of a subtree of $T$. 
We will refer to such subtrees as the {\em fibers} of the quotient $T\twoheadrightarrow Q$.
%Given a vertex $\alpha\in V(T)$,
%we will sometimes write $\ol\alpha\in V(Q)$ for its image, and $T_{\ol\alpha}\subset T$ for the fiber containing $\alpha$.

By a {\em partition} of a tree $T$, we will mean a  collection of subtrees $T_i\in T$, for $i\in I$, that are disjoint $V(T_i) \cap V(T_j) = \emptyset$, for $i\not = j$, and cover $V(T) = \coprod_{i\in I} V(T_i)$.
Note that  the data of a quotient $T\twoheadrightarrow Q$ is equivalent to the partition of $T$ into  the fibers.

By a {\em rooted tree} $\dirT=(T, \rho)$, we will mean  a tree  $T$ equipped with a distinguished vertex $\rho\in V(T)$  called the {\em root vertex}.
The vertices $V(\dirT)$ of a rooted tree naturally form a poset with the root vertex $\rho\in V(\dirT)$ the  unique minimum
and $\alpha< \beta \in V(\dirT)$ if the former is nearer to $\rho$ than the latter. 
To each non-root vertex $\alpha \not = \rho \in V(\dirT)$ there is a unique {\em parent vertex} $\hat \alpha\in V(\dirT)$ such that $\hat\alpha<  \alpha$ and there are no  vertices strictly between them. The data of the root vertex
$\rho$ 
and parent vertex relation $\alpha\mapsto \hat \alpha$  recover the poset structure and in turn the rooted tree.

By a  {\em forest} $F$, we
will mean a
nonempty, finite, possibly disconnected graph with acyclic connected components. 
Thus 
$F = \coprod_i T_i$ is a nonempty disjoint union of finitely many trees.

By a  {\em rooted forest} $\dirF$, we will mean  a forest  $F$ equipped with a distinguished root vertex in each of its connected components.
Thus $\dirF =  \coprod_i \cT_i = \coprod_i (T_i, \rho_i) $ is a nonempty disjoint union of finitely many rooted trees.
The vertices $V(\dirF)$ of a rooted forest naturally form a poset with minima the root vertices and vertices in distinct connected components incomparable.

%%%%%%%%%%%%%%%%%%%%%%%%%%%%%%%%%%%%%%%%%%%%%%%%%%%%%%%

\subsection{Arboreal singularities}

%We denote its vertices by $V(T)$ and its edges by $E(T)$.
To each tree $T$, there is associated a stratified space $\sL_T$  called an {\em arboreal singularity} (see \cite{Narb} and in particular the characterization recalled in Thm.~\ref{thm: arboreal poset} below). It is of pure dimension $|T| - 1$ where we write $|T|$ for  the number of vertices of $T$.
It comes equipped with a compatible   metric and    contracting $\R_{>0}$-action with a single fixed point.  We refer to the compact subspace $\sL_T^\link \subset \sL_T$ of points unit distance from the fixed point
as the {\em arboreal link}. The   $\R_{>0}$-action provides   a canonical identification
$$
\xymatrix{
\sL_T \simeq \Cone(\sL_T^\link)
}
$$
so that one can regard the arboreal singularity $\sL_T$ and arboreal link $\sL_T^\link$ as  respective local models for a normal slice and normal link to a stratum in a stratified space. It follows easily from the constructions that the arboreal link $\sL_T^\link$ is homotopy equivalent to a bouquet of $|T|$ spheres each of dimension $|T|-1$.

As a stratified space, the arboreal link $\sL_T^\link$, and hence the arboreal singularity $\sL_T$ as well, admits a simple  combinatorial description. 
 To each tree $T$, there is a natural finite poset $\fP_T$ whose elements are correspondences of trees
$$
\xymatrix{
\fp=(R &  \ar@{->>}[l]_-q  S \ar@{^(->}[r]^-i & T)
}
$$
where $i$ is the inclusion of  a subtree and $q$ is a quotient  of trees. Thus the tree $S$ is the full subgraph (or vertex-induced subgraph) on a subset of vertices of $T$; the tree $R$ results from contracting a subset of edges of $S$. 
Two such correspondences   
$$
\xymatrix{
\fp=(R &  \ar@{->>}[l]_-q  S \ar@{^(->}[r]^-i & T)
&
\fp'=(R' &  \ar@{->>}[l]_-{q'}  S' \ar@{^(->}[r]^-{i'} & T')
}
$$
satisfy $\fp\geq  \fp'$ 
if there is another  correspondence of the same form
$$
\xymatrix{
\fq=(R &  \ar@{->>}[l]  Q \ar@{^(->}[r] & R')
}
$$
 such that $\fp = \fq\circ\fp'$. In particular, the poset $\fP_T$ contains a unique minimum representing the identity correspondence
 $$
\xymatrix{
\fp_0=(T &  \ar@{->>}[l]_-=  T \ar@{^(->}[r]^-= & T)
}
$$

Recall that a {\em finite regular cell complex} is a Hausdorff space $X$
with a finite collection of closed cells $c_i \subset X$ whose interiors $c_i^\circ \subset c_i$ provide a partition of $X$ and boundaries $\partial c_i \subset X$
are unions of cells. A finite regular cell complex $X$ has the {\em intersection property} if the intersection of any two cells $c_i, c_j\subset X$ is 
either another cell or empty. The {\em face poset} of a finite regular cell complex $X$ is the poset with elements the cells of $X$ with
relation $c_i\leq c_j$ whenever $c_i \subset  c_j$. The {\em order complex} of a poset is the natural simplicial complex with simplices the finite totally-ordered chains of the poset. 
%Finally, recall that a finite regular cell complex can be recovered from its face poset by taking its order complex. 
%(Useful references include~ \cite{bb, bjorner, wachs}.)
%
%One could take the following as a combinatorial definition of arboreal links and 
%hence arboreal singularities.

\begin{thm}[\cite{Narb}]\label{thm: arboreal poset}
Let $T$ be a tree.

The arboreal link $\sL_T^\link$ is a finite regular cell complex, with the intersection property, with face poset $
 \fP_T\setminus \{\fp_0\}$, and  thus  homeomorphic to the order complex of $\fP_T\setminus \{\fp_0\}$.
\end{thm}

%\begin{remark}
%Recall that a finite poset is {\em shellable} if its order complex is pure-dimensional and 
%its top-dimensional simplices can be ordered so that each one (other than the first) intersects the union of its predecessors in a nonempty union of maximal proper faces. 
%
%
%We show that  $\fP_T\setminus \{\fp_0\}$ is shellable giving an alternative proof that the arboreal link $\sL_T^\link$ is a bouquet of spheres.
%\end{remark}

\begin{remark}
It follows that the normal slice to the stratum  $\sL_T(\fp) \subset \sL_T$ indexed by a partition
$$
\xymatrix{
\fp=(R &  \ar@{->>}[l]_-q  S \ar@{^(->}[r]^-i & T) %\in \fP_T \setminus\{\fp_0\}
}
$$
is homeomorphic to the arboreal singularity $\sL_R$.
\end{remark}

%\begin{remark}

%\begin{figure}[h!]
%\includegraphics[scale=0.5]{a3.pdf}
%\caption{Arboreal singularity for $A_3$ tree.}
%\end{figure}
%

\begin{example}
Let us highlight the simplest class of trees.

When $T$ consists of a single vertex, $\sL_T$ is a single point.

When $T$ consists of two vertices $v_1, v_2$ (necessarily connected by an edge), $\sL_T$ is the local trivalent graph given by the cone over the three distinct points $\sL_T^\link$ representing the three correspondences
 $$
\xymatrix{
(\{v_1\} &  \ar@{->>}[l]_-= \{v_1\} \ar@{^(->}[r] & T)
&
(\{v_2\} &  \ar@{->>}[l]_-= \{v_2\} \ar@{^(->}[r] & T)
&
(\{v\} &  \ar@{->>}[l] T \ar@{^(->}[r]^-= & T)
}
$$
%\end{remark}

More generally, consider  the class of  $A_n$-trees $T_n$ consisting of $n$ vertices %$v_1, \ldots, v_n$ 
connected by $n-1$ successive edges. % $e_{1,2}, \ldots, e_{n-1, n}$.
The associated arboreal singularity $\sL_{T_n}$ 
admits an identification with the cone of the  $(n-2)$-skeleton of the $n$-simplex
$$
\xymatrix{
\sL^\link_{T_n} \simeq \Cone(sk_{n-2} \Delta^n)
}$$
or in a dual realization, the $(n-1)$-skeleton of the polar fan of the $n$-simplex.

\end{example}

%%%%%%%%%%%%%%%%%%%%%%%%%%%%%%%%%%%%%%%%%%%%%%%%%%%%%%%

\subsection{Arboreal hypersurfaces}

The basic notions and results about arboreal hypersurfaces from \cite{Narb} generalize immediately from trees to forests. 
We will review this material in this generality and only comment where there is any slight deviation from the presentation
of \cite{Narb}.

On the one hand, by convention, given a forest $F= \coprod_i T_i$, we set the corresponding arboreal space to be the disjoint union
of products of arboreal singularities  with Euclidean spaces
$$
\xymatrix{
\sL_F = \coprod_i (\sL_{T_i} \times \R^{F\setminus T_i})
}
$$
where $\R^{F\setminus T_i}$ denotes the Euclidean space of real tuples
$$
\xymatrix{
\{x_\gamma\},
\text{ with } \gamma\in V(F)\setminus V(T_i).
}
$$
or in other words, the  Euclidean space 
of functions
$$
\xymatrix{
\{x_\gamma\}: V(F)\setminus V(T_i) \ar[r] & \R 
}
$$

On the other hand, we can repeat  the constructions of \cite{Narb} for arboreal hypersurfaces starting from a rooted forest.
%Let us briefly highlight the form the main definitions and results take. 
Throughout the brief summary that follows, fix once and for all a rooted forest $\dirF$ which we can express 
as a disjoint union of rooted trees $\cF = \coprod_i \cT_i = \coprod_i (T_i, \rho_i)$.

%%%%%%%%%%%%%%%%%%%%%%%%%%%%%%%%%%%%%%%%%%%%%%%%%%%%%%%

\subsubsection{Rectilinear version}

 Let us write $\R^\dirF$ for the Euclidean space 
of real tuples
$$
\xymatrix{
\{x_\gamma\},
\text{ with } \gamma\in V(\dirF)
}
$$
or in other words, the  Euclidean space 
of functions
$$
\xymatrix{
\{x_\gamma\}: V(\dirF) \ar[r] & \R 
}
$$
so that we have
the evident identity
$$
\R^\dirF = \prod_i \R^{\dirT_i}
$$

\begin{defn} Fix a  vertex $\alpha\in V(\dirF)$.

 (1) Define the quadrant $Q_\alpha \subset \R^\dirF$ to be the closed subspace
$$
Q_\alpha = \{ x_\beta \geq 0 \text{ for all } \beta \leq \alpha\}
$$

(2) Define the hypersurface $H_\alpha \subset \R^\dirF$ to be the boundary
$$
H_\alpha = \partial Q_\alpha = \{ x_\beta \geq 0 \text{ for all } \beta \leq \alpha, \text{ and } x_\gamma = 0  \text{ for some } \gamma \leq \alpha \}
$$
\end{defn}

\begin{remark}\label{rem euclidean space}
Note that the hypersurface $H_\alpha \subset \R^\dirF$ is homeomorphic (in a piecewise linear fashion) to a Euclidean space of dimension $|V(\dirF)| - 1$.
\end{remark}

\begin{defn}
The {\em rectilinear arboreal hypersurface} $H_\dirF$ associated to  a rooted forest $\dirF$ is the union of hypersurfaces
$$
H_\dirF =\bigcup_{\alpha\in V(\dirF)}   H_\alpha \subset \R^\dirF
$$
\end{defn}

The rectilinear arboreal hypersurface  admits the following less redundant presentations.
Introduce the subspaces
$$
\xymatrix{
P_\alpha = \{x_\alpha = 0,  x_\beta \geq  0 \text{ for all } \beta<  \alpha \} \subset \R^{\dirF}
&
P_\alpha^\circ = \{x_\alpha = 0,  x_\beta >  0 \text{ for all } \beta<  \alpha \}    \subset \R^{\dirF}
}
$$

\begin{lemma}\label{lemma: less redundant}
$$
\xymatrix{
H_\dirF =\bigcup_{\alpha\in V(\dirF)} P_\alpha  \subset \R^\dirF
&
H_\dirF =\bigcup_{\alpha\in V(\dirF)} P_\alpha^\circ  \subset \R^\dirF
}
$$
\end{lemma}

\begin{proof}
For the first identity, if $p\in P_\alpha$,
then $x_\alpha(p) = 0$,  $x_\beta(p) \geq  0$ for all  $\beta<  \alpha$, and hence $p\in H_\alpha$.
Conversely, if $p\in H_\alpha$, then 
$x_\gamma(p) = 0$  for some  $\gamma \leq \alpha$, and $ x_\beta(p) \geq 0$  for all  $\beta \leq \alpha$,
in particular  $ x_\beta(p) \geq 0$  for all  $\beta\leq \gamma$, and hence $p\in P_\beta$.

To see the second identity, clearly $P^\circ_\alpha\subset P_\alpha$, and
observe that if $p\in P_\alpha\setminus P^\circ_\alpha$, then 
 $x_\beta(p) = 0$, for some $\beta<\alpha$, and if we take the minimum such $\beta$, then we have
$p\in P_\beta^\circ$.
\end{proof}

%Express the rooted forest $\cF$ as a disjoint union of rooted trees $\cF = \coprod_i \cT_i$. 
\begin{remark}
Introduce the inverse images under the natural projections
$$
\xymatrix{
H_{\dirF_i} = \pi_i^{-1}(H_{\dirT_i}) \subset \R^{\dirF}
&
\pi_i :\R^\dirF = \prod_i \R^{\dirT_i}\ar[r] & \R^{\dirT_i}
}
$$
Then we  have the evident identities
$$
\xymatrix{
H_{\dirF_i} \simeq H_{\dirT_i} \times \R^{\dirF\setminus \dirT_i}
&
H_\dirF = \bigcup_i H_{\dirF_i}
}
$$
Moreover, the inverse images  $H_{\dirF_i}$ are multi-transverse  hypersurfaces
 being the inverse images of complementary projections.
\end{remark}

%%%%%%%%%%%%%%%%%%%%%%%%%%%%%%%%%%%%%%%%%%%%%%%%%%%%%%%

\subsubsection{Smoothed version}\label{sect sm arb}

We recall here  the smoothed version of arboreal hypersurfaces. We  recall in the next section that the smoothed and rectilinear versions are homeomorphic as embedded hypersurfaces inside of Euclidean space.

Fix once and for all a small $\delta>0$. 

All of our constructions will depend on the choice of three functions denoted by 
$$
\xymatrix{
b:\R\ar[r] & \R &
f:\R^2\ar[r] & \R 
& c:\R\ar[r] & \R
}$$  the first two of which we will
select now.

Choose a continuous function $b:\R\to \R$, smooth away from $0\in \R$, with the  properties:
\begin{enumerate}
\item $|b(t)|<\delta/4$, for all $t\in \R$.
\item $b(t) = 0$ outside of the interval $0< t<   \delta/4$.
%\item $\lim_{t\to 0} b(t) = 0$.
\item $\lim_{t\to 0^+} b'(t) = -\oo$. %, for all $k\geq 1$.
\end{enumerate}

%Given the function $b:\R\to \R$, 
Choose a continuously differentiable function $f:\R^2\to \R$ with 
the  properties:
\begin{enumerate}
\item $f$ is a submersion.
\item $\{f(x_1, x_2)=0\} = \{x_1 = 0, x_2\geq 0\} \cup \{x_1> 0 , x_2 = b(x_1)\}$.
%\item $\{f(x_1, x_2)>0\} =\{x_1>0, x_2>b(x_1)\}$.
%\item $\{f(x_1, x_2)<0\} = \{ x_1 < 0\} \cup \{ x_1 = 0, x_2<0\} \cup \{x_1>0,  x_2<b(x_1)\}$.
\item $f(x_1, x_2) = x_2$ over  $\{x_1> 2\delta, |x_2| < \delta\}$.
\item $f(x_1, x_2) = x_1$  over  $\{|x_1|< \delta, x_2 > 2 \delta\}$.
\item $f(x_1, x_2)< \delta$ implies $x_1<\delta$ or $x_2<\delta$.
\end{enumerate}

\begin{remark}
If  preferred, one can fix some $N\geq 1$, and arrange that  $\lim_{t\to 0^+} b^{(k)}(t) = -\oo$, for all $1\leq k\leq N$. Then one can choose $f$ to be correspondingly
highly differentiable. One can also take $N=\infty$ and then choose $f$ to be smooth. 
\end{remark}

\begin{defn}
(1) For a root vertex $\rho\in V(\dirF)$,  set
$$
\xymatrix{
h_\rho = x_\rho:\R^{\dirF} \ar[r] & \R
}
$$

(2) For a non-root vertex $\alpha\in V(\dirF)$,  inductively define 
$$
\xymatrix{
h_\alpha :\R^{\dirF} \ar[r] & \R & h_\alpha = f(h_{\hat \alpha}, x_\alpha)
}
$$
where $\hat \alpha \in V(\dirF)$ is the  parent vertex of $\alpha$.
\end{defn}

\begin{remark}\label{rem coord depend} For all $\alpha\in V(\dirF)$, note that:

(1)  $h_\alpha$ is a submersion.

(2)    $h_\alpha$  depends only on the coordinates $x_\beta$, for  $\beta\leq \alpha$.

(3)   $h_\alpha  \geq 0$ implies   $ h_{\beta}  \geq 0$, for  $\beta\leq \alpha$. % so that $h_\alpha = 0$ implies  $h_{\hat\alpha}\geq 0$.

%
%(3) $\{h_\alpha = 0\} = \{h_\alpha =0, h_{\hat\alpha}>0\} \cup 
% \{h_\alpha =0, h_{\hat\alpha} = 0\}$.

\end{remark}

\begin{defn}
Fix a vertex $\alpha\in V(\dirF)$.
 
(1) Define the halfspace $\sQ_\alpha \subset \R^\dirF$ to be the closed subspace
$$
\sQ_\alpha = \{ h_\alpha \geq  0\} 
$$

(2) Define the hypersurface $\sH_\alpha \subset \R^\dirF$ to be the zero-locus
$$
\sH_\alpha = \{ h_\alpha = 0\} 
$$

\end{defn}

\begin{defn}

The {\em smoothed arboreal hypersurface} $\sH_\dirF$ associated to  a rooted forest $\dirF$ is the union
of hypersurfaces
$$
\sH_\dirF =\bigcup_{\alpha\in V(\dirF)}   \sH_\alpha \subset \R^{\dirF}
$$

\end{defn}

\begin{remark}
Introduce the subspaces
$$
\xymatrix{
\sP_\alpha = \{h_\alpha = 0,  h_{\hat\alpha} \geq 0 \} \subset \R^{\dirF}
&
\sP_\alpha^\circ = \{h_\alpha = 0,   h_{\hat\alpha} > 0 \} \subset \R^{\dirF}
}
$$
where $\hat \alpha \in V(\dirF)$ is the  parent vertex of $\alpha$.
Then
the smoothed arboreal hypersurface  admits the less redundant presentations 
$$
\xymatrix{
\sH_\dirF =\bigcup_{\alpha\in V(\dirF)} \sP_\alpha   \subset \R^\dirF
&
\sH_\dirF =\bigcup_{\alpha\in V(\dirF)} \sP^\circ_\alpha   \subset \R^\dirF
}
$$
\end{remark}

%Express the rooted forest $\cF$ as a disjoint union of rooted trees $\cF = \coprod_i \cT_i$.
\begin{remark}
Introduce the inverse images under the natural projections
$$
\xymatrix{
\sH_{\dirF_i} = \pi_i^{-1}(\sH_{\dirT_i}) \subset \R^{\dirF}
&
\pi_i :\R^\dirF = \prod_i \R^{\dirT_i}\ar[r] & \R^{\dirT_i}
}
$$
Then we  have the evident identities
$$
\xymatrix{
\sH_{\dirF_i} \simeq \sH_{\dirT_i} \times \R^{\dirF\setminus \dirT_i}
&
\sH_\dirF = \bigcup_i \sH_{\dirF_i}
}
$$
Moreover, the inverse images  $\sH_{\dirF_i}$ are multi-transverse hypersurfaces
 being the inverse images of complementary projections.
\end{remark}

%%%%%%%%%%%%%%%%%%%%%%%%%%%%%%%%%%%%%%%%%%%%%%%%%%%%%%%

\subsubsection{Comparison}
We  recall here that the rectilinear and smoothed arboreal hypersurfaces are homeomorphic as embedded hypersurfaces inside of Euclidean space.

Choose a smooth bump function $c:\R\to [0,1]$ with the properties:
\begin{enumerate}

\item $c(t) =0$ outside  the interval  $\{|t| \leq \delta\}$.

\item $c(t)  = 1$  inside the interval $\{|t|\leq \delta/2\}$.

\end{enumerate}

Using the functions $b, c:\R\to \R$, introduce
the vector field 
$$
\xymatrix{
v =  -b(x_1) c(x_2) \partial_{x_2} \in \Vect(\R^2)
}
$$
Observe that $v$ is smooth 
except  along the axis  $\{(0,x_2) \, |\, x_2\in \R\}\subset \R^2$  and satisfies:

\begin{enumerate}

\item $v =0$, outside  the rectangle  $\{0\leq x_1 \leq \delta/4, |x_2| \leq \delta\}$.

\item $v = -b(x_1)\partial_{x_2}$,  inside the domain $\{|x_2|\leq \delta/2\}$.

\end{enumerate}

Define
the homeomorphism
$
\Phi:\R^2\to \R^2
$ 
to be the unit-time flow  of the vector field $v$.
Observe that $\Phi$ is smooth 
except  along the axis  $\{(0,x_2) \, |\, x_2\in \R\}\subset \R^2$  and satisfies:

\begin{enumerate}

\item $\Phi(x_1, x_2) = (x_1, x_2)$, outside  the rectangle  $\{0\leq x_1 \leq \delta/4, |x_2| \leq \delta\}$.

\item $\Phi(x_1, x_2) = (x_1, x_2-b(x_1))$,  inside the domain $\{ |x_2|\leq \delta/4\}$.

\item For any fixed $a_1\in \R$, the restriction $\Phi|_{x_1 = a_1}:\R\to \R^2$ is smooth.

\end{enumerate}
The second property follows from the fact that $|b(x_1)|<\delta/4$, and $c(x_2) = 1$
when $|x_2|\leq \delta/2$, hence for less than or equal to unit-time, the flow of $v=  -b(x_1) c(x_2) \partial_{x_2} $ starting
from
 inside the domain $\{ |x_2|\leq \delta/4\}$  stays inside
   the domain $\{|x_2|\leq \delta/2\}$.

Introduce the continuous function $\varphi = x_2\circ \Phi:\R^2\to \R$ given by the second coordinate of $\Phi$.
Observe that $\varphi$ is smooth  except  along the axis  $\{(0,x_2) \, |\, x_2\in \R\}\subset \R^2$  and satisfies:

\begin{enumerate}

\item $\varphi(x_1, x_2) = x_2$, outside  the rectangle  $\{0\leq x_1 \leq \delta/4, |x_2| \leq \delta\}$.

\item $\varphi(x_1, x_2) = x_2-b(x_1)$, inside the domain $\{ |x_2|\leq \delta/4\}$.

\item For any fixed $a_1\in \R$, the restriction $\varphi|_{x_1 = a_1}:\R\to \R$ is a diffeomorphism. 

\end{enumerate}

\begin{defn}
(1) For a root vertex $\rho\in V(\dirF)$,  set
$$
\xymatrix{
F_\rho:\R^{\dirF} \ar[r] & \R &  F_\rho = x_\rho
}
$$

(2) For a non-root vertex $\alpha\in V(\dirF)$,  set
$$
\xymatrix{
F_\alpha :\R^{\dirF} \ar[r] & \R & F_\alpha = \varphi(h_{\hat \alpha}, x_\alpha)
}
$$
where $\hat \alpha \in V(\dirF)$ is the unique parent of $\alpha$.

(3)
Define the continuous map
$$
\xymatrix{
F_\dirF:\R^{\dirF} \ar[r] & \R^{\dirF} & F_\dirF = \{F_\alpha\}
}
$$
In other words, the coordinates of $F_\dirF$ are given by $x_\alpha\circ F_\dirF  = 
 F_\alpha$.
\end{defn}

\begin{remark}
Note that $F_\alpha$  depends only on the coordinates $x_\beta$, for  $\beta\leq \alpha$.
\end{remark}

%Express the rooted forest $\cF$ as a disjoint union of rooted trees $\cF = \coprod_i \cT_i$.
The map $F_\dirF$ is evidently the product of maps
$$
\xymatrix{
F_\dirF = \prod_i F_{\dirT_i}: \prod_i \R^{\dirT_i}\ar[r] &  \prod_i \R^{\dirT_i}
}
$$

Consequently, the analogous result for trees from~\cite{Narb} immediately implies the following extension to forests.

\begin{thm}\label{thm compare}
The map $F_\dirF:R^\dirF\to \R^\dirF$ is a homeomorphism and satisfies
$
F_\dirF(\sH_\dirF) = H_\dirF,
$
and in fact $F_\dirF(\sQ_\alpha) = Q_\alpha$, $F_\dirF(\sH_\alpha) = H_\alpha$, for all $\alpha\in V(\dirF)$.
\end{thm}

\begin{remark}
It follows that we  also have 
 $F_\dirF(\sP_\alpha) = P_\alpha$, $F_\dirF(\sP^\circ_\alpha) = P^\circ_\alpha$, for all $\alpha\in V(\dirF)$.
\end{remark}

\begin{remark}
For $\alpha\in V(\dirF)$,  introduce the continuous map
$$
\xymatrix{
\tilde F_\alpha :\R^{\dirF} \ar[r] & \R^{\dirF} 
}
$$
with coordinates given by
$$
x_\beta\circ \tilde F_\alpha  = 
\left\{ \begin{array}{ll}
 F_\alpha, &  \beta = \alpha\\
  x_\beta, & \beta \not = \alpha
\end{array}
\right.$$

Fix a total order on $V(\cF)$ compatible with its natural partial order.  Write $\alpha_1, \ldots, \alpha_{n+1} \in V(\cF)$ for the ordered vertices. 
Observe that $F_\cF$ factors as the composition
$$
\xymatrix{
F_\cF = \tilde F_{\alpha_1} \circ \cdots\circ \tilde F_{\alpha_{n+1}}  %:\R^{\dirF} \ar[r] & \R^{\dirF} 
}
$$
In particular,  since $F_\dirF$ is a homeomorphism, each $\tilde F_{\alpha}$
is itself a homeomorphism.

More precisely, observe that  
 for $\rho$ a root vertex, $\tilde F_\rho$ is the identity,
and for $\alpha$ not a root vertex,
each
$
\tilde F_{\alpha}
$
is the unit-time flow of the vector field 
$$
\xymatrix{
v_{\alpha} = -b(h_{\hat \alpha}) c(x_\alpha) \partial_{x_{\alpha}}\in \Vect(\R^\cF)
}$$
%In particular, $v_{\alpha}$ points solely in the  $\partial_{x_\alpha}$ direction,  and  
%%where the function $g_{\alpha}:\R^\cF\to \R$ 
%only depends on $x_{\beta}$, for $\beta< \alpha$.
In particular, $\tilde F_\alpha$ is the identity when $h_{\hat \alpha}\leq 0$ and is smooth when  $h_{\hat \alpha}>0$.

\end{remark}

\begin{remark}
By scaling the original function $b$ by a positive constant, one  obtains a family of smoothed arboreal hypersurfaces
all compatibly homeomorphic. Moreover, their limit as the scaling constant goes to zero is  the  rectilinear arboreal hypersurface.
Thus one can view the smoothed arboreal hypersurface  as a topologically trivial deformation of the rectilinear arboreal hypersurface.
\end{remark}

%%%%%%%%%%%%%%%%%%%%%%%%%%%%%%%%%%%%%%%%%%%%%%%%%%%%%%%

\subsubsection{Microlocal geometry}
Finally, we recall the relation between arboreal singularities and smoothed arboreal hypersurfaces.

Recall that the smoothed arboreal hypersurface $\sH_\dirF$  is the  union
$$
\xymatrix{
\sH_\dirF =\bigcup_{\alpha\in V(\dirF)}   \sH_\alpha \subset \R^{\dirF}
}$$
of  hypersurfaces cut out by submersions 
$$
\xymatrix{
\sH_\alpha =\{h_\alpha=0\} \subset \R^\dirF
}$$
Thus
 $\sH_\dirF \subset \R^\dirF$ is in good position, and moreover,  each  hypersurface $\sH_\alpha  \subset \R^\dirF$
  comes equipped with a preferred coorientation $\sigma_\alpha$
%  $$
%  \xymatrix{
%  \sigma_\alpha:\cL^*_{\sH_\alpha} \ar[r] & \cR^*_{\sH_\alpha}
% }
%  $$ 
given by the codirection pointing towards the halfspace 
  $$
  \xymatrix{
\sQ_\alpha =\{h_\alpha\geq 0\} \subset \R^\dirF
}  $$

Moreover, recall the inverse images under the natural projections
$$
\xymatrix{
\sH_{\dirF_i} = \pi_i^{-1}(\sH_{\dirT_i}) \subset \R^{\dirF}
&
\pi_i :\R^\dirF = \prod_i \R^{\dirT_i}\ar[r] & \R^{\dirT_i}
}
$$
and the evident identities
$$
\xymatrix{
\sH_{\dirF_i} \simeq \sH_{\dirT_i} \times \R^{\dirF\setminus \dirT_i}
&
\sH_\dirF = \bigcup_i \sH_{\dirF_i}
}
$$
Note that the inverse images are multi-transverse hypersurfaces being the inverse images of complementary projections.
%Thus we have the evident disjoint union identity
%$$
%\xymatrix{
%\cL_{\sH_\dirF}^* =   \coprod_i \cL_{\sH_{\dirF_i}}^* \subset P^*\R^\dirF
%}
%$$
By definition, we also have a parallel disjoint union identity
$$
\xymatrix{
\sL_F = \coprod_i (\sL_{T_i} \times \R^{F\setminus T_i})
}
$$

Thus the analogous result for trees from~\cite{Narb} immediately implies the following extension to forests.

\begin{thm}\label{thm: local smoothing}
Let $\dirF$ be a rooted forest
with  arboreal singularity $\sL_F$
and smoothed arboreal hypersurface $\sH_\dirF \subset \R^\dirF$,

(1) The smoothed arboreal hypersurface $\sH_\dirF \subset \R^\dirF$ is in good position with a natural coorientation $\sigma$ whose restriction to each $\sH_\alpha\subset \sH_\dirF$ is  the  coorientation $\sigma_\alpha$.

(2)  The positive coray  bundle $\Lambda_{\dirF}  \subset S^*\R^\dirF$ of the directed hypersurface $\sH_{\dirF}\subset \R^\dirF$ with coorientation $\sigma$ is  homeomorphic to  $\sL_F$.
%$$
%\xymatrix{
%\Phi_\dirF:\ar[r]^-\sim &  \Lambda_\dirF
% }
% $$
%whose   composition with the  natural  projection $\pi: \Lambda_\dirF\to \sH_\dirF$ restricts to homeomorphisms
% $$
% \xymatrix{
%\pi\circ \Phi_\dirF|_{\sL_F(\alpha)}: \sL_F(\alpha)\ar[r]^-\sim &  \sH_\alpha, & 
%\text{ for all }  \alpha\in V(\dirF)
%}
%$$
\end{thm}

%%%%%%%%%%%%%%%%%%%%%%%%%%%%%%%%%%%%%%%%%%%%%%%%%%%%%%%
%%%%%%%%%%%%%%%%%%%%%%%%%%%%%%%%%%%%%%%%%%%%%%%%%%%%%%%
%%%%%%%%%%%%%%%%%%%%%%%%%%%%%%%%%%%%%%%%%%%%%%%%%%%%%%%

\subsection{Generalized arboreal singularities}

We introduce here a modest generalization of arboreal singularities akin to the generalization from manifolds to manifolds with boundary.

Let $\dirF = \coprod \dirT_i = \coprod_i  (T_i, \rho_i)$ be a rooted forest. 

By the {\em leaf vertices} $\cL = \coprod_i \cL_i \subset V(\dirF) =  \coprod_i V(\cT_i)$, we will mean the set of %non-root 
vertices that are maxima 
with respect to the natural partial order. (A root vertex is  a maximum only if it is the sole vertex in its connected component; by the above definition such a vertex is %not 
also a leaf vertex.) 

%Since each non-root vertex $\alpha\in V(\dirF)$ has a unique parent vertex $\hat\alpha \in V(\dirF)$, each leaf vertex $\alpha\in \cL$ has a unique parent vertex
%$\hat \alpha \in V(\dirF) \setminus \cL$.

By a  {\em leafy rooted forest} $\dirF^* = (\dirF, \ell)= \coprod (\dirT_i, \ell_i) = \coprod_i  (T_i, \rho_i, \ell_i)$, we will mean  a rooted forest $\cF =   \coprod \dirT_i = \coprod_i  (T_i, \rho_i)$ together with a subset $\ell = \coprod_i \ell_i\subset \cL= \coprod_i \cL_i $
of marked leaf vertices.

To any leafy rooted forest $\dirF^* = (\dirF, \ell) $, we associate a  rooted forest $\dirF^+$ by starting  with $\dirF$ with its natural partial order and adding
a new maximum $\alpha^+\in V(\dirF^+)$ above each marked leaf vertex $\alpha\in \ell\subset V(\dirF)$. We continue to denote by $\ell\subset V(\dirF) \subset V(\dirF^+)$ the originally marked vertices. We denote by 
$\ell^+= V(\dirF^+)\setminus  V(\dirF)$ the newly added  vertices.
 Note that each $\alpha\in \ell^+ \subset V(\dirF^+)$ has  parent vertex  $\hat \alpha^+ = \alpha \in \ell \subset V(\dirF^+)$.
%%$$
%\xymatrix{
%V(\dirF^+) = V(\dirF) \coprod \ell 
%}
%$$

Throughout what follows,
let $\dirF^* = (\dirF, \ell)$ be a leafy rooted forest, and let  $\dirF^+$ be its associated rooted forest.
Our constructions will devolve to previous ones when $\ell = \emptyset$ and hence $\dirF^+ = \dirF$.

%%%%%%%%%%%%%%%%%%%%%%%%%%%%%%%%%%%%%%%%%%%%%%%%%%%%%%%

\subsubsection{Rectilinear version}

For any directed forest and in particular $\dirF^+$, recall
the rectilinear arboreal hypersurface  $H_{\dirF^+} \subset \R^{\dirF^+}$ admits the presentation as a union of closed subspaces
$$
\xymatrix{
H_{\dirF^+} =\bigcup_{\alpha\in V(\dirF^+)} P_\alpha   \subset \R^{\dirF^+}
&
P_\alpha = \{x_\alpha = 0,  x_\beta \geq 0 \text{ for all } \beta  <  \alpha \} \subset \R^{\dirF^+}
}
$$

\begin{defn}
The {\em rectilinear arboreal hypersurface} $H_{\dirF^*}$ associated to  the leafy rooted forest $\dirF^* = (\dirF, \ell)$ is the union of closed subspaces
$$
H_{\dirF^*} =\bigcup_{\alpha\in V(\dirF^+)\setminus \ell}   P_\alpha \subset H_{\dirF^+} \subset \R^{\dirF^+}
$$
\end{defn}

\begin{remark}
If $\ell = \emptyset$, so that $\dirF^+ = \dirF$, then $H_{\dirF^*} = H_{\dirF}$.
\end{remark}

\begin{example}\label{ex deg arb}
If $\ell = \dirF =\{\alpha\}$ consists of a single vertex, then $\dirF^+ = \{\alpha, \alpha^+\}$ consists  
of two vertices satisfying $\alpha < \alpha^+$. 
The  rectilinear arboreal singularity $H_{\dirF^*}$ is the closed half-line
$$
H_{\dirF^*} = P_{\alpha^+} = \{x_{\alpha^+} = 0,  x_\alpha \geq 0\}
$$
\end{example}

%%%%%%%%%%%%%%%%%%%%%%%%%%%%%%%%%%%%%%%%%%%%%%%%%%%%%%%

\subsubsection{Smoothed version}

For any directed forest and in particular $\dirF^+$, recall
the smoothed arboreal hypersurface  $\sH_{\dirF^+} \subset \R^{\dirF^+}$ admits the presentation as a union of closed subspaces
$$
\xymatrix{
\sH_{\dirF^+} =\bigcup_{\alpha\in V(\dirF^+)} \sP_\alpha   \subset \R^{\dirF^+}
&
\sP_\alpha = \{h_\alpha = 0,  h_{\hat\alpha} \geq 0 \} \subset \R^{\dirF^+}
}
$$

\begin{defn}
The {\em smoothed  arboreal hypersurface} $\sH_{\dirF^*}$ associated to  the leafy rooted forest $\dirF^* = (\dirF, \ell)$ is the union of closed subspaces
$$
\xymatrix{
\sH_{\dirF^*} 
=\bigcup_{\alpha\in V(\dirF^+) \setminus \ell} \sP_\alpha
  \subset \sH_{\dirF^+} \subset \R^{\dirF^+}
}
$$
\end{defn}

\begin{remark}
If $\ell = \emptyset$, so that $\dirF^+ = \dirF$, then $\sH_{\dirF^*} = \sH_{\dirF}$.
\end{remark}

\begin{remark}\label{rem: gen arb via homeo}

Recall the homeomorphism 
$$
\xymatrix{
F_{\dirF^+}: \R^{\dirF^+} \ar[r]^-\sim & \R^{\dirF^+}
}
$$
and that it satisfies
$
F_{\dirF^+}(\sP_{\alpha}) =  P_{\alpha}
$

Alternatively,  
the smoothed  arboreal hypersurface
$\sH_{\dirF^*}$ is
the image of the rectilinear arboreal hypersurface $H_{\dirF^*}$ under the inverse homeomorphism
$$
\xymatrix{
\sH_{\dirF^*}  = F_{\dirF^+}^{-1}(H_{\dirF^*}) 
  \subset \sH_{\dirF^+} \subset \R^{\dirF^+}}
  $$
  \end{remark}

%
%\begin{remark}
%By construction, the homeomorphism 
%$$
%\xymatrix{
%F_{\dirF^+}: \sH_{\dirF^+} \ar[r]^-\sim & H_{\dirF^+}
%}
%$$
%satisfies 
%$$
%\xymatrix{
%F_{\dirF^+}(\sH_{\dirF^*}) =  H_{\dirF^*}
%}
%$$
%\end{remark}

%%%%%%%%%%%%%%%%%%%%%%%%%%%%%%%%%%%%%%%%%%%%%%%%%%%%%%%

\subsubsection{Microlocal geometry}

For any rooted forest and in particular $\dirF^+$, recall that the smoothed arboreal hypersurface
$\sH_{\dirF^+}\subset \R^{\dirF^+}$ is a directed hypersurface with a natural coorientation, and its  positive coray bundle 
$\Lambda_{\dirF^+} \subset S^*\R^{\dirF^+}$ 
is homeomorphic to the arboreal singularity $\sL_{F^+}$.

By definition, the smoothed arboreal hypersurface $\sH_{\dirF^*}\subset \R^{\dirF^+}$ is a closed subspace of $\sH_{\dirF^+} \subset \R^{\dirF^+}$, and hence it is in good position and inherits a natural coorientation. Thus its 
 positive coray bundle 
$\Lambda_{\dirF^*} \subset S^*\R^{\dirF^+}$ 
is a closed subspace of $\Lambda_{\dirF^+} \subset S^*\R^{\dirF^+}$, and hence homeomorphic to a closed subspace of the arboreal singularity $\sL_{F^+}$.

To identify this closed subspace, let us identify its open complement. Recall that 
$\sL_{F^+} $ is stratified by cells indexed by correspondences
of the form 
$$
\xymatrix{
\fp=(R &  \ar@{->>}[l]_-q  S \ar@{^(->}[r]^-i & F^+)
}
$$
where $i$ is the inclusion of  a subtree and $q$ is a quotient  of trees. (Strictly speaking, we have only stated this
cell decomposition for trees, but it holds immediately for forests:
by definition, the arboreal space of a forest is the disjoint union of the arboreal spaces of the connected components
of the forest; and 
for correspondences of the above form, the inclusion $i$ must take its domain tree to a single connected
component of its codomain forest.)

Given a marked leaf vertex $\alpha\in \ell \subset V(\dirF^+)$, with added maximum vertex $\alpha^+ \in \ell^+ \subset V(\dirF^+)$ so that $\hat \alpha^+ = \alpha$, consider the two correspondences
$$
\xymatrix{
\fp_\alpha=(\{pt\} &  \ar@{->>}[l]  \{\alpha\} \ar@{^(->}[r] & F^+)
}
$$
$$
\xymatrix{
\fp_{\alpha^+, \alpha}=(\{pt\} &  \ar@{->>}[l]  \{\alpha^+, \alpha\} \ar@{^(->}[r] & F^+)
}
$$
Since the correspondences begin with a singleton $\{pt\}$, they are maxima in the correspondence poset,
and hence index open cells in $\sL_{F^+}$.

\begin{prop}\label{prop: microlocal ends} The positive coray  bundle $\Lambda_{\dirF^*} \subset S^*\R^{\dirF^+}$ is homeomorphic to the closed subspace
of the arboreal singularity $\sL_{F^+}$ given by deleting the open cells indexed by 
the correspondences $\fp_\alpha$, $\fp_{\alpha^+, \alpha}$, for all $\alpha\in \ell \subset V(\dirF^+)$.

\end{prop}

\begin{proof}
For each $\alpha\in \ell\subset V(\dirF^+)$, introduce the  subspaces
$$
\xymatrix{
P^*_\alpha =\{x_\alpha = 0, x_\beta >0 \text{ for all } \beta<\alpha, , x_{\alpha^+} \not = 0\}
\subset H_{\dirF^+} 
\subset \R^{\dirF^+}
}
$$
$$
\xymatrix{
\sP^*_\alpha =F^{-1}_{\dirF^+}(P^*_\alpha) 
=  \{h_\alpha = 0, h_{\hat \alpha} >0, x_{\alpha^+} \not = 0\}
\subset \sH_{\dirF^+} \subset \R^{\dirF^+}
}
$$

Observe that $\sP^*_\alpha$ is an open submanifold of $\sH_\alpha =\{h_\alpha = 0\}$  and hence comes with a preferred coorientation $\sigma_\alpha$
with associated positive coray bundle $\Lambda_{\sP^*_\alpha} \subset \Lambda_{\dirF^+}$.

\begin{lemma}\label{lemma: microlocal ends}
$$
\xymatrix{
\Lambda_{\dirF^+} = \Lambda_{\dirF^*} \cup \bigcup_{\alpha\in \ell} \Lambda_{\sP^*_\alpha}
}$$
Moreover,    each $\Lambda_{\sP^*_\alpha}$
 is disjoint from $\Lambda_{\dirF^*}$ and each other.
\end{lemma}

\begin{proof}[Proof of Lemma~\ref{lemma: microlocal ends}]

Observe that
$$
\xymatrix{
H_{\dirF^+} = H_{\dirF^*} \cup \bigcup_{\alpha\in \ell} P^*_\alpha
}$$
To see this,  recall that 
$$
\xymatrix{
H_{\dirF^+} = H_{\dirF^*} \cup \bigcup_{\alpha\in \ell} P_\alpha & 
H_{\dirF^*} =\bigcup_{\alpha\in V(\dirF^+)\setminus \ell}  P_\alpha
}$$
Suppose $p\in P_\alpha\setminus P^*_\alpha$. Then either $x_\beta(p) = 0$, for some $\beta<\alpha$, in which case
$p\in P_\beta \subset H_{\dirF^*}$, or 
$x_{\alpha^+}(p) = 0$, in which case
$p\in P_{\alpha^+} \subset H_{\dirF^*}$.

Thus applying $F_{\dirF^+}^{-1}$, we  also obtain
$$
\xymatrix{
\sH_{\dirF^+} = \sH_{\dirF^*} \cup \bigcup_{\alpha\in \ell} \sP^*_\alpha
}$$
Observe further that $\sP^*_\alpha$ is tangent to $\sP_{\hat \alpha}$ along 
$h_{\hat \alpha} = 0$, and   tangent to $\sP_{\alpha^+}$ along $x_{\alpha^+} = 0$.
Thus the boundary $\ol \Lambda_{\sP^*_\alpha} \setminus \Lambda_{\sP^*_\alpha}$ is contained in $\sH_{\dirF^*}$,
and hence we obtain the first assertion
$$
\xymatrix{
\Lambda_{\dirF^+} = \Lambda_{\dirF^*} \cup \bigcup_{\alpha\in \ell} \Lambda_{\sP^*_\alpha}
}$$

Now  let us turn to the second assertion.

First,  $P^*_\alpha\cap P_{\alpha^+} = \emptyset$, since 
$p\in P^*_\alpha$ implies $x_{\alpha^+}(p) \not = 0$
so $p\not\in P_{\alpha^+}$.

Similarly, if $\beta<\alpha$, then $P^*_\alpha\cap P_\beta = \emptyset$, since $p\in P^*_\alpha$ implies $x_\beta(p) >0$
so $p\not\in P_\beta$.

Finally, if $\gamma$ and $\alpha$ are incomparable, in particular if $\gamma$  also lies in $\ell$, then the following intersection is obviously transverse
%$
%P^*_\alpha \cap P^\circ_\gamma.
%$
%will be nonempty. Replacing $\gamma$ by another if necessary,
%by Lemma~\ref{}, we may focus on 
$$
\xymatrix{
P^*_\alpha \cap P_\gamma = \{x_\alpha = 0, x_\beta >0 \text{ for all } \beta<\alpha, x_{\alpha^+} \not = 0\} \cap
\{x_\gamma = 0, x_\delta \geq 0 \text{ for all } \delta<\gamma\}
}
$$

We claim that the homeomorphism $F_{\dirF^+}^{-1}$ preserves the transversality of the above intersection
thus establishing the second assertion.
To check this, 
fix a total order on $V(\cF)$ compatible with its natural partial order, write $\alpha_1, \ldots, \alpha_{n+1} \in V(\cF)$ for the ordered vertices, and 
recall the factorization
$
F_\cF = \tilde F_{\alpha_1} \circ \cdots\circ \tilde F_{\alpha_{n+1}}$.

Since each $\tilde F_{\beta}$ preserves all coordinates except $x_{\beta}$, we are reduced to showing
that $\tilde F_{\alpha}^{-1}$ and $\tilde F_{\gamma}^{-1}$ 
preserve the transversality of the above intersection.
But each only changes the corresponding coordinate as a function of the coordinates less than it
in the partial order. Since $\alpha$ and $\gamma$ are incomparable by assumption, the asserted transversality follows.
 \end{proof}

Finally, to complete the proof of Prop.~\ref{prop: microlocal ends}, 
by construction~\cite{Narb}, the disjoint union of the open cells of $\sL_{F^+}$ indexed by 
the correspondences $\fp_\alpha$, $\fp_{\alpha^+, \alpha}$ maps homeomorphically to 
$
\sP^*_\alpha 
$ 
under the natural projection
$$
\xymatrix{
S^*\R^{\dirF^+} \ar[r] &  \R^{\dirF^+}.
}$$ 
More precisely,  the open cell indexed by 
$\fp_\alpha$ maps to the locus $\sP^*_\alpha\cap \{x_{\alpha^+} < 0\}$, and  the open cell 
indexed by 
$\fp_{\alpha^+, \alpha}$ maps to the locus  $\sP^*_\alpha\cap \{x_{\alpha^+} > 0\}$.
\end{proof}

\begin{remark}
If $\ell = \emptyset$, so that $\dirF^+ = \dirF$, then $\Lambda_{\dirF^*}$ is homeomorphic to $\sL_{F}$ itself.
\end{remark}

\begin{example}

If $\ell = \dirF =\{\alpha\}$ consists of a single vertex, then $\dirF^+ =\{\alpha, \alpha^+\}$ consists  
of two vertices satisfying $\alpha < \alpha^+$. 

Recall that $\sL_{F^+}$  is the local trivalent graph given by the cone over  three  points  indexed by the three correspondences
 $$
\xymatrix{
(\{pt\} &  \ar@{->>}[l]_-= \{\alpha\} \ar@{^(->}[r] & F^+)
&
(\{pt \} &  \ar@{->>}[l]_-= \{\alpha^+\} \ar@{^(->}[r] & F^+)
&
(\{pt\} &  \ar@{->>}[l] F^+ \ar@{^(->}[r]^-= & F^+)
}
$$

To obtain $\Lambda_{\dirF^*}$, we  start with
$\sL_{F^+}$ and delete the two open cells indexed by the first and third of the above correspondences.
What results is a closed half-line, the cone over the remaining point indexed by the middle correspondence.
Note the agreement with Example~\ref{ex deg arb}.
\end{example}

%%%%%%%%%%%%%%%%%%%%%%%%%%%%%%%%%%%%%%%%%%%%%%%%%%%%%%%
%%%%%%%%%%%%%%%%%%%%%%%%%%%%%%%%%%%%%%%%%%%%%%%%%%%%%%%
%%%%%%%%%%%%%%%%%%%%%%%%%%%%%%%%%%%%%%%%%%%%%%%%%%%%%%%

\section{Expansion algorithm}\label{s exp}

%%%%%%%%%%%%%%%%%%%%%%%%%%%%%%%%%%%%%%%%%%%%%%%%%%%%%%%

\subsection{Setup}
 \label{ss: strat}

Let $M$ be a  manifold.

Let $H\subset M$ be a directed hypersurface with 
positive coray bundle $\Lambda \subset S^*M$.

Fix a Whitney stratification $\{H_{\ul i}\}_{\ul i\in \ul I}$ of the hypersurface $H\subset M$.
(The reason for the presently superfluous underlining of the indices will become apparent soon below.)
As usual, we will regard the index set $\ul I$ of the stratification as a  poset with partial order 
$$
\xymatrix{
\ul i <\ul j & \mbox{ if and only if } & 
H_{\ul i}\subset \ol H_{\ul j}, H_{\ul i}\not = H_{\ul j}
}$$

To simplify the exposition, we will make the following  first of several mild assumptions.

\begin{assumption}
We will assume  there is a compactification $M\subset \ol M$ so that
the stratification of $H\subset M$ is the restriction of a stratification of the closure $\ol H\subset \ol M$.
\end{assumption}

In particular, this implies the  index set $\ul I$ of the stratification is finite.

%\begin{remark}
%Thus to apply our constructions in a local situation, one could take $H\subset M$ to be the intersection of a global 
%hypersurface with a small closed ball around a point.
%\end{remark}

For each $\ul i\in \ul I$, introduce  the restriction of the positive coray bundle
$$
\xymatrix{
\Lambda_{\ul i} =\Lambda\times_H H_{\ul i} \subset \Lambda
}
$$

Next, we will assume the following additional simplifying property 
of the stratification which can be achieved by refining the stratification if necessary,
for example so that the strata %$H_{\ul i} \subset H$ 
are simply-connected.

\begin{assumption}
%(1) For each $\ul i \in \ul I$, we will assume  $H_{\ul i}$ is connected and locally connected.
%(2) 
For each $\ul i\in \ul I$, we will assume  the finite map $\Lambda_{\ul i} \to H_{\ul i}$  is a trivial bundle.
\end{assumption}

For each $\ul i \in \ul I$,  fix once and for all a  trivialization 
$$
\xymatrix{
\Lambda_{\ul i} \simeq H_{\ul i} \times F_{\ul i}
}
$$ 
where  $F_{\ul i}$ is a finite set. 

Introduce the set $I$  of pairs $i = (\ul i, f)$ where $\ul i \in \ul I$ and $f \in F_{\ul i}$,
 and the natural projection 
$$
\xymatrix{
I \ar[r] & \ul I &  i = (\ul i, f) \ar@{|->}[r] & \ul i
}$$
For each $\ul i\in \ul I$, we will regard $F_{\ul i}$ as a subset of $I$, 
and often write $i\in F_{\ul i}$ when $i \mapsto \ul i$ without specifying that $i = (\ul i, f)$.

For each $i = (\ul i, f) \in F_{\ul i}$, we will write $\Lambda_{i} \subset \Lambda$ for the subspace
$$
\xymatrix{
\Lambda_i = H_{\ul i} \times \{f\} \subset H_{\ul i} \times F_{\ul i} \simeq \Lambda_{\ul i} \subset \Lambda
}$$
Note that projection provides a diffeomorphism
$$
\xymatrix{
\Lambda_i \ar[r]^-\sim & H_{\ul i}
}
$$ 

We have a disjoint decomposition into submanifolds 
$$
\xymatrix{
\Lambda = \coprod_{i\in I} \Lambda_i
}
$$
The decomposition satisfies the axiom of the frontier but we will not worry about whether it is a Whitney stratification.
We will regard the index set $I$ as a finite poset with partial order 
$$
\xymatrix{
 i < j & \mbox{ if and only if } & 
\Lambda_i\subset \ol \Lambda_j, \Lambda_i\not = \Lambda_j
}$$ 
 The projection $I \to \ul I$ respects the poset structures in the sense that $i <j$ implies $\ul i < \ul j$ (though not necessarily the converse).

Finally, to further simplify future notational demands, we will assume the following simplifying property 
of the stratification which can be achieved by further refining the stratification if necessary.

\begin{assumption}
%(1) For each $\ul i \in \ul I$, we will assume  $H_{\ul i}$ is connected and locally connected.
%(2) 
For each $\ul i\in \ul I$, we will assume  the stratum $H_{\ul i} \subset H$  is locally connected.
\end{assumption}

The assumption has the following implication which will help simplify the exposition and notation around further constructions.

\begin{lemma}
Given $\ul i\in \ul I$ and $j\in I$ with $\ul i <\ul j$, there exists a unique $i \in F_{\ul i}$ such that $i< j$.
\end{lemma}

\begin{proof}
First, note there exists $i \in F_{\ul i}$ with $i< j$ since the projection $\Lambda\to H$ is proper.
Next, if there were two such $i, i'\in F_{\ul i}$, % so that $\Lambda_i, \Lambda_{ i'} \subset \Lambda_j$, 
then $ H_{\ul j}\subset H$ would not be locally connected near $H_{\ul i}\subset H$. Namely, if we choose disjoint open neighborhoods $U_i\subset \Lambda$ of  $\Lambda_i \subset \Lambda$, for all $i\in F_{\ul i}$, then
near $H_{\ul i}\subset H$, we would have that
$ H_{\ul j}\subset H$ is the disjoint union of the homeomorphic images of the open subsets $\Lambda_j \cap U_i \subset \Lambda_j$.
\end{proof}

The above assertion immediately implies the following useful statements. Given a poset $I$, and an element $j\in I$,
we will write $I_{\leq j} = \{ i \in I \, | \, i \leq  j\}$ and $I_{\geq j} = \{ i \in I \, | \, i \geq  j\}$ for the induced subposets. 
Given a subset $J\subset I$, we will write  $I_{\leq J} = \cup_{j \in J} I_{\leq j}$ and   
$I_{\geq J} = \cup_{j \in J} I_{\geq j}$ for the induced subposets.

\begin{corollary}\label{cor: simplifying assumptions}
(1) For each $j\in I$, the natural projection of subposets 
$$
\xymatrix{
I_{\leq j} \ar[r] & I_{\leq \ul j}
}
$$ 
is an isomorphism.

(2) Given $\ul i \in \ul I$ with preimage $F_{\ul i} \subset I$, we have the  decomposition
$$
\xymatrix{
I_{\geq F_{\ul i}} = \coprod_{ i \in F_{\ul i}} I_{\geq i}
}
$$
into disjoint incomparable subposets.
\end{corollary}

The first assertion of the corollary implies for each $j\in I$, the natural projection of closed subspaces\
is a homeomorphism
$$
\xymatrix{
\bigcup_{i\leq j} \Lambda_i \ar[r]^-\sim & \bigcup_{\ul i\leq \ul j} H_{\ul i}
}
$$
The second assertion  implies for each $\ul i \in \ul I$, there is a disjoint union decomposition of open subspaces
$$
\xymatrix{
\bigcup_{j \in I_{\geq F_{\ul i}} } \Lambda_j =  \coprod_{i\in F_{\ul i }} \bigcup_{j \geq i} \Lambda_{j}
}
$$

%%%%%%%%%%%%%%%%%%%%%%%%%%%%%%%%%%%%%%%%%%%%%%%%%%%%%%%
%%%%%%%%%%%%%%%%%%%%%%%%%%%%%%%%%%%%%%%%%%%%%%%%%%%%%%%

\subsection{Expanded cylinder}\label{s: exp cyl}

We continue with the setup of the preceding section.

Fix a  compatible system of control data 
$
\{(T_{\ul i}, \rho_{\ul i}, \pi_{\ul i})\}_{\ul i \in \ul I}
$ 
%compatible with the Whitney stratification $\{H_{\ul i}\}_{\ul i\in \ul I}$ of the hypersurface $H\subset M$.

%%%%%%%%%%%%%%%%%%%%%%%%%%%%%%%%%%%%%%%%%%%%%%%%%%%%%%%

\subsubsection{Multi-transverse functions}

For each $i\in I$, choose a small positive radius $r_i\in \R_{>0}$ so that $r_i \not = r_{i'}$ whenever $\ul i = \ul i'$. 
%Eventually, we  be sure that we choose $r_i$ much smaller than $r_j$ whenever $i>j$.

\begin{defn}
For each $i\in I$, introduce the function 
$$
\xymatrix{
\ssf_i: T_{\ul i} \ar[r] & \R & \ssf_i = \rho_{\ul i} - r_i 
}
$$
\end{defn}

%Recall the notion of multi-transversality from Sect.~\ref{}.

%that for a finite collection of open domains  $U_a\subset M$,
%and functions $f_a:U_a\to \R$, indexed by $a\in A$, 
%we say the collection of functions $\{f_a\}_{a\in A}$ is multi-transverse at its total zero value
%if the following holds: for any subset $B\subset A$, the external product function
%$$
%\xymatrix{
%\prod_{b\in B} f_b:\bigcap_{b\in B} U_b \ar[r] & \R^B
%}
%$$
%is a submersion at its total zero value $0\in \R^B$. In particular, zero is a regular value of each function and
%the zero-loci of any two functions are transverse.

\begin{lemma}
The collection of functions $\{\ssf_i\}_{i\in I}$ is multi-transverse at its total zero value.
\end{lemma}

\begin{proof}
Since the radii $r_i\in \R_{>0}$ are distinct $r_i \not = r_{i'}$ whenever $\ul i = \ul i'$,
the zero locus of a subcollection of functions is nonempty only if for each $\ul i \in \ul I$, the subcollection contains at most one function indexed  by an $i\in I$ lying over $\ul i$. For such subcollections, the multi-transversality
is the usual multi-transversality of the collection  $\{\rho_{\ul i}\}_{\ul i\in \ul I}$ of tubular distance functions at any collection of non-zero values.
\end{proof}

\subsubsection{Truncated strata}

\begin{defn} For each $i\in I$,
define the {\em truncated stratum} $H^\trunc_i \subset H_{\ul i}$ to be the closed subspace of $x\in H_{\ul i}$ cut out by the equations
$$
\xymatrix{
\ssf_a(x) \geq  0, \mbox{ whenever }  a<i \mbox{ and } x\in H_{\ul i} \cap T_{\ul a}
}
$$
\end{defn}

\begin{lemma}\label{lemma: truncated strata}
(1) The truncated stratum $H^\trunc_i\subset H_{\ul i}$ is a closed submanifold with corners.

(2) The codimension $k$ corners of $H^\trunc_i$ are indexed by $a_1, \ldots, a_k\in I$ with $a_1< \cdots <a_k<i$.

\end{lemma}

\begin{proof}
Thanks to statement (1) of Corollary~\ref{cor: simplifying assumptions}, the lemma reduces to the same assertion for the 
tubular distance functions of a system of control data
which is a  standard  fact.
\end{proof}

\begin{remark}
Of course if $i\in I$ is a minimum, so that $H_{\ul i}\subset H$ is a closed stratum, then 
we have $H_i^\trunc = H_{\ul i}$.
\end{remark}

%%%%%%%%%%%%%%%%%%%%%%%%%%%%%%%%%%%%%%%%%%%%%%%%%%%%%%%

\subsubsection{Truncated cylinders}

\begin{defn}\label{def: truncated cyls}
For each $i\in I$, define
the {\em truncated cylinder} $C_i \subset T_{\ul i}$ to be the subspace 
of $x\in T_{\ul i}$ cut out by the equations
$$
\xymatrix{
\ssf_i(x) = 0
&
\ssf_a(x) \geq  0, \mbox{ whenever }  a<i \mbox{ and } x\in T_{\ul i} \cap T_{\ul a}
}
$$
\end{defn}

\begin{remark}
Equivalently, by the axioms of a control system, the truncated cylinder  $C_i \subset T_{\ul i}$ is the subspace 
of $x\in T_{\ul i}$ cut out by the equations
$$
\xymatrix{
\ssf_i(x) = 0 & \pi_{\ul i}(x) \in H^\trunc_i  
}
$$
\end{remark}

\begin{lemma}\label{lemma: truncated cyls}
(1) The truncated cylinder $C_i\subset T_{\ul i}$ is a closed submanifold with corners.

(2) The projection $\pi_{\ul i}$ exhibits $C_i$ as a $(\on{codim}_M H_{\ul i}-1)$-sphere bundle over $H^\trunc_{\ul i}$.
\end{lemma}

\begin{proof}
Immediate from Lemma~\ref{lemma: truncated strata}.
\end{proof}

\begin{remark}
Of course if $i\in I$ is a minimum, so that $H_{\ul i}\subset H$ is a closed stratum, then 
the truncated cylinder $C_i\subset M$ is also closed and cut out simply by $\ssf_i(x) = 0$.
\end{remark}

%%%%%%%%%%%%%%%%%%%%%%%%%%%%%%%%%%%%%%%%%%%%%%%%%%%%%%%

\subsubsection{Total cylinder}

\begin{defn}\label{def: total cyl}
Define the {\em total cylinder} $C \subset M$ to be the union of truncated cylinders
$$
\xymatrix{
C = \bigcup_{i\in I} C_i
}
$$
\end{defn}

\begin{prop}\label{prop: cyl sings}
The singularities of the total cylinder $C \subset M$ are rectilinear arboreal hypersurface singularities.
\end{prop}

\begin{proof}
Fix a point $p\in M$. 

Let $I_p \subset I$ comprise indices $i\in I$ such that $p\in C_i \subset T_{\ul i}$, so  in particular
$\ssf_i:T_{\ul i} \to \R$ vanishes at $p$. We will regard $I_p \subset I$ as a poset with the induced partial order: $i, j \in I_p$
satisfy $i<j$ inside of $I_p$
if and only if $i<j$ inside of $I$.

%Recall that the collection of functions $\{\ssf_i\}_{i\in I}$ is multi-transverse. 
By construction, it suffices to see that $I_p$ is the poset of a rooted forest $\cI_p$,
and thus
the singularity of $C$ at the point $p$ is
the rectilinear arboreal hypersurface $H_{\cI_p}$. More precisely, there will be an open ball 
$U\subset  \R^{\cI_p} \times \R^{k_p}$,
with $k_p = \dim M - |I_p|$ and $0\in U$, and
  a smooth open embedding 
 $$
 \xymatrix{
 \varphi:U \ar[r]^-\sim &  \varphi(U) \subset M
 }$$ 
 such that the following holds
 $$
 \xymatrix{
 \varphi(0) = p &
 \varphi(U \cap (H_{\cI_p} \times \R^{k_p})) = \varphi(U) \cap C 
 }$$
 $$
 \xymatrix{
x_i =   \ssf_i \circ \varphi: U \ar[r] & \R &
\mbox{ for all }  i\in I_p
 }
 $$
 Then for $i\in \cI_p$, the constructions with the  coordinates $x_i$ immediately match those of Definitions~\ref{def: truncated cyls} with the functions $\ssf_i$. 
 
So let us check that  $I_p$ is the poset of a rooted forest $\cI_p$. For this, it suffices to show that for any $i\in I_p$ 
that is not a minimum, there is a unique parent $\hat \imath\in I_p$ such that $\hat \imath < i$ and no $j\in I_p$ satisfies
$\hat \imath < j< i$. Recall that $i\in I_p$ means $p\in C_i$. By Lemmas~\ref{lemma: truncated strata} and ~\ref{lemma: truncated cyls}, $C^\trunc_i \subset M$ is a closed submanifold with codimension $k$ corners indexed by
(possibly empty) sequences $a_1, \ldots, a_k\in I$ with $a_1< \cdots < a_k< i$
such that $\ssf_{j}(p) = 0$ if and only if $j = a_\ell$ for some $\ell=1,\ldots, k$.
Now $p\in C_i$ lies in some corner indexed by  such a sequence. If the sequence is empty, then clearly $i\in I_p$ is a minimum, else the unique parent of $i\in I_p$ is  clearly the maximum of the sequence $\hat\imath  = a_k\in I_p$.
\end{proof}

%%%%%%%%%%%%%%%%%%%%%%%%%%%%%%%%%%%%%%%%%%%%%%%%%%%%%%%
%%%%%%%%%%%%%%%%%%%%%%%%%%%%%%%%%%%%%%%%%%%%%%%%%%%%%%%

\subsection{Smoothing into good position}

%We continue with the constructions of the preceding section. 

The total cylinder $C\subset M$
is a hypersurface with rectilinear arboreal hypersurface singularities. Our aim here is to amend its construction to produce a  homeomorphic deformation of it to a directed hypersurface $\sC\subset M$ with smoothed arboreal
hypersurface  singularities.

%%%%%%%%%%%%%%%%%%%%%%%%%%%%%%%%%%%%%%%%%%%%%%%%%%%%%%%
%%%%%%%%%%%%%%%%%%%%%%%%%%%%%%%%%%%%%%%%%%%%%%%%%%%%%%%

\subsubsection{Good charts}
Fix a point $p\in M$. 

Let $I_p \subset I$ comprise indices $i\in I$ such that $p\in C_i \subset T_{\ul i}$, so  in particular
$\ssf_i:T_{\ul i} \to \R$ vanishes at $p$. We will regard $I_p \subset I$ as a poset with the induced partial order: $i, j \in I_p$
satisfy $i<j$ inside of $I_p$
if and only if $i<j$ inside of $I$.
 Prop.~\ref{prop: cyl sings} confirms that $I_p$ is the poset of a rooted forest $\cI_p$ and the arboreal singularity of $C$ at the point $p$ is
that associated to $\cI_p$.

By a {\em good chart} $(U, \varphi)$ centered at $p \in C$, we will mean an open ball 
$U\subset  \R^{\cI_p} \times \R^{k_p}$,
with $k_p = \dim M - |I_p|$ and $0\in U$, and
  a smooth open embedding 
 $$
 \xymatrix{
 \varphi:U \ar[r]^-\sim &  \varphi(U) \subset M
 }$$ 
 such that the following holds
 $$
 \xymatrix{
 \varphi(0) = p &
 \varphi(U \cap (H_{\cI_p} \times \R^{k_p})) = \varphi(U) \cap C 
 }$$
 $$
 \xymatrix{
x_i =   \ssf_i \circ \varphi: U \ar[r] & \R &
\mbox{ for all }  i\in I_p
 }
 $$
The proof of Prop.~\ref{prop: cyl sings} confirms there is a good chart centered at any point.

 \begin{remark}
A good chart $(U, \varphi)$ centered at $p \not \in H$ so that $\cI_p = \emptyset$ is simply a coordinate chart
such that $\varphi(U)\cap C = \emptyset$.
\end{remark}

 \begin{remark}
Suppose  $(U_1, \varphi_1)$,  $(U_2,  \varphi_2)$ are good charts centered at $p \in H$. Introduce the 
open subsets
$$
\xymatrix{
U'_1 =  \varphi_1^{-1}(\varphi_1(U_1) \cap \varphi_2(U_2)) \subset \R^{\cI_p} \times \R^{k_p}
&
U_2' =  \varphi_2^{-1}(\varphi_1(U_1) \cap \varphi_2(U_2))\subset \R^{\cI_p} \times \R^{k_p}
}$$
and the diffeomorphism
$$
\xymatrix{
\psi = \varphi_2^{-1} \circ \varphi_1: U_1' \ar[r]^-\sim & U_2'
}$$
By construction, $\psi$ satisfies
$$
 \xymatrix{
x_i = x_i \circ \psi:U_1'\ar[r] & \R &
\mbox{ for all }  i\in I_p
 }
 $$
 Thus $\psi$ is a shearing transformation in the sense that it takes the form
 $$
 \xymatrix{
\psi = \id_{\R^{\cI_p}} \times \tilde \psi: U_1' \ar[r]^-\sim &  U_2'
&
\tilde \psi : U_1'\ar[r] &  \R^{k_p}
 }
 $$
 
 More generally, suppose  $(U_1, \varphi_1)$,  $(U_2,  \varphi_2)$ are good charts centered at $p_1, p_2 \in H$
 respectively. 
 Then in the same notation as above, the diffeomorphism $\psi$ satisfies
%$$
%\xymatrix{
%\psi(U_1\cap (H_{\cI_p} \times \R^k)) = U_2\cap (H_{\cI_p} \times \R^k)
%}
%$$
$$
 \xymatrix{
x_i = x_i \circ \psi:U_1'\ar[r] & \R &
\mbox{ for all }  i\in I_{p_1} \cap I_{p_2}
 }
 $$
  \end{remark}

 %%%%%%%%%%%%%%%%%%%%%%%%%%%%%%%%%%%%%%%%%%%%%%%%%%%%%%%

 \subsubsection{Global smoothing}

Choose an open covering of $M$ by good charts 
$
\{(U_a, \varphi_a)\}_{a\in A}
$
centered at points $p_a\in M$.

Set $I_a\subset I$ to contain 
 indices $i\in I$ such that $p_a\in C_i \subset E_{\ul i}$, so  in particular
$\ssf_i:T_{\ul i} \to \R$ vanishes at $p_a$. 
Recall that $I_a$ is the poset of a rooted forest $\cI_a$ and the arboreal singularity of $C$ at the point $p_a$ is
that associated to $\cI_a$.

We will only be interested in a neighborhood of $C\subset M$, so will throw out any $a\in A$ such that $\varphi_a(U_a)\cap C=\emptyset$.  Since $H$ is assumed to be compactifiable,  $C$ is also compactifiable, and hence we may assume
$A$ is finite. 
%Let us write  $N=\cup_{a\in A} \varphi_a(U_a)\subset M$ for the resulting neighborhood of $C\subset M$.

By % taking $\delta>0$ small and 
adjusting constants and refining the cover $\{(U_a, \varphi_a)\}_{a\in A}$ if necessary, we can and will assume that they satisfy the following
convenient conditions:
\begin{enumerate}
 \item $\varphi_a(U_a)\cap \varphi_b(U_b)  \not =  \emptyset$ implies $I_a \subset I_b$ or $I_b \subset I_a$.
\item $\varphi_a(U_a) \cap \{\ssf_i = 0\}  = \emptyset$
implies $\varphi_a(U_a) \cap \{\ssf_i \leq 2\delta\} = \emptyset$.
\end{enumerate}

For $\alpha\in \cI_a$, recall the function %and vector field
$$
\xymatrix{
h_\alpha:\R^{\cI_a}  \ar[r] &  \R % & v_\alpha \in \Vect(\R^{\cI_a})
}$$
appearing in the smoothing of Sect.~\ref{sect sm arb}.
Via 
 the inclusion and projection
$$
\xymatrix{
U_a \ar@{^(->}[r] & \R^{\cI_a} \times \R^k \ar[r] & \R^{\cI_a}
}
$$ 
and   diffeomorphism $\varphi_a$,
we can pull back and transfer $h_\alpha$ to a function
$$
\xymatrix{
h_{a, \alpha}:\varphi_a(U_a)\ar[r] & \R % & v_{a, \alpha} \in \Vect(\varphi_a(U_a)) 
}
$$
 
 Recall that for any non-root vertex $\alpha\in V(\cI_a)$ there is a unique parent vertex which we will denote here by
 $\hat \alpha_a\in  V(\cI_a)$ emphasizing its dependence on the poset $\cI_a$.

\begin{lemma}\label{lem: h functions}
For $a, b\in A$, with $\varphi_a(U_a)\cap \varphi_b(U_b) \not =  \emptyset$, suppose $I_b \subset I_a$. Then for any $\alpha\in I_b \subset I_a$,
%that is not a minimum in $\cI_b$, and hence not a minimum in $\cI_a$,
we have the equality of functions
$$
\xymatrix{
h_{a, \alpha} = h_{b,  \alpha}
}
$$
over the common domain $\varphi_a(U_a)\cap \varphi_b(U_b)$.
\end{lemma}

\begin{proof}
It suffices to assume $I_a = I_b \coprod \{c\}$ for some $c\in I$.  

Recall for $\alpha\in \cI_a$,  by definition for a root vertex $\rho\in V(\dirF)$, we have
$$
\xymatrix{
h_\rho = x_\rho:\R^{\dirF} \ar[r] & \R
}
$$
and for a non-root vertex $\alpha\in V(\dirF)$,  we inductively have
$$
\xymatrix{
h_\alpha :\R^{\dirF} \ar[r] & \R & h_\alpha = f(h_{\hat \alpha}, x_\alpha)
}
$$
where $\hat \alpha \in V(\dirF)$ is the  parent vertex of $\alpha$.

Thus it suffices to suppose $c = \hat\alpha_a $, or in other words, that $c$ is the parent of $\alpha$ inside of $\cI_a$.
Now we will
consider two cases: 

(i) $c$ is a minimum in $\cI_a$. Then it suffices to show 
\begin{equation}
\label{eqn: case 1}
\xymatrix{
h_{a, \alpha} = h_{a, \hat\alpha_a}
&
\mbox{ over }  \varphi_a(U_a)\cap \varphi_b(U_b)
}
\end{equation}

Recall that 
$\varphi_b(U_b) \cap \{\ssf_c = 0\}  = \emptyset$
implies $\varphi_b(U_b) \cap \{\ssf_c \leq 2\delta\} = \emptyset$. Thus by construction 
$h_{a, \alpha} = \ssf_c$ over $ \varphi_a(U_a)\cap \varphi_b(U_b)$ and so \eqref{eqn: case 1} holds.

(ii) $c$ is not a minimum in $\cI_a$.
Then it suffices to show 
\begin{equation}
\label{eqn: case 2}
\xymatrix{
h_{a, c} = h_{a, \hat c_a}
&
\mbox{ over }  \varphi_a(U_a)\cap \varphi_b(U_b)
}
\end{equation}

Recall that 
$\varphi_b(U_b) \cap \{\ssf_c = 0\}  = \emptyset$
implies $\varphi_b(U_b) \cap \{\ssf_c \leq 2\delta\} = \emptyset$. Thus by construction 
 \eqref{eqn: case 1} holds.

\end{proof}

Next, for $\alpha\in \cI_a$, recall the vector field
$$
\xymatrix{
v_\alpha = -b(h_{\hat\alpha}) c(x_\alpha) \partial_{x_\alpha} \in \Vect(\R^{\cI_a})
}$$
appearing in the smoothing of Sect.~\ref{sect sm arb}. It naturally lifts to a vector field on the product
$\R^{\cI_a} \times \R^{k_a}$, then via 
 the inclusion 
$$
\xymatrix{
U_a \ar@{^(->}[r] & \R^{\cI_a} \times \R^k 
}
$$ 
and   diffeomorphism $\varphi_a$,
we can restrict and transfer it  to a vector field
$$
\xymatrix{
& v_{a, \alpha} \in \Vect(\varphi_a(U_a)) 
}
$$

%the natural inclusion and projection
%$$
%\xymatrix{
%\R^{\cI_a} \ar@{^(->}[r] &  \R^{\cI_a} \times \R^k\ar[r] & \R^{\cI_a} 
%}
%$$ 

\begin{prop}\label{prop: shear}
For $a, b\in A$, with $\varphi_a(U_a)\cap \varphi_b(U_b) \not =  \emptyset$, suppose $I_b \subset I_a$. Then for any $\alpha\in I_b$,
%that is not a minimum in $\cI_b$, and hence not a minimum in $\cI_a$,
we have the equality of vector fields
$$
\xymatrix{
v_{a,\alpha} = v_{b,   \alpha} + w
}
$$
over the common domain $\varphi_a(U_a)\cap \varphi_b(U_b)$, where the vector field $w$, transported via $\varphi_b^{-1}$,
points along the second factor of the product
$\R^{\cI_b} \times \R^{k_b}$.

\end{prop}

\begin{proof}
By Lemma~\ref{lem: h functions}, the ambiguity under change of good charts of the vector field $v_\alpha=-b(h_{\hat\alpha}) c(x_\alpha) \partial_{x_\alpha}$ is the ambiguity of the coordinate vector field $\partial_{x_\alpha}$, and this is captured precisely by the shearing vector field $w$.
\end{proof}

\begin{remark}
Thanks to the axioms of a control system, we can additionally arrange so that the projection $\pi_{\ul \alpha}:T_{\ul \alpha}\to X_{\ul \alpha} $ is invariant with respect to $v_{a,\alpha}$ in the sense that $d\pi_{\ul \alpha}(v_{a,\alpha}) = 0$. This then in turn implies  for $\ul i\in \ul I$, with $\ul i\leq \ul \alpha$, that the projection $\pi_{\ul i}:T_{\ul i}\to X_{\ul i} $ 
is also invariant with respect to $v_{a,\alpha}$.
We also have for $\ul i\in \ul I$, with $ \ul \alpha$ and $ \ul i$ incomparable, that the vector field $v_{a,\alpha}$ vanishes near $H_{\ul i}$.
\end{remark}

Next fix a partition of unity $\{\delta_a\}_{a\in A}$ subordinate to the open cover $\{(U_a, \varphi_a)\}_{a\in A}$.

For any  $a\in A$ and $i\in I$ with $i\not \in \cI_a$, set 
$$
v_{a, i} = 0\in \Vect(\varphi_a(U_a))
$$

For each $i\in I$, introduce the global vector field 
$$
\xymatrix{
v_i = \sum_{a\in A} \delta_a v_{a, i}\in \Vect(M)
}
$$

For each $i\in I$, define the homeomorphism 
$$
\xymatrix{
\Phi_i:M\ar[r]^-\sim &  M
}
$$ 
to be the unit-time flow of the vector field $v_i$.

\begin{remark}
Note that we have arranged so that for $i\in I$, $\ul j\in \ul I$, with $\ul i \not < \ul j$, the 
projection $\pi_{\ul j}:T_{\ul j}\to X_{\ul j} $ is invariant with respect to each $v_{a, i}$, thus also 
 with respect to
 $v_i$, and thus finally  with respect to $\Phi_i$.
\end{remark}

Fix a total order on $I$ compatible with its natural partial order.  Write $i_0, i_1, \ldots, i_{N} \in I$ for the ordered elements. 
Define the composite homeomorphism
$$
\xymatrix{
\Phi = \Phi_{i_0} \circ\Phi_{i_1} \circ \cdots \circ \Phi_{i_N}:M\ar[r]^-\sim & M
}
$$ 

\begin{corollary}\label{cor: global smoothing}
For any $a\in A$, under the good chart $\varphi_a$, the homeomorphism $\Phi$ takes the form
$F_{\cI_a}\times \tilde \psi$.
\end{corollary}

\begin{proof}
Immediate from Prop.~\ref{prop: shear}.
\end{proof}
%
%%%%%%%%%%%%%%%%%%%%%%%%%%%%%%%%%%%%%%%%%%%%%%%%%%%%%%%%
%
%
%\subsubsection{Smoothed cylinder}

For each $i\in I$, introduce the inverse homeomorphism 
$$
\xymatrix{
\Psi_i = \Phi_i^{-1}:M\ar[r]^-\sim &  M
}
$$ 
Introduce the smoothing homeomorphism 
$$
\xymatrix{
\Psi =  \Psi_{i_N} \circ \cdots \circ \Psi_{i_1} \circ \Psi_{i_0}:M\ar[r]^-\sim & M
}
$$

\begin{defn}
Define the {\em directed cylinder} $\sC \subset M$ to be the image  of the total cylinder
$$
\xymatrix{
\sC =\Psi(C)
}
$$
\end{defn}

\begin{thm}\label{thm: dir cyl}
The directed cylinder $\sC \subset M$ is a hypersurface in good position with a canonical coorientation
and smoothed arboreal hypersurface singularities.
\end{thm}

\begin{proof}  
Immediate from Thm.~\ref{thm: local smoothing}, Prop.~\ref{prop: cyl sings}, and Cor.~\ref{cor: global smoothing}.
\end{proof}

We will write $\Lambda_\sC\subset S^*M$ for the positive coray bundle of the directed cylinder  $\sC \subset M$.

%%%%%%%%%%%%%%%%%%%%%%%%%%%%%%%%%%%%%%%%%%%%%%%%%%%%%%%
%%%%%%%%%%%%%%%%%%%%%%%%%%%%%%%%%%%%%%%%%%%%%%%%%%%%%%%
%%%%%%%%%%%%%%%%%%%%%%%%%%%%%%%%%%%%%%%%%%%%%%%%%%%%%%%

\subsection{Expanded hypersurface}
We continue with the constructions of the preceding sections,  arriving in this section at our goal.
Now taking into account the positive coray bundle $\Lambda\subset S^*M$, we cut out an expanded hypersurface $E \subset M$  inside the total cylinder
$C\subset M$,
and a  directed
expansion  $\sE \subset M$
inside the directed cylinder $\sC \subset M$.

%%%%%%%%%%%%%%%%%%%%%%%%%%%%%%%%%%%%%%%%%%%%%%%%%%%%%%%

\subsubsection{Conormal sections}

Recall for $\ul i \in \ul I$, and each $i \in F_{\ul i} \subset I$, we have the subspace 
$$
\xymatrix{
\Lambda_{i}   \subset S^*M |_{H_{\ul i}}
}
$$
 The Whitney conditions imply the subspace lies in the spherically projectivized conormal bundle
$$
\xymatrix{
\Lambda_i \subset  S_{H_{\ul i}}^*M
}
$$

Furthermore, the projection $S^*M\to M$ restricts to a diffeomorphism
$$
\xymatrix{
\Lambda_i \ar[r]^-\sim & H_{\ul i}
}$$ 
and thus the subspace $\Lambda_i \subset  S_{H_{\ul i}}^*M$ is  the image of a unique section
$$
\xymatrix{
\lambda_i:H_{\ul i}\ar[r] & S^*_{H_{\ul i}} M 
}
$$
% of the spherical projectivization
%of the conormal bundle $T^*_{H_{\ul i}} M \subset T^* M$.

Note that the conormal bundle $T^*_{H_{\ul i}} M \to H_{\ul i} $ is canonically isomorphic to the dual of the normal bundle
$E_{\ul i} \to H_{\ul i}$.
Hence for any inner product on the normal bundle $E_{\ul i}\to H_{\ul i}$, the section $\lambda_i$ naturally determines
a unit-length section
$$
\xymatrix{
\lambda_i:H_{\ul i} \ar[r] & E^*_{\ul i}
}
$$
Thus via the structures of the tubular neighborhood $(T_{\ul i}, \rho_{\ul i}, \pi_{\ul i})$, the section $\lambda_i$ naturally determines 
a fiber-wise linear function
$$
\xymatrix{
\lambda_i:T_{\ul i}\ar[r] & \R
}
$$

%%%%%%%%%%%%%%%%%%%%%%%%%%%%%%%%%%%%%%%%%%%%%%%%%

\subsubsection{Expanded strata}

Recall that to construct the total cylinder, we fixed a small positive radius $r_i\in \R_{>0}$,  for each $i\in I$, 
so that $r_i \not = r_{i'}$ whenever $\ul i = \ul i'$.

Now in addition, choose a small positive displacement $d_i\in \R_{>0}$,  for each $i\in I$,
and a small value $s_i\in \R$,  for each $i\in I$. 

%In this section, we will work with 
%$d_i, r_i  \in (0,1)$, chosen independently, 
%and then  $s_i\in \R$ chosen sufficiently small with respect to $d_i, r_i$.
%Eventually in  subsequent sections, we will  choose the constants sufficiently small  in the following order.
%First, we will follow the poset structure on $I$ and work from the minima to the maxima. 
%Second, for each $i\in I$, we will choose the small constants
%in the listed order 
%$
% d_i \in \R_{>0}$, $ r_i \in \R_{>0}$,  $s_i\in \R$.

\begin{defn}
For each $i\in I$, introduce the fiber-wise affine functions  
$$
\xymatrix{
 \ssg_i: T_{\ul i} \ar[r] & \R & \ssg_i(x) =  \lambda_{i}(x) + r_id_i -s_i
}
$$
\end{defn}

\begin{defn}
For each $i\in I$, define
the {\em expanded stratum} $E_i \subset T_{\ul i}$ to be the subspace 
of $x\in T_{\ul i}$ cut out by the equations
$$
\xymatrix{
\ssf_i(x) = 0 & \ssg_i(x) \geq 0
&
\ssf_a(x) \geq  0, \mbox{ whenever }  a<i \mbox{ and } x\in T_{\ul i} \cap T_{\ul a}
}
$$

%Define
%the {\em horizontal  boundary} $\partial_h E_i \subset E_i \subset T_{\ul i}$ to be the subspace 
%of $x\in T_{\ul i}$ cut out by the equations
%$$
%\xymatrix{
%\ssf_i(x) = 0 & \ssg_i(x) = 0
%&
%\ssf_a(x) \geq  0, \mbox{ whenever }  a<i \mbox{ and } x\in T_{\ul i} \cap T_{\ul a}
%}
%$$

\end{defn}

\begin{remark}
Recall  the truncated cylinder $C_i\subset T_{\ul i}$
introduced in the previous section.
Putting together the definitions,  the expanded stratum  $E_i \subset T_{\ul i}$ is the subspace 
of $x\in C_{i}$ cut out by the equation
$
\ssg_i(x) \geq 0.
$
\end{remark}

\begin{lemma}
Fix any  $d_i , r_i\in (0, 1)$, and  then sufficiently small $s_i\in \R$.

(1) The expanded stratum $E_i\subset T_{\ul i}$ is a closed submanifold with corners.

(2) The projection $\pi_{\ul i}$ exhibits $E_i$ as a closed $(\on{codim}_M H_{\ul i}-1)$-ball bundle over $H_{\ul i}$.

%(3) The projection $\pi_{\ul i}$ exhibits $\partial_h  E_i$ as a closed $(\on{codim}_M H_{\ul i}-2)$-sphere bundle over $H_{\ul i}$.
\end{lemma}

\begin{proof}
For the moment, set $s_i = 0$, so that $\ssg_i(x) =  \lambda_{i}(x) + r_id_i$.
%Choose any small $d_i\in \R_{>0}$, in particular with $d_i <1/2$. 
Observe that for $d_i, r_i \in (0,1)$,  the pair  $\{\ssf_i, \ssg_i\}$ of functions is multi-transverse at their total zero value
$0\in \R^2$
and the restriction $\ssg_i|_{\{\ssf_i = 0\}}$ takes both positive and negative values.
Choosing small $s_i \in \R$, so that  $\ssg_i(x) =  \lambda_{i}(x) + r_id_i -s_i$,
the above facts continue to hold.
Now the assertions follow from Lemmas~\ref{lemma: truncated strata} and \ref{lemma: truncated cyls}.
\end{proof}

%%%%%%%%%%%%%%%%%%%%%%%%%%%%%%%%%%%%%%%%%%%%%%%%%%%%%%%

\subsubsection{Total expansion}

Recall that our constructions depend on constants $d_i \in \R_{>0}$, $r_i \in \R_{>0}$, $s_i \in \R$,
for $i\in I$.
In what follows, we will always choose them in the following order.
First, we will  independently choose $d_i \in (0,1)$, for each $i\in I$. 
Second, we will follow the poset structure on $I$, working from the minima to the maxima, and choose small
$r_i \in \R_{>0}$, for each $i\in I$.
Finally, we will again follow the poset structure on $I$, working from the minima to the maxima, and choose small
$s_i \in \R$, for each $i\in I$.
We will refer to such sufficiently small choices of constants as {\em sequentially small}.

Recall that the set $\{\ssf_i\}_{i\in I}$ of functions  is multi-transverse at its total zero value $0\in \R^I$.
Recall the role of the constants $d_i \in \R_{>0}$, $r_i \in \R_{>0}$, $s_i \in \R$,  for $i\in I$, in the definition of the functions $\ssg_i(x) =  \lambda_{i}(x) + r_id_i -s_i$. In particular, since we select  the values $s_i \in \R$, for $i\in I$, after the others, we may select sequentially small constants such that
the extended set $\{\ssf_i\}_{i\in I} \coprod \{\ssg_i\}_{i\in I}$
of functions   is multi-transverse at  its total zero value $(0, 0) \in \R^I \times \R^I$.

\begin{defn}
Define the {\em  total expansion} $E \subset M$ to be the union of expanded strata
$$
\xymatrix{
E = \bigcup_{i\in I} E_i
}
$$
\end{defn}

\begin{prop}\label{prop: gen rect hyp}
There exist sequentially small constants 
$
 d_i \in \R_{>0}$, $ r_i \in \R_{>0}$,  $s_i\in \R$,
 for $i\in I$,
such that
 the singularities of the total expansion $E \subset M$ are generalized rectilinear arboreal
hypersurface singularities.
\end{prop}

\begin{proof}
Let us first appeal to Prop.~\ref{prop: cyl sings}.

Fix a point $p\in M$. 

Let $I_p \subset I$ comprise indices $i\in I$ such that $p\in C_i \subset T_{\ul i}$, so  in particular
$\ssf_i(p) = 0$. We will regard $I_p \subset I$ as a poset with the induced partial order: $i, j \in I_p$
satisfy $i<j$ inside of $I_p$
if and only if $i<j$ inside of $I$.

Recall that Prop.~\ref{prop: cyl sings} established
that $I_p$ is the poset of a rooted forest $\cI_p$,
and thus
the singularity of the total cylinder $C \subset M$ at the point $p$ is
the rectilinear arboreal hypersurface $H_{\cI_p}$.
 More precisely, there is an open ball 
$U\subset  \R^{\cI_p} \times \R^{k_p}$,
with $k_p = \dim M - |I_p|$ and $0\in U$, and
  a smooth open embedding 
 $$
 \xymatrix{
 \varphi:U \ar[r]^-\sim &  \varphi(U) \subset M
 }$$ 
 such that the following holds
 $$
 \xymatrix{
 \varphi(0) = p &
 \varphi(U \cap (H_{\cI_p} \times \R^{k_p})) = \varphi(U) \cap C 
 }$$
 $$
 \xymatrix{
x_i =   \ssf_i \circ \varphi: U \ar[r] & \R &
\mbox{ for all }  i\in I_p
 }
 $$
% Then for $i\in \cI_p$, the union and equations of Remark~\ref{rmk: less redundant} in terms of the coordinates $x_i$ immediately match those of Definitions~\ref{def: truncated cyls} in terms of the functions $\ssf_i$. 
%

Now let $J_p \subset I_p$ comprise indices $i\in I_p$ such that $p\in E_i\subset C_i$, so additionally $\ssg_i(p) \geq 0$. We will regard $J_p \subset I_p$ as a poset with the induced partial order: $i, j \in J_p$
satisfy $i<j$ inside of $J_p$
if and only if $i<j$ inside of $I_p$. It will follow from the discussion below that  at most
$J_p$ results from deleting from $I_p$ some of its leaf vertices.

Let $\ell_p \subset J_p$ comprise indices 
$i\in J_p$ so that $\ssg_i(p) = 0$.
It will follow from the discussion below that  $\ell_p$ is a subset of the leaf  vertices of $J_p$.

To see the poset $J_p$ (if nonempty), together with the marked vertices $\ell_p$, arise from a leafy rooted forest
$\cJ_p^* = (\cJ_p, \ell_p)$,
%Note that if $i\in I_p\setminus J_p$, then $p\in C^\trunc_i \setminus E_i$, and so $\ssg_i(p)<0$.
it suffices to establish the claim: for sequentially small constants, if $\ssg_i(p) \leq  0$,  for some $i\in I_p$, then $i$ is a leaf vertex of $I_p$.
If the claim holds, then the above embedding $\varphi$ will identify the singularity of the  total expansion $E \subset M$ at the point $p$ with
the  rectilinear arboreal hypersurface $H_{\cJ^*_p}$.

To prove the claim, we will appeal to the following.

\begin{lemma}
For any $d_i \in (0, 1)$, sufficiently small $r_i \in \R_{>0}$,  further sufficiently small $s_i \in \R$,
and any $a\in I$ with $a>i$,
 the restriction of $\ssg_i:T_{\ul i} \to \R$ to the intersection $H_{\ul a} \cap C_i\subset T_{\ul i}$ is strictly positive.
\end{lemma}

\begin{proof}
Recall that $\ssg_i(x) =  \lambda_{i}(x) + r_id_i -s_i$. Thus it suffices to  prove the assertion with $s_i = 0$.
%so that $\ssg_i(x) =  \lambda_{i}(x) + r_id_i$. 

Fix $d_i \in (0, 1)$. Suppose there is a sequence of radii $r_i(n)\in \R_{>0}$, with $r_i(n) \to 0$, with corresponding truncated cylinder $C_i(n) \subset  T_{\ul i}$, and points $x(n) \in H_{\ul a} \cap C_i(n)$, with $x(n) \to x \in H_{\ul i}$, such that $\ssg_i(x(n)) \leq 0$. Then  it is a simple calculation to check with respect to any local coordinates that a subsequence of the secant lines $[x(n), x]$  converges to a line not contained in $\ker (\lambda_i) \subset T_x M$. But this  contradicts Whitney's condition $B$
for the pair of strata $H_{\ul i}\subset \ol{ H}_{\ul a}$.
\end{proof}

Returning to the claim, for any $i \in I$, we can invoke the lemma to choose a small radius $r_i \in \R_{>0}$
to be sure 
that
 the restriction of $\ssg_i:T_{\ul i} \to \R$ to the intersection $H_{\ul a} \cap C_i\subset T_{\ul i}$ is strictly positive,
 for all $a\in I$ with $a>i$. Then later in our sequence of choices of constants,  for each  $a\in I$ with $a>i$, we can choose  a small radius $r_a\in \R_{>0}$, so that $C_a\subset T_{\ul i}$
is  as close as we like to $H_{\ul a}$, hence ensuring that  
 the restriction of $\ssg_i:T_{\ul i} \to \R$ to the intersection $C_a \cap C_i\subset T_{\ul i}$ is strictly positive.
Thus if $i \in I_p$ is not a leaf vertex, so there is $a\in I_p$ with $a>i$, we must have $\ssg_i(p) >  0$.

Thus the claim holds and this completes the proof of the proposition.
\end{proof}

%%%%%%%%%%%%%%%%%%%%%%%%%%%%%%%%%%%%%%%%%%%%%%%%%%%%%%%

\subsubsection{Smoothed total expansion}

Recall the smoothing homeomorphism 
$$
\xymatrix{
\Psi :M\ar[r]^-\sim & M
}
$$

\begin{defn}
Define the {\em directed expansion} $\sE \subset M$ to be the image  of the total expansion
$$
\xymatrix{
\sE =\Psi(E)
}
$$
\end{defn}

\begin{thm}\label{thm dir exp is arb}
The directed expansion $\sE \subset M$ is a hypersurface in good position with a canonical coorientation
and generalized smooth arboreal hypersurface singularities.
\end{thm}

\begin{proof}
Immediate from Thm.~\ref{thm: local smoothing}, Cor.~\ref{cor: global smoothing}, and Prop.~\ref{prop: gen rect hyp}.
\end{proof}

We will write $\Lambda_\sE\subset S^*M$ for the positive coray bundle of the directed expansion  $\sE \subset M$.

%%%%%%%%%%%%%%%%%%%%%%%%%%%%%%%%%%%%%%%%%%%%%%%%%%%%%%%
%%%%%%%%%%%%%%%%%%%%%%%%%%%%%%%%%%%%%%%%%%%%%%%%%%%%%%%
%%%%%%%%%%%%%%%%%%%%%%%%%%%%%%%%%%%%%%%%%%%%%%%%%%%%%%%
%%%%%%%%%%%%%%%%%%%%%%%%%%%%%%%%%%%%%%%%%%%%%%%%%%%%%%%
%%%%%%%%%%%%%%%%%%%%%%%%%%%%%%%%%%%%%%%%%%%%%%%%%%%%%%%
%%%%%%%%%%%%%%%%%%%%%%%%%%%%%%%%%%%%%%%%%%%%%%%%%%%%%%%

\section{Invariance of sheaves}\label{s inv}

%%%%%%%%%%%%%%%%%%%%%%%%%%%%%%%%%%%%%%%%%%%%%%%%%%%%%%%

Fix once and for all a field $k$ of characteristic zero.

Let $M$ be a manifold with spherically projective cotangent bundle $\pi: S^*M\to M$.

\subsection{Singular support}

For the material reviewed here, the standard reference is~\cite{KS}.

\subsubsection{Basic notions}
Let $\Sh(M)$ denote the dg category of complexes of sheaves of $k$-vector spaces on $M$
such that each object is
constructible with respect to some Whitney stratification. (This choice of definition has the pitfall that finite collections of Whitney stratifications
do not necessarily  admit a common refinement, but we will always work with specific Whitney stratifications and never come near this danger.)
We will  abuse terminology and refer to objects of $\Sh(M)$ as sheaves on $M$.
%Note that $\Sh(M)$ is the global sections of a sheaf $\Sh$ of dg categories that assigns to an open $U\subset M$ the dg category $\Sh(U)$.

%Our main object of study will be the sheaf $\Sh_{\Lambda}$ of microlocal sheaves along $\Lambda$. Of the many ways to think about $\Sh_{\Lambda}$, let us review a particularly concrete approach.

To any object $\cF\in \Sh(M)$, one can associate its singular support $\ssupp(\cF) \subset S^*M$. This is a closed  Legendrian recording those codirections in which the propagation of sections of $\cF$ is obstructed. 
Its behavior under standard functors is well understood including its behavior under Verdier duality $\ssupp(\D_M(\cF) ) = -\ssupp(\cF)$.
One has the vanishing $\ssupp(\cF) = \emptyset$ if and only if the cohomology sheaves of $\cF$ are locally constant. 
We will abuse terminology and refer to such objects of $Sh(M)$ as local systems on $M$.

\begin{example}
To fix conventions, suppose $i:U\to M$ is the inclusion of an open submanifold whose closure is a submanifold with boundary modeled on a Euclidean halfspace. Then the singular support  $\Lambda_U = \ssupp(i_! k_U) \subset S^*M$ of the extension by zero $i_! k_U\in \Sh(M)$ %of the constant sheaf on $U$ 
consists 
of the spherical projectivization of the outward conormal codirection along the boundary $\partial U \subset M$.
If near a point $p\in \partial U$, we have $U = \{x < 0\}$, for 
a local coordinate $x$, then $\Lambda_U|_p = \ssupp(i_! k_U)|_p$ is the spherical projectivization of the ray
$\R_{\geq 0} \langle dx\rangle $. 

More generally, suppose $i:U\to M$ is the inclusion of an open submanifold whose closure is a submanifold with corners
modeled  on a Euclidean quadrant. 
Then the singular support  $\Lambda_U = \ssupp(i_! k_U) \subset S^*M$ consists of the spherical projectivization of the outward conormal cone along the boundary $\partial U \subset M$. If near a point $p\in \partial U$, we have $U = \{x_1, \ldots, x_k < 0\}$, for 
local coordinates $x_1, \ldots, x_k$, then $\Lambda_U|_p =  \ssupp(i_! k_U)|_p$
 is the spherical projectivization of the cone
$\R_{\geq 0} \langle dx_1, \ldots, dx_k\rangle$. 

\end{example}

Fix a closed Legendrian $\Lambda\subset S^*M$.
For example, given $\cS = \{X_\alpha\}_{\alpha\in A}$ a Whitney stratification of $M$, one could take the union
of the spherically projectivized conormals to the strata
$$\xymatrix{
\Lambda_{\cS} = \bigcup_{\alpha\in A} S^*_{X_{\alpha}} M \subset S^*M
}$$  
In general, given any closed Legendrian $\Lambda\subset S^*M$, we will always assume $M$ admits a Whitney stratification $\cS$
such that $\Lambda\subset \Lambda_{\cS}$.

Let $\Sh_{\Lambda}(M) \subset \Sh(M)$ denote the full dg subcategory of objects with singular support lying in $\Lambda\subset S^*M$. 
For example,   for $\cS$ a Whitney stratification,  $\Sh_{\Lambda_\cS}(M) \subset \Sh(M)$ consists precisely of 
$\cS$-constructible sheaves.
In general,  
 if $\Lambda\subset\Lambda_{\cS}$, then  objects of  $\Sh_{\Lambda}(M) \subset \Sh(M)$ are in particular $\cS$-constructible,
 while possibly satisfying further constraints.
%Note that  $\Sh_{\Lambda}(M) \subset \Sh(M)$ is the global sections of a subsheaf $\Sh_{\Lambda} \subset \Sh$  of dg categories that assigns to an open $U\subset M$ the full dg subcategory 
%of $\Sh(U)$ of those objects with singular support lying in $\Lambda|_U$. 

\subsubsection{Non-characteristic isotopies}

Let us recall a key property of singular support. Suppose $\Lambda_1, \Lambda_2\subset S^*M$ are 
closed Legendrians, and $\psi_t: M \to M$ is an isotopy such that $\psi_t(\Lambda_1) \cap \Lambda_2 = \emptyset$, for all $t$. Then for any $\cF_1\in \Sh_{\Lambda_1}(M)$, $\cF_2\in \Sh_{\Lambda_2}(M)$, the complex
$\Hom_{\Sh(M)}(\psi_t(\cF_1), \cF_2)$ is locally independent of $t$ in the sense that it forms a local system on the space of 
parameters $t$. 
%Such families of calculations are often called non-characteristic with respect to the given data.

For a basic example of this, recall 
that given an open subset  $i:U\to M$, there  is a functorial identification 
$$
\xymatrix{
\Gamma(U, \cF) \simeq \Hom_{\Sh(M)}( i_{!} k_U, \cF)
}$$
Suppose  $\psi_t: M \to M$ is an isotopy, and $i_t:U_t\to M$ 
 is family of open submanifolds with boundary  
 given by the isotopy $U_t = \psi_t(U_0)$. Let $\Lambda\subset S^*M$ be a closed Legendrian disjoint from the outward conormal direction $\Lambda_{U_t} \subset S^*M$ along the boundary
 $\partial U_t\subset M$, for all $t$. Then for any $\cF\in \Sh_\Lambda(M)$, the sections
 $$
 \xymatrix{
 \Gamma(U_t, \cF) \simeq \Hom_{\Sh(M)}( i_{t!} k_{U_t}, \cF)
 }
 $$  are locally independent of $t$.
 Similarly, for the closed complement $j_t:Y_t = M\setminus U_t\to M$,  the sections 
 $$
 \xymatrix{
 \Gamma_{Y_t}(M, \cF) \simeq\Gamma(Y_t, j_t^!\cF) \simeq \Cone(\Gamma(M, \cF) \to \Gamma(U_t, \cF))[-1]
 }
 $$
 are locally independent of $t$.

For a specific instance of this, suppose 
$\Lambda\subset S^*M$ is a closed Legendrian, and $f:M\to N$ is a proper fibration that is $\Lambda$-non-characteristic 
in the sense that the spherical projectivization of $\im(df^*) \subset T^*M$ is disjoint from $\Lambda$.
Then for any $\cF\in \Sh_{\Lambda}(M)$, the pushforward $f_*\cF \in \Sh(N)$ is a local system.
This can be put into the above setup by recalling  for $U\subset N$ an open subset with inverse image $i:f^{-1}(U) \to M$, the
 functorial identifications 
$$
\xymatrix{
\Gamma(U, f_*\cF) \simeq \Gamma(f^{-1}(U), \cF) \simeq   \Hom_{\Sh(M)}(i_!k_{f^{-1}(U)}, \cF)
}
$$ 
%When $f$ is $\Lambda$-non-characteristic, 
%the outward conormal direction $\Lambda_{f^{-1}(U)} \subset S^*M$ is disjoint from $\Lambda$.

%%%%%%%%%%%%%%%%%%%%%%%%%%%%%%%%%%%%%%%%%%%%%%%%%%%%%%%

\subsection{Projections and orthogonality}

Let $H\subset M$ be a directed hypersurface with positive coray bundle $\Lambda\subset S^*M$.
Fix a Whitney stratification of $H\subset M$ satisfying the setup of Sect.~\ref{ss: strat} and fix a
 compatible  system of control data.

In this section, we will focus on a single closed stratum and its tubular neighborhood, and thus break from our usual notational conventions to reduce clutter.

\subsubsection{Microlocal projections}

Let $i_Y:Y\to H$ be the inclusion of a closed stratum 
with tubular neighborhood $T \subset M$,  tubular distance function  $\rho: T\to \R$ and
 tubular  projection $\pi:T\to Y$. Let $j_Y:T' = T\setminus Y\to T$ be the inclusion of the open complement.
 In what follows, we can take $M=T$.

Recall there are finitely many codirections $\lambda_i:Y\to  S^*_Y M$, for $i = 1, \ldots, k$, as well as
 disjoint union decompositions
$$
\xymatrix{
\Lambda|_Y =  \coprod_{i=1}^k \lambda_i(Y) 
&
\Lambda|_T = \coprod_{i=1}^k \Lambda_i
}
$$
such that $\Lambda_i|_Y = \lambda_i(Y)$. The front projection of $\Lambda_i\subset S^*T$
is itself a directed hypersurface $H_i \subset T$ with positive coray bundle $\Lambda_i\subset S^*T$. 

We have the evident fully faithful inclusions $\Sh_{\Lambda_i}(T) \subset \Sh_\Lambda(T)$. 
%(Here and throughout  what follows, we write $\Sh_{\Lambda}(T)$ in place of $\Sh_{\Lambda|_T}(T)$
%and use similar shorthand to reduce clutter.)
In the other direction,
microlocal cut-offs provide canonical  functors 
$$
\xymatrix{
\fP_i:\Sh_{\Lambda}(T) \ar[r] & \Sh_{\Lambda_i} (T)
}
$$
equipped with natural transformations
$$
\xymatrix{
p_i:\cF\ar[r] &  \fP_i(\cF)
&
 \cF\in \Sh_{\Lambda}(T)
 }
 $$
 Taking the direct sum, we obtain a natural transformation
 $$
\xymatrix{
\oplus_{i=1}^k p_i:\cF\ar[r] &  \oplus_{i = 1}^k \fP_i(\cF)
&
\cF\in \Sh_{\Lambda}(T)
}
$$

The cone $\cL = \Cone(\oplus_{i=1}^k p_i)$ has no singular support so is a local system. We have a functorial presentation of $\cF \in \Sh_\Lambda(B)$ itself as a cone
$$
\xymatrix{
\cF \simeq \Cone( \oplus_{i = 1}^k  \fP_i(\cF) \ar[r] & \cL)
}
$$

\subsubsection{Single codirection}

Now suppose further that $\Lambda|_Y = \lambda(Y)$ for a single codirection $\lambda:Y\to S^*_Y M$. Then for $\cF \in \Sh_\Lambda(T)$, we have two canonical morphisms 
$$
\xymatrix{
\gamma:\pi^*\pi_*\cF  \ar[r] & \cF
&
\gamma_c:\cF   \ar[r] &  \pi^!\pi_!\cF
}
$$
Observe that $\pi_*\cF$, $\pi_!\cF$ are local systems, so $\pi^*\pi_*\cF$, $ \pi^!\pi_!\cF$ are local systems, since  
 $\pi$ is non-characteristic with respect to $\cS$, and hence with respect to $\Lambda$
 since $\Lambda\subset \Lambda_\cS$.

Introduce the  full subcategories
$$
\xymatrix{
\Sh_{\Lambda}(T)_*^0\subset \Sh_{\Lambda}(T)& 
\Sh_{\Lambda}(T)_!^0\subset \Sh_{\Lambda}(T)
}
$$
of $\cF\in \Sh_{\Lambda}(T)$
with $\pi_*\cF \simeq 0$ %(equivalently, $i_Y^*\cF \simeq 0$, )
respectively $\pi_!\cF \simeq 0$.
% equivalently, $i^!_Y\cF \simeq 0$.
Observe that $\cF\in \Sh_{\Lambda}(T)_*^0$ respectively  $\cF\in \Sh_{\Lambda}(T)_!^0$ if and only if $i_Y^*\cF \simeq 0$
respectively 
$i_Y^!\cF \simeq 0$, or in turn,
if and only if the canonical map $j_{Y!}j_Y^!\cF \to \cF$ respectively  $ \cF\to  j_{Y*}j_Y^*\cF$ is an isomorphism.
Verdier duality restricts to an equivalence 
$$
\xymatrix{
\D_B:(\Sh_{\Lambda}(T)_*^0)^{op} \ar[r]^-\sim & \Sh_{-\Lambda}(T)_!^0
}
$$

The  cones $\cF_*^0 = \Cone(\gamma),  \cF^0_! = \Cone(\gamma_c)$ satisfy the vanishing 
$\pi_*\cF_*^0 \simeq 0, \pi_!cF_!^0 \simeq 0$ or in other words lie in the full subcategories
$$
\xymatrix{
\cF_*^0
\in \Sh_{\Lambda}(T)_*^0 & \cF^0_! \in \Sh_{\Lambda}(T)^0_!
}
$$
There are  functorial presentations of $\cF \in \Sh_\Lambda(T)$ itself as a cone
$$
\xymatrix{
\cF\simeq  \Cone(\cF_*^0[-1] \ar[r] & \pi^*\pi_*\cF) 
&
\cF\simeq  \Cone( \pi^!\pi_!\cF\ar[r] & \cF^0_!)[-1]
}
$$

Continuing with $\Lambda|_Y = \lambda(Y)$ for a single codirection $\lambda:Y\to S^*_Y M$, choose any smooth path $\ell:\R\to T$ so that $\ell(0) \in Y$ is the only intersection of $\ell(\R)$ with $H$, and also $\lambda(\ell'(0)) >0$.  Then for any $\cF \in \Sh_\Lambda(T)$, the pullbacks
$
\ell^*(\cF), \ell^!(\cF)  \in \Sh(\R)
$
are constructible with respect to $\{0\}$, $\R\setminus \{0\}$.
Furthermore, the singular support conditions imply the following are  local systems
$$
\xymatrix{
\ell^*(\cF)|_{\R_{\geq 0}} \in \Loc(\R_{\geq 0})
&
\ell^!(\cF)|_{\R_{\leq 0}}  \in \Loc(\R_{\leq 0})
}
 $$
Thus in particular 
for 
$\cF^0_*
\in \Sh_{\Lambda}(T)_*^0$,  $\cF^0_! \in \Sh_{\Lambda}(T)^0_!$,
the vanishings $i_Y^*(\cF^0_*) \simeq 0$,  $i_Y^!(\cF^0_!) \simeq 0$
respectively imply the  vanishings
\begin{equation}
\label{eqn vanish}
\xymatrix{
 \ell^*(\cF^0_*)|_{\R_{> 0}}  \simeq 0
&
 \ell^!(\cF^0_!)|_{\R_{< 0}}}\simeq 0
\end{equation}
Informally speaking, if we think of $\lambda$ as pointing ``up" along $Y$, then $\cF^0_*$ vanishes ``above" $H$,
and $\cF^0_!$ vanishes ``below"   $H$.

\subsubsection{Orthogonality of codirections}

Now let us return to the possibility that $\Lambda|_Y$ has more than one codirection and focus on the interaction
of two distinct codirections $\lambda_1, \lambda_2:Y\to \Lambda|_Y$ with $\Lambda_1 =\lambda_1(Y), \Lambda_2 = \lambda_2(Y)$.

\begin{lemma}\label{lem orthog of codirs}
For any $\cF_1\in  \Sh_{\Lambda_1}(T)_*^0$,
$\cF_2\in  \Sh_{\Lambda_2}(T)^0_*$,
we have $\Hom_{\Sh(T)} (\cF_1, \cF_2) \simeq 0$.

For any $\cF_1\in  \Sh_{\Lambda_1}(T)^0_!$,
$\cF_2\in  \Sh_{\Lambda_2}(T)^0_!$,
we have $\Hom_{\Sh(T)} (\cF_1, \cF_2) \simeq 0$.

\end{lemma}

\begin{proof}
The second statement follows from the first by duality.

To prove the first, we will move $\cF_2$ through a non-characteristic isotopy to a position
where it is evident that $\Hom_{\Sh(T)} (\cF_1, \cF_2) \simeq 0$.  

Note that it suffices to prove the assertion locally in $Y$. Thus we may fix a smooth identification $T\simeq \R^{k  +\ell+1}$,
$Y\simeq \R^k \times \{0\}$ such that $\pi:T\to Y$ is the standard projection $\R^{k + \ell+1} \to \R^k$.
%
% and $\rho:T\to \R$ is the pullback under the standard projection $\R^{k +\ell+1} \to \R^{\ell+1}$ of the  standard Euclidean inner product on $\R^{\ell+1}$.
Moreover, for each $i = 1, 2$, we can arrange that $\Lambda_i|_Y \simeq \R^k \times \{\lambda_i\} 
\subset S^*_Y T \simeq \R^k \times S^{\ell}$,
and that $\Lambda_i \subset S^*T \simeq  \R^{k + \ell+1} \times S^{k+\ell}$ lies within a small neighborhood of $\R^{k + \ell +1} \times \{\lambda_i\}$.

\medskip

{\em (Step 1)} If $\lambda_2 = -\lambda_1$, then proceed to {\em (Step 2)} below. 
Else $\lambda_1, \lambda_2$ are linearly independent so span a two-dimensional plane $P \subset \R^{\ell+1}\subset \R^{k+\ell+1}$. For $\theta\in [0,1]$, let $R_\theta:\R^{k+\ell+1} \to \R^{k+\ell+1}$ be the orthogonal rotation  of $P$ fixing $P^\perp$, such that $R_0 = \id$, $R_1(\lambda_2) = -\lambda_1$,
and $R_\theta(\lambda_2)$,  for $\theta\in [0,1]$,  traverses the  short arc of directions in $P$  from $\lambda_2$ to $-\lambda_1$ (so not passing through $\lambda_1$).

Viewing $R_\theta:\R^{k+\ell+1}\to \R^{k+\ell+1}$ as an isotopy, observe that it satisfies $R_\theta(\Lambda_2) \cap \Lambda_1 = \emptyset$, for $\theta\in [0,1]$.
Thus $\Hom_{\Sh(B)} (\cF_1, R_{\theta*}(\cF_2))$ is independent of $\theta\in [0,1]$.

\medskip

{\em (Step 2)} By {\em (Step 1)}, we may assume   $\lambda_2 = -\lambda_1$. Without loss of generality, we may further assume $\lambda_1 = dy_{0}$
so $\lambda_2 = -dy_{0}$.
For $t\in \R$, let $T_t:\R^{k+\ell+1} \to \R^{k+\ell+1}$ be the translation $T_t(x_1, \ldots, x_k, y_0, y_1, \ldots, y_\ell) = (x_1, \ldots, x_k, y_0+t, y_1,  \ldots, y_\ell)$.
Viewing  $T_t:\R^{k+\ell+1} \to \R^{k+\ell+1}$ as an isotopy, observe that it satisfies $T_t(\Lambda_2) \cap \Lambda_1 = \emptyset$, for $t\in \R$.
Thus $\Hom_{\Sh(T)} (\cF_1, T_{t*}(\cF_2))$ is independent of $t\in \R$.

Finally, for $t\gg 0$,  the vanishing~\eqref{eqn vanish}  implies the supports of $\cF_1, T_{t*}(\cF_2)$ are disjoint. 
Hence $\Hom_{\Sh(T)} (\cF_1, T_{t*}(\cF_2)) \simeq 0$ and we are done.
\end{proof}

%%%%%%%%%%%%%%%%%%%%%%%%%%%%%%%%%%%%%%%%%%%%%%%%%%%%%%%

\subsection{Specialization of sheaves}

Let $X \subset M$ be a closed subspace with Whitney stratification
 $\cS=\{X_\alpha\}_{\alpha \in A}$. 
Fix a compatible  system of control data $\{(T_\alpha, \rho_\alpha, \pi_\alpha)\}_{\alpha\in A}$.

Fix a small $\epsilon>0$. For each $\alpha\in A$,  recall the mapping
$\Pi_\alpha:M \to M
$
and the almost retraction
$$
\xymatrix{
r:M \ar[r] & M &
r = \Pi_{\alpha_0} \Pi_{\alpha_1} \cdots  \Pi_{\alpha_N}
}
$$
where $N+1=|A|$ and the indices $\alpha_i \in A$ can be arbitrarily ordered.

We will record some of its simple properties; we leave the details of the proofs to the reader.

\begin{lemma}\label{lem spec of loc sys}
For each $\alpha\in A$, pushforward along  $\Pi_\alpha:M\to M$ 
 is  canonically equivalent to the identity  when restricted to local systems
$$
\xymatrix{
\Pi_{\alpha*} \simeq\id:\Loc(M) \ar[r]^-\sim & \Loc(M)
}
$$

More generally, it is  canonically equivalent to the identity  when restricted to $\cS$-constructible sheaves
$$
\xymatrix{
\Pi_{\alpha*} \simeq \id:\Sh_\cS(M) \ar[r]^-\sim & \Sh_\cS(M)
}
$$

The same assertions hold for pushforward along  $r:M\to M$. 
\end{lemma}

\begin{proof}
We leave the assertions for $\Pi_\alpha$ to the reader.
Since
$r = \Pi_{\alpha_0} \Pi_{\alpha_1} \cdots  \Pi_{\alpha_N},
$
  the assertions for $\Pi_{\alpha}$ imply them for $r$.
%This follows form the fact that $\Pi_\alpha$ is proper with contractible fibers, and homotopic to the identity.
\end{proof}

\begin{lemma}\label{lem inverse to restrict}
Let $X_0\subset X$ be a closed stratum with tubular neighborhood $T_0\subset M$. 

%Consider the open inclusions $M\setminus X_0 \subset M\setminus T_0[\leq \epsilon] \subset M$, and let $\cS$
%denote the induced stratification of each.

Restriction of $\cS$-constructible sheaves is an equivalence
$$
\xymatrix{
\Sh_\cS(M \setminus X_0) \ar[r]^-\sim & \Sh_\cS(M \setminus T_0[\leq \epsilon])
}
$$
with an inverse provided by
the pushforward  
$$
\xymatrix{
\Pi_{0*}:\Sh_\cS(M \setminus T_0[\leq \epsilon]) \ar[r] & \Sh_\cS(M \setminus X_0)
}
$$ 

Suppose in addition $X$ is a directed  hypersurface with positive coray bundle $\Lambda$.
Then 
restriction of sheaves is an equivalence
$$
\xymatrix{
\Sh_\Lambda(M \setminus X_0) \ar[r]^-\sim & \Sh_\Lambda(M \setminus T_0[\leq \epsilon])
}
$$
with an inverse provided by
the pushforward  
$$
\xymatrix{
\Pi_{0*}:\Sh_\Lambda(M \setminus T_0[\leq \epsilon]) \ar[r] & \Sh_\Lambda(M \setminus X_0)
}
$$ 

\end{lemma}

\begin{proof}
For the first assertion, the mapping $\Pi_{0}:M\setminus T_0[\leq \epsilon] \to M\setminus X_0$ is a stratum-preserving homeomorphism
and the identity on $M\setminus T_0[\leq 2\epsilon]$.

For the second assertion, thanks to the first, it suffices to show $\Pi_{0*}$ does not introduce any spurious singular support
outside of $\Lambda$.
More generally, it suffices to show the following.  Let $p\in X$ be a point in a closed stratum $X_0\subset X$, and $B(p)\subset M$ a small open ball  around $p$.
Let
$q\in B(p)$ be another point in the same stratum $X_0 \subset X$,  and $B(q)\subset B(p)$ a small open ball around $q$.
Then  it suffices to show for any $\cF\in \Sh_\cS(M)$, if $\ssupp(\cF)|_{B(q)} \subset \Lambda$, then
$\ssupp(\cF)|_{B(p)} \subset \Lambda$. 

The assertion is local and we may assume $M= T_0 = \BR^{k+\ell +1}$, $X_0 = \BR^k$,
and the projection $\pi_0:T_0\to X_0$ is the standard projection $\BR^{k+\ell +1} \to \BR^k$. 
%For $q\in \BR^k$, let $ \BR^{\ell+1}_q = \pi_0^{-1}(q)$ denote the slice. It suffices to bound $\ssupp(\cF| \BR^{\ell+1}_q)$.

Suppose some $\xi \in T^*_p  \BR^{k+\ell+1} \simeq \BR^{k+\ell+1}$ represents a point of $\ssupp(\cF)$ but not a point of $ \Lambda$. Since $\cF$ is $\cS$-constructible, we have $\xi \in T^*_{\BR^{k}} \BR^{k+\ell+1}|_p \simeq \BR^{\ell+1} $. Consider the corresponding linear function $\xi:\BR^{k+\ell+1}\to \BR$. 
Fix a small $\epsilon<0$, and consider the inclusion $i:\{\xi \leq  \epsilon\} \to \BR^{k+\ell+1}$.
Then it suffices to see that $\pi_{0*} i^*\cF$ is locally constant on $\BR^k$, since then its vanishing at some $q$ will imply its vanishing at $p$.
But since $X$ is in good position, and $\xi$ does not 
represent a point of $ \Lambda$,
the map $\pi_0 \times \xi$ is non-characteristic near the value $\xi = \epsilon$,
and the assertion follows.
\end{proof}

\begin{lemma}\label{lem inductive constructibility}

Let $X_0\subset X$ be a closed stratum with tubular neighborhood $T_0\subset M$. 

Introduce the mapping 
$$
\xymatrix{
r':M\setminus X_0 \ar[r] & M\setminus X_0 &
r '= \Pi_{\alpha_1} \Pi_{\alpha_2} \cdots  \Pi_{\alpha_N}
}
$$
where  $\Pi_0$ is omitted from the composition.

For $\cF\in \Sh(M)$, suppose $r'_*(\cF|_{M\setminus X_0}) \in \Sh(M\setminus X_0)$ is $\cS$-constructible.

 Then there is a  functorial equivalence
 $$
\xymatrix{
 r_*(\cF) |_{M\setminus X_0} \simeq 
 r'_*(\cF|_{M\setminus X_0})
}
 $$
\end{lemma}

\begin{proof}
By Lemma~\ref{lem inverse to restrict}, we have
$$
\xymatrix{
r'_*(\cF|_{M\setminus X_0})\simeq  \Pi_{0*}( r'_*(\cF|_{M\setminus X_0}) |_{M\setminus T_0[\leq \epsilon]})  
}
$$

By construction, we also have
$$
\xymatrix{
 \Pi_{0*}( r'_*(\cF|_{M\setminus X_0}) |_{M\setminus T_0[\leq \epsilon]})  \simeq
 \Pi_{0*}( r'_*(\cF|_{M\setminus X_0})) |_{M\setminus X_0}   \simeq 
 r_*(\cF) |_{M\setminus X_0}
}
$$
\end{proof}

%%%%%%%%%%%%%%%%%%%%%%%%%%%%%%%%%%%%%%%%%%%%%%%%%%%%%%%

\subsection{Singular support of specialization}

Let $H\subset M$ be a directed hypersurface with positive coray bundle $\Lambda\subset S^*M$.
Fix a  Whitney stratification $\cS = \{H_{\ul i}\}_{\ul i \in \ul I}$ satisfying the setup of Sect.~\ref{ss: strat} and fix a
 compatible  system of control data $\{(T_{\ul i},\rho_{\ul i}, \pi_{\ul i}) \}_{\ul i \in \ul I}$.

Fix a small $\epsilon>0$. For each $\ul i\in \ul I$,  recall the mapping
$\Pi_{\ul i} \to M
$
and the almost retraction
$$
\xymatrix{
r:M \ar[r] & M &
r = \Pi_{\ul i_0} \Pi_{\ul i_1} \cdots  \Pi_{\ul i_{\ul N}}
}
$$
where $\ul N+1=|\ul I|$ and the indices %$\ul i_k \in \ul I$ 
can be arbitrarily ordered.

Fix sequentially small parameters  $d_i>0$, $r_i>0$, $s_i >0$, for $i\in I$. 
Fix a small smoothing constant $\delta>0$.
Recall the  directed cylinder  $\sC\subset M$ with  positive coray bundle $\Lambda_\sC\subset S^*M$,
and  the directed expansion  $\sE\subset M$ with   positive coray bundle $\Lambda_\sE\subset S^*M$.

The main goal of this section is Theorem~\ref{thm constructible} below that states
that pushforward along the almost retraction $r:M\to M$ induces a functor
$$
\xymatrix{
r_*:\Sh_{\Lambda_\sE}(M) \ar[r] & \Sh_{\Lambda}(M)
}
$$
In other words, pushforward takes sheaves with singular support in $\Lambda_\sE$ to sheaves
with singular support in $\Lambda$.

\subsubsection{Interaction with directed cylinder}
We will arrive at our main goal after the following  coarser estimate. 
%Recall that $\sH \subset \sC$ and $\Lambda_\sH\subset \Lambda_{\sC}$.

\begin{prop} \label{prop cylindrical constructibile}
Pushforward along the almost retraction $r:M\to M$ induces a functor
$$
\xymatrix{
r_*:\Sh_{\Lambda_\sC}(M) \ar[r] & \Sh_{\Lambda_\cS}(M)
}
$$
In other words, pushforward takes sheaves with singular support in $\Lambda_\sC$ to sheaves
 with singular support in $\Lambda_\cS$, or in other words, to $\cS$-constructible sheaves.

\end{prop}

\begin{proof}

By induction on the number of strata of $H\subset M$. 

The base case $H=\emptyset$ is immediate: $r$ is the identity map of $M$.

Suppose given a closed stratum $i_0:H_{0}\to M$. 

Set $M[>\epsilon] = M\setminus T_{ 0}[\leq \epsilon]$, $H[>\epsilon]= H\cap  M[>\epsilon] $.
%and $\cS[>\epsilon]= \cS \cap M[>\epsilon]$.
The Whitney stratification,  system of control data and family of lines for $H\subset M$ immediately provide the same for
$H[>\epsilon]\subset M[>\epsilon] $.
Denote the induced Whitney stratification by  $\cS[>\epsilon]= \cS \cap M[>\epsilon]$,
and the resulting almost retraction by $r[>\epsilon] = r|_{M[>\epsilon]}$.
Furthermore, starting with these data, the expansion algorithm yields a directed cylinder $\sC[>\epsilon] \subset 
M[>\epsilon]$ with positive coray bundle $\Lambda_{\sC[>\epsilon]} \subset S^*M[>\epsilon]$ such that 
$\sC[>\epsilon]  = \sC\cap M[>\epsilon] $, $\Lambda_{\sC[>\epsilon]} = \Lambda_\sC|_{M[>\epsilon]}$.

By induction, 
since $H[>\epsilon]$ has fewer strata than $H$,
pushforward  induces a functor 
$$
\xymatrix{
r[>\epsilon]_*:\Sh_{\Lambda_{\sC[>\epsilon]} }(M[>\epsilon]) \ar[r] & \Sh_{\Lambda_{\cS[>\epsilon]}} (M[>\epsilon])
}
$$

For any $\cF \in \Sh_{\Lambda_\sC}(M)$,
by construction, we have
$$
\xymatrix{
r_*(\cF)|_{M\setminus H_0} \simeq \Pi_{0*} r[>\epsilon]_*(\cF|_{M[>\epsilon]})
}
$$
Thus Lemma~\ref{lem inductive constructibility} implies that $r_*(\cF)|_{M\setminus H_0}$ is $\cS$-constructible.

Now it suffices to show $i_0^! r_* \cF$ is a local system.
By base change, we have 
$$\xymatrix{
i_0^! r_* \cF \simeq r_*  j[\leq \epsilon]^! \cF
}
$$
where $j[\leq \epsilon]:T_0[\leq \epsilon] \to T_0 \to  M$ is the inclusion.

Recall that the almost retraction $r$ is independent of the ordering of the indices so that in particular $r = \Pi_{{\ul i_1}} \cdots \Pi_{{\ul i_N}} \Pi_0$.
Recall that $\Pi_{0}|_{T_0[\leq \epsilon]} = \pi_0$, and that $\Pi_{\ul i }|_{H_0} = \id$ for all $\ul i\not = 0$.
%hence
%$\Pi_{0*} j[\leq \epsilon]^! \cF \simeq \pi_{0*} j[\leq \epsilon]^!\cF$,
Thus we have
$$
\xymatrix{
r_*  j[\leq \epsilon]^! \cF \simeq
 \Pi_{\ul i_1*} \cdots \Pi_{\ul i_{N}*} \Pi_{0*}  j[\leq \epsilon]^! \cF \simeq 
 \Pi_{\ul i_1*} \cdots \Pi_{\ul i_{N}*}
 \pi_{0*} j[\leq \epsilon]^!  \cF
\simeq
 \pi_{0*} j[\leq \epsilon]^!  \cF% \simeq
% r_* \tilde \imath^! \Psi^*\cF
}
$$

Recall the smoothing homeomorphism
$$
\xymatrix{
\Psi =  \Psi_{i_N} \circ \cdots \circ \Psi_{i_1} \circ \Psi_{0}:M\ar[r]^-\sim & M
}
$$
where $N + 1 = |I|$, and the elements $0, i_1, \ldots, i_{N} \in I$ are ordered compatibly with the partial order on $I$. 
Recall that for $j\in I$ with image $\ul j \in \ul I$, the restriction $\Psi_j|_{T_0}$ 
satisfies $ \Psi_j|_{T_0} = \id$, 
when $\ul j$ is incomparable to $0$,
and 
 $\pi_0 \Psi_j|_{T_0} = \pi_0$, 
when $\ul j$ is greater than or equal to $0$.
Thus altogether $\pi_0\Psi|_{T_0} = \pi_0$.

For any $i\in I$, the projection $\pi_0$ is non-characteristic with respect to the truncated cylinder $C_i \subset M$,
and hence to the total cylinder $C \subset M$. Thanks to the identity $\pi_0\Psi|_{T_0} = \pi_0$, %  properties of the restrictions  $\Psi_j|_{T_0}$, for $j\in I$, 
the projection  $\pi_0$ is also non-characteristic with respect to 
the smoothing $\sC_i =\Psi(C_i)$, and hence  to the directed cylinder
$ \sC = \Psi(C)$, and in particular  to its positive coray bundle $\Lambda_\sC$.
Finally, since $\pi_0$ is a proper fibration and non-characteristic with respect to $\Lambda_{\sC}$, and hence
with respect to $\ssupp(\cF) \subset \Lambda_{\sC}$,
we conclude  that $\pi_{0*}j[\leq \epsilon]^!   \cF$ is a local system on 
$H_0$.
 \end{proof}

\subsubsection{Interaction with directed expansion}

 Now we will prove the main assertion of this section.

\begin{thm}\label{thm constructible}
Pushforward along the almost retraction $r:M\to M$ induces a functor
$$
\xymatrix{
r_*:\Sh_{\Lambda_{\sE}}(M) \ar[r] & \Sh_{\Lambda}(M)
}
$$
In other words, pushforward takes sheaves with singular support in $\Lambda_{\sE}$ to sheaves
with singular support in $\Lambda$.
\end{thm}

\begin{proof}

By induction on the number of strata of $H\subset M$. 

The base case $H=\emptyset$ is immediate: $r$ is the identity map of $M$.

Suppose given a closed stratum $i_0:H_{0}\to M$.

Set $M[>\epsilon] = M\setminus T_{ 0}[\leq \epsilon]$, $H[>\epsilon]= H\cap  M[>\epsilon] $,
$\Lambda[>\epsilon] = \Lambda|_{ M[>\epsilon] }$.
%and $\cS[>\epsilon]= \cS \cap M[>\epsilon]$.
The Whitney stratification,  system of control data and family of lines for $H\subset M$ immediately provide the same for
$H[>\epsilon]\subset M[>\epsilon] $.
Denote %the induced Whitney stratification by  $\cS[>\epsilon]= \cS \cap M[>\epsilon]$,
the resulting almost retraction by $r[>\epsilon]:M[>\epsilon] \to M[>\epsilon]$ and note that  $r[>\epsilon] = r|_{M[>\epsilon]}$.
Additionally, starting with these data, our constructions give a directed expansion $\sE[>\epsilon] \subset 
M[>\epsilon]$ with positive coray bundle $\Lambda_{\sE[>\epsilon]} \subset S^*M[>\epsilon]$ such that 
$\sE[>\epsilon]  = \sE\cap M[>\epsilon] $, $\Lambda_{\sE[>\epsilon]} = \Lambda_\sE|_{M[>\epsilon]}$.

By induction, 
since $H[>\epsilon]$ has fewer strata than $H$,
pushforward  induces a functor 
$$
\xymatrix{
r[>\epsilon]_*:\Sh_{\Lambda_{\sE[>\epsilon]} }(M[>\epsilon]) \ar[r] & \Sh_{\Lambda[>\epsilon]} (M[>\epsilon])
}
$$
By the second assertion of Lemma~\ref{lem inverse to restrict}, for $\cF\in \Sh_{\Lambda_\sE}(M)$, we then have
$$
\xymatrix{
(r_*\cF)|_{M \setminus H_0} \in \Sh_{\Lambda}(M\setminus H_0)
}
$$

Therefore, for $\cF \in \Sh_{\Lambda_\sE}(M)$, it only remains to show 
$$
\xymatrix{
\ssupp(r_*\cF)|_{H_{ 0}} \subset \Lambda|_{H_{ 0}}
}
$$

By construction, we have an inclusion of directed hypersurfaces $\sE \subset \sC$ and positive
coray bundles $\Lambda_\sE\subset \Lambda_{\sC}$. Thus we have  
$\Sh_{\Lambda_\sE}(M) \subset \Sh_{\Lambda_\sC}(M)$,
hence 
thanks to Prop.~\ref{prop cylindrical constructibile}, we have $r_*\cF\in \Sh_\cS(M)$,
so in particular 
$$
\xymatrix{
\ssupp(r_*\cF)|_{H_{ 0}} \subset S^*_{H_{ 0}} M
}
$$
Hence for each $x\in H_{ 0}$, we may restrict to the normal slice
$$
\xymatrix{
  \pi_{ 0}^{-1}(x) \subset T_0
}$$

Without loss of generality, we may assume $ \pi_{ 0}^{-1}(x) = \R^n$, $x = 0$, and $\rho_0|_{ \pi_{ 0}^{-1}(x)}$ is
the standard Euclidean inner product. The positive corays $\Lambda|_x = \{\lambda_1, \ldots, \lambda_k\}$ are represented by pairing with nonzero vectors $v_1, \ldots, v_k \in \R^n$.

For $\xi\not \in \Lambda|_x$,
we seek a small $\delta>0$ so that
$$
\xymatrix{
\Gamma_{\xi\geq 0}(B(\delta), r_*\cF)\simeq 0
}
$$
where $B(\delta)\subset \R^n$ is the open ball of radius $\delta>0$ around $0$.

For any $t\in [0, 2\epsilon]$, introduce the subspace
$B(t, 2\epsilon)^-\subset B(2\epsilon)$,  where $t\leq \rho \leq 2\epsilon$, $\xi<0$,
so in particular $B(0, 2\epsilon)^- = B(2\epsilon)^-$.
Unpacking the constructions, we 
we seek that
$$
\xymatrix{
H^*(B(2\epsilon), B(\epsilon, 2\epsilon)^-; \cF)\simeq 0
}
$$
where $\epsilon>0$ is the original constant selected once and for all.

We will proceed by induction on the finite set $\Lambda|_x$.

The arguments in the base case, when $\Lambda|_x =\{\xi_1\}$ is a single codirection, and in the general inductive step will be similar.  We will show by a series of non-characteristic moves that we can change the subspace $B(\epsilon, 2\epsilon)^-$ of the pair
 to be the entire ambient space $B(2\epsilon)$.

\medskip

{\em (Step 1)} The natural map is an isomorphism
$$
\xymatrix{
H^*(B(2\epsilon), B(\epsilon, 2\epsilon)^-; \cF)\ar[r]^-\sim &
H^*(B(2\epsilon), B(2\epsilon)^-; \cF)
}
$$
since the isotopy of pairs
$$
\xymatrix{
(B(2\epsilon), B(t, 2\epsilon)^-)
& 
t\in [0, \epsilon]
}
$$
is non-characteristic with respect to $\Lambda_\sE$.

\medskip

{\em (Step 2)} 
Suppose $\Lambda|_x =\{\lambda_1, \ldots, \lambda_k\}$ with corresponding radii constants $r_1 < \cdots<r_k < \epsilon$.

For any $t\in [0, 2\epsilon]$, introduce the subspace
$U(t, 2\epsilon)^-\subset B(2\epsilon)$,  where either $\rho\leq t$ or $t\leq \rho \leq 2\epsilon$, $\xi<0$, so
the union 
$$\xymatrix{
U(t, 2\epsilon)^- = B(t)\cup B(t, 2\epsilon)^-
}
$$ In particular, we have $U(0, 2\epsilon)^- = B(2\epsilon)^-$
and $U(2\epsilon, 2\epsilon)^- = B(2\epsilon)$.

Fix $r$ such that $r_{k-1} < r< r_k$ (when $k=1$, fix  $r$ such that $0 < r< r_k$).
We claim the natural map is an isomorphism
$$
\xymatrix{
H^*(B(2\epsilon), B(\epsilon, 2\epsilon)^-; \cF)\ar[r]^-\sim &
H^*(B(2\epsilon), U(r
, 2\epsilon)^-; \cF)
}
$$

If $k=1$, then the assertion is clear since 
the intersection of the isotopy of pairs
$$
\xymatrix{
(B(2\epsilon), U(t, 2\epsilon)^-)
& 
t\in [0, r]
}
$$
is non-characteristic with respect to $\Lambda_\sE$, since in fact it has constant intersection
with $\sE$.

If $k>1$, then 
the claim follows by induction: in the locus $r< \rho \leq 2\epsilon $, the pairs are precisely the same, and in the locus 
$ \rho\leq r$, the pairs and $\Lambda_\sE$ are precisely what one encounters for 
$\Lambda|_x =\{\xi_1, \ldots, \xi_{k-1}\}$.

\medskip

{\em (Step 3)} 
Continuing with the notation of {\em (Step 2)}, recall that $\xi_ 0 \not = \lambda_k$. For $\theta\in [0, 1]$, let $\xi_\theta$ be the  short arc (not passing through $\lambda_k$) of the great circle of codirections passing through $\xi_0 = \xi$, $\xi_1 = -\lambda_k$.

For any $\theta \in [0, 1]$, introduce the subspace
$U(t, 2\epsilon)_\theta^-\subset B(2\epsilon)$,  where either $\rho\leq t$ or $t\leq \rho \leq 2\epsilon$, $\xi_\theta<0$, so
in particular 
$
U(t, 2\epsilon)_0^- = U(t, 2\epsilon)^-.
$

We claim that the isotopy of pairs
$$
\xymatrix{
(B(2\epsilon), U(r, 2\epsilon)_\theta^-)
& 
\theta\in [0, 1]
}
$$
is non-characteristic with respect to $\Lambda_\sE$. 

First, note that we can excise the locus $\rho < r$.
The only remaining issue is when $\xi_\theta$ passes through or near to some $\lambda_i$, for $i< k$. But here we have two transverse functions $\rho$ and $\theta$ and $\Lambda_\sE$ is the disjoint union of something $\rho$-characteristic and something $\theta$-characteristic. In general, such a situation is non-characteristic.

\medskip

{\em (Step 4)} 
Finally, observe that any linear function lifting $\xi_1 = -\lambda_k$ is non-characteristic with respect to $\Lambda_\sE$
on the pair 
$$
\xymatrix{
(B(2\epsilon), U(r, 2\epsilon)_1^-)
& 
\theta\in [0, 1]
}
$$
Thus the relative cohomology vanishes
$$
\xymatrix{
H^*(B(2\epsilon), U(r, 2\epsilon)_1^-; \cF)\simeq 0
}
$$
This completes the proof of the theorem.
\end{proof}

%%%%%%%%%%%%%%%%%%%%%%%%%%%%%%%%%%%%%%%%%%%%%%%%%%%%

\subsection{Inverse functor}

Our aim here is to show that the functor 
$$
\xymatrix{
r_*:\Sh_{\Lambda_\sE}(M) \ar[r] & \Sh_{\Lambda}(M)
}
$$
of Theorem~\ref{thm constructible} is in fact an equivalence.

We will first establish an inductive version of the assertion.
Suppose given a closed stratum $i_0:H_{0}\to M$. 
Suppose $\Lambda|_{H_0} = \lambda({H_0})$ for a single codirection $\lambda:{H_0}\to S^*_{H_0} M$. 

Recall the constant $r_0>0$, indexed by the minimum $0\in I$, appearing in the construction of the directed expansion $\sE\subset M$
with positive coray bundle $\Lambda_\sE\subset S^*M$.

Set $M[>r_0] = M \setminus T_0[\leq r_0]$, $H[>r_0] = H \cap M[>r_0] $, $\Lambda[>r_0] = \Lambda_{M[>r_0]}$.
The Whitney stratification,  system of control data and family of lines for $H\subset M$ immediately provide the same for
$H[>r_0]\subset M[>r_0] $.
Denote %the induced Whitney stratification by  $\cS[>\epsilon]= \cS \cap M[>\epsilon]$,
the resulting almost retraction by $r[>r_0]:M[>r_0] \to M[>r_0]$.
The expansion data for  $H\subset M$ immediately restricts to expansion data for
 $H[>r_0]\subset M[>r_0] $.
Denote the resulting directed expansion by $\sE[>r_0]  = \sE \cap M[>r_0] $
with positive coray bundle 
$\Lambda[>r_0] = \Lambda|_{M[>r_0] }$.
Note as well that 
$$
\xymatrix{
\sE \cap T_0[\leq r_0] = \sE \cap S_0[r_0] =  E_0
}
$$
where $E_0 \subset T_0$ is the expanded stratum of $H_0$.
% cut out by the equations
%$$
%\xymatrix{
%\ssf_0(x) = 0 & \ssg_0(x)\geq 0
%}
%$$

\begin{prop}\label{prop single codir}
Let $H_0\subset H$ be a closed stratum.
%Fix a small open ball $B\subset M$ around a point $x\in H$ and introduce the punctured ball
%$B' = B\setminus\{x\}$.

Suppose $\Lambda|_{H_0} = \lambda({H_0})$ for a single codirection $\lambda:{H_0}\to S^*_{H_0} M$. 

Suppose the  functor
$$
\xymatrix{
r[>r_0]_*:\Sh_{\Lambda_{\sE[>r_0]}}(M[>r_0]) \ar[r] & \Sh_{\Lambda[> r_0]}(M[>r_0])
}
$$
is an equivalence. 

Then the functor 
$$
\xymatrix{
r_*:\Sh_{\Lambda_\sE}(M) \ar[r]^-\sim & \Sh_{\Lambda}(M)
}
$$
is also an equivalence.
\end{prop}

\begin{proof}
 We will construct an explicit inverse functor denoted by
 $$
\xymatrix{
s: \Sh_{\Lambda}(M)\ar[r] & \Sh_{\Lambda_\sE}(M) 
}
$$
By our hypotheses, there exists an inverse
$$
\xymatrix{
s[>r_0]: \Sh_{\Lambda[>r_0]}(M[>r_0])\ar[r]^-\sim & \Sh_{\Lambda_{\sE[>r_0]}}(M[>r_0]) 
}
$$

Observe that it suffices to prove the assertion for $M =T_0$. 
%Set $T'_0 = T_0 \setminus H_0$. Let $i_0:H_0 \to T_0$, $j_0: T'_0 \to T_0$ denote the inclusions.

 Recall $\Sh_{\Lambda}(T_0)^0_!\subset \Sh_{\Lambda}(T_0)$ denotes the full subcategory
of $\cF \in \Sh_{\Lambda}(T_0)$ with $ \pi_{0!}\cF\simeq 0$,
or equivalently $i_0^!\cF \simeq 0$, or again equivalently, the canonical map 
$
\cF\to j_{0*}j_0^* \cF
$
is an isomorphism.
More generally, recall  for  $\cF \in \Sh_\Lambda(T_0)$, the functorial presentation 
$$
\xymatrix{
\cF\simeq  \Cone(\pi_{0}^!\pi_{0!}\cF\ar[r] & \cF^0_!)[-1]
}
$$
where $\cF^0_!\in \Sh_{\Lambda}(T_0)^0_!$ is the cone of of the canonical morphism
$ \cF\to  \pi_{0}^!\pi_{0!}\cF$. Note as well that
$\pi_{0}^!\pi_{0!}\cF \in \Loc(T_0)$.

By Lemma~\ref{lem spec of loc sys}, on the full subcategory $\Loc(T_0) \subset \Sh_{\Lambda}(T_0)$, we may set
the inverse  $s$ to be the identity.
Now we will  construct the inverse on the full subcategory
 $\Sh_{\Lambda}(T_0)^0_!\subset \Sh_{\Lambda}(T_0)$ as a composition of several functors.
To do so, let us walk back through some steps in  the construction of  the directed hypersurface $\sE$
with positive coray bundle $\Lambda_\sE$.

 Introduce 
the open inclusion 
$$
\xymatrix{
j[>r_0]:T_0[>r_0] \ar@{^(->}[r] &  T_0
}
$$
%
%Outside of $T_0[\leq r_0]$, we have the agreement of total expansions
%$$
%\xymatrix{
%E \cap T_0[>r_0] = (H')^\expand \cap T_0[>r_0]
%}
%$$
%
%The smoothing homeomorphisms $\Psi, \Psi'$ preserve each of the  subspaces 
%$$
%\xymatrix{
% T_0[>r_0], T_0[\leq r_0] \subset T_0
% }
% $$ 
% and are the identity when restricted to $ T_0[\leq r_0]$.
%%Similarly,   the smoothing homeomorphisms $\Psi'$ preserves each of the complementary subspaces 
%%$$
%%\xymatrix{
%% T'_0(> r_0) = T_0(>r_0), T_0'(\leq r_0) = T_0(\leq r_0)\setminus H_0 \subset T'_0 = T_0 \setminus H_0
%% }
%% $$ 
%% and is the identity when restricted to $ T'_0(\leq r_0)$.
%%  
%  In particular, we have identities of functors
%  $$
% \xymatrix{
%  \Psi_*J_{*} \simeq J_{*}\Psi_*&
%  \Psi_*J^{*} \simeq J^{*}\Psi_*&
%  \Psi'^*J^{*} \simeq J^{*}\Psi'^* &
%    \Psi'^*J_{*} \simeq J_{*}\Psi'^*
% } $$
% 
 
%Away from the open annulus $T_0[> r_0] \cap T_0[< r_0+ \delta]$, the individual smoothing homeomorphism  $\Psi_0$
%is the identity. Thus
%  $\Psi = \Psi'\circ\Psi_0$ and $\Psi'$ coincide over its complement.
%Furthermore, over the open annulus $T_0[> r_0] \cap T_0[< r_0+ \delta]$, by construction $\rho_0 \circ \Psi_0 = \rho_0$.
%

For $\cF\in \Sh_{\Lambda}(T)^0_!$, define the candidate inverse to be the functorial composition
  $$
\xymatrix{
s(\cF) = j[>r_0]_{*} s[>r_0](\cF|_{T_0[>r_0]}) \in  \Sh(T_0) 
} $$

 \begin{claim} 
 For $\cF\in \Sh_{\Lambda}(T_0)_!^0$, we have $\ssupp(s(\cF)) \in \Lambda_\sE$. 
\end{claim}

\begin{proof}
Over $T_0[>r_0]$, the assertion is evident  by construction.

Over $T_0[<  r_0]$, we have $s(\cF) = 0$ by definition.

Along $S_0[r_0]$, we have 
$$
\xymatrix{
\ssupp(s(\cF)) \subset \R_{\geq 0}\langle  d\rho_0\rangle
}
$$ 
thanks to the fact that 
%$\Lambda^\expand$ is closed with
$$
\xymatrix{
\Lambda_\sE|_{S_0(r)} \subset  \R_{\geq 0}\langle  d\rho_0\rangle
}
$$
and the behavior of  singular support under $*$-pushforwards~\cite{SV}.

Thus it remains to see that 
$$
\xymatrix{
\ssupp(s(\cF))|_{S[r_0] \setminus E_0} = \emptyset
}
$$
In fact, we have that 
$$
\xymatrix{
s(\cF)|_{S[r_0] \setminus E_0} = 0
}
$$
To see this, recall that
 $\cF\in \Sh_{\Lambda}(T_0)^0_!$ implies  by equation \eqref{eqn vanish} that $\cF$ vanishes on the $(\lambda\ll 0)$-component of the complement $T_0 \setminus H$. Thus $s[>r_0](\cF|_{T_0[>r_0]})$ vanishes
 on the $(\lambda\ll 0)$-component of the complement $T_0[>r_0] \setminus \sE[>r_0]$.
 Since $S[r_0] \setminus E_0$ is in the closure of this component, we obtain the asserted vanishing.
\end{proof}

The claim confirms we have a well-defined functor $s:\Sh_{\Lambda}(T_0)_!^0 \to \Sh_{\Lambda_\sE}(T_0)$.

 \begin{claim} $r_*\circ s \simeq \id$.
 \end{claim}
 
 \begin{proof}
Recall that $\epsilon > r_0$,
and that $r^{-1}(T_0[>0]) = T_0[>\epsilon]$.

Thus by induction, we have a canonical isomorphism
  $$
\xymatrix{
(r_*s(\cF))|_{T[>0]}  \simeq \cF|_{T[>0]}
} $$

Recall that $\cF\in \Sh_{\Lambda}(T_0)^0_!$ implies the canonical map $ \cF\to  j_{0*}j_0^* \cF$
is an isomorphism. Therefore it suffices to show that $r_*s(\cF) \in \Sh_{\Lambda}(T_0)^0_!$,
or in other words that $i_0^!r_*s(\cF) \simeq 0$.

Working locally in $H_0$, by base change, it suffices to show that
$$
\xymatrix{
\Gamma_{T_0[\leq \epsilon]}(T_0, s(\cF)) \simeq  0
}
$$
Unwinding the definitions, we seek
$$
\xymatrix{
\Gamma_{T_0[>r_0] \cap T_0[\leq \epsilon]}( T_0, s[>r_0](\cF|_{T_0[>r_0]})) \simeq 0   
}
$$ 

We have seen that $\ssupp(s[r_0>0](\cF|_{T_0[>r_0]})) \subset \Lambda_{\sE[>r_0]}$. Since $d \rho_0$ is disjoint from  $\Lambda_{\sE[>r_0]}$,  the above relative cohomology vanishes.
\end{proof}

 \begin{claim} 
 For $\cL\in \Loc(T_0)$, $\cF\in \Sh_{\Lambda}(T_0)_!^0$, we have canonically
 $$
 \xymatrix{
 \Hom_{\Sh(T_0)}(s(\cF) , \cL) \simeq 0
 &
 \Hom_{\Sh(T_0)}(\cL, \cF) \simeq 
  \Hom_{\Sh(T_0)}(\cL, s(\cF)) 
}
 $$ 
 \end{claim}

\begin{proof}
We may work locally in $H_0$, so in particular may assume $\cL$ is constant.

For the first assertion, by duality, it suffices to show 
$$
\xymatrix{
\Gamma(T_0, \D(s(\cF)))\simeq 0
}
$$
Unwinding the definitions, we seek
$$
\xymatrix{
\Gamma( T_0, j[>r_0]_!\D(s[>r_0](\cF|_{T_0[>r_0]}))) \simeq 0   
}
$$ 
We have seen that $\ssupp(s[>r_0](\cF|_{T_0[>r_0]})) \subset \Lambda_{\sE[>r_0]}$. Since $d \rho_0$ is disjoint from  $-\Lambda_{\sE[>r_0]}$,  the above relative cohomology vanishes.

For the second assertion, 
 it suffices to show 
$$
\xymatrix{
\Gamma(T_0, \cF) \simeq \Gamma(T_0, s(\cF))
}
$$
But by the previous claim, we have  
$$
\xymatrix{
\Gamma(T_0, \cF) \simeq  \Gamma(T_0, r_*s(\cF)) \simeq \Gamma(T_0, s(\cF)) 
}
$$
\end{proof}

The claim confirms the functor extends $s:\Sh_{\Lambda}(T_0) \to \Sh_{\Lambda_\sE}(T_0)$.

 \begin{claim} $s \circ r_*  \simeq \id$.
 \end{claim}

 \begin{proof}
Thanks to what we have seen, it suffices to check the assertion on 
 the full subcategory $\Sh_{\Lambda_\sE}(T_0)^0_! \subset \Sh_{\Lambda_\sE}(T_0)$
 given by objects  $\cF \in \Sh_{\Lambda_\sE}(T_0)$ with $ \pi_{0!}\cF\simeq 0$.

Since $d \rho_0$ is disjoint from  $\Lambda_{\sE[>r_0]}$,  we have 
$ \pi_{0!}\cF\simeq 0$ if and only if 
 the canonical map $ \cF\to  j[>r_0]_{*}j[>r_0]^* \cF$
is an isomorphism.
But we have seen that then $r_*\cF \in \Sh_{\Lambda}(T_0)^0_!$.
By induction, we then have that 
$$
\xymatrix{
\cF|_{T_0[>r_0]} \simeq s[>r_0] ((r_*\cF)|_{T_0[>r_0]})
}
$$
and so by the definition of $s$, we obtain the assertion.
\end{proof}

 This concludes the proof of the proposition.
\end{proof}

Now we will use the previous proposition
to establish our main goal.

\begin{thm}\label{thm main result}
Pushforward along the almost retraction induces an equivalence
$$
\xymatrix{
r_*:\Sh_{\Lambda_\sE}(M) \ar[r]^-\sim & \Sh_{\Lambda}(M)
}
$$
\end{thm}

\begin{proof}
By induction on the number of strata of $H$.

The base case $H = \emptyset$ is immediate: $r$ is the identity map of $M$.

It suffices to focus on a closed stratum $H_0\subset M$ with tubular neighborhood $T_0\subset M$
and in fact to assume $M= T_0$.

Recall the disjoint union decompositions
$$
\xymatrix{
\Lambda|_{H_0} =  \coprod_{i=1}^k \lambda_i(H_0) 
&
\Lambda|_{T_0} = \coprod_{i=1}^k \Lambda_i
}
$$
such that $\Lambda_i|_{H_0} = \lambda_i(H_0)$. The front projection of $\Lambda_i\subset S^*T_0$
is itself a directed hypersurface $H_i \subset T_0$ with positive coray bundle $\Lambda_i\subset S^*T_0$.

%
%
%In general, 
% by passing to a normal slice, we may assume that the point $x\in H$
%is itself a stratum denoted by $H_0 \subset H$. 
%
%% We will construct an explicit inverse functor denoted by
%% $$
%%\xymatrix{
%%s_B: \Sh_{\Lambda|_B}(B)\ar[r] & \Sh_{\Lambda^{\expand}|_B}(B) 
%%}
%%$$
%By induction, the restriction over the punctured ball
%$B' = B\setminus\{x\}$
% is an equivalence
%$$
%\xymatrix{
%r|_{B'*}: \Sh_{\Lambda|_{B'}}(B')\ar[r]^-\sim & \Sh_{\Lambda^{\expand}|_{B'}}(B') 
%}
%$$
%
%
%Observe that we have disjoint union decompositions
%$$
%\xymatrix{
%\Lambda|_x = T^*_xM \cap \Lambda=\{\lambda_1, \ldots, \lambda_k\}
%&
%\Lambda|_B = T^*B \cap \Lambda = \coprod_{i=1}^k \Lambda_i
%}
%$$
%such that $\Lambda_i \cap T_x^*M = \{\lambda_i\}$. The front projection $H_i = \pi(\Lambda_i)$
%is itself a directed hypersurface with positive coray bundle $\Lambda_i\subset S^*M$. 

Recall  the functorial presentation of any $\cF \in \Sh_\Lambda(T_0)$ as a cone
$$
\xymatrix{
\cF \simeq \Cone( \oplus_{i = 1}^k \cF_i \ar[r] & \cL)
}
$$
where $\cF_i\in \Sh_{\Lambda_i}(T_0)$, and $\cL\in \Loc(T_0)$.

%By Lemma~\ref{}, $r|_{B*}$ is an equivalence on the full subcategory $\Loc(B) \subset \Sh_{\Lambda^\expand|_B}(B)$.

By Proposition~\ref{prop single codir},  
the restriction of $r_*$ to each full subcategory $ \Sh_{\Lambda_{\sE_i}}(T_0)  \subset \Sh_{\Lambda_\sE}(T_0)$
is an equivalence.
We also have the full inclusions $\Loc(T_0) \subset \Sh_{\Lambda_{\sE_i}}(T_0)$.
Thus in particular $r_*$ is essentially surjective.

It remains to check that for distinct  codirections $\lambda_1, \lambda_2:H_0\to \Lambda|_{H_0}$ with $\Lambda_1 =\lambda_1(H_0), \Lambda_2 = \lambda_2(H_0)$,
and $\cF_1 \in  \Sh_{\Lambda_{\sE_1}}(T_0)$,
$\cF_2 \in  \Sh_{\Lambda_{\sE_2}}(T_0)$,
 the functorial map 
$$
\xymatrix{
r_*:\Hom_{\Sh(T_0)}(\cF_1, \cF_2) \ar[r] & \Hom_{\Sh(T_0)}(r_*\cF_1, r_*\cF_2) 
}
$$
is an isomorphism. 

Since the assertion is clear when one sheaf is a local system, it suffices to check it when 
 $\pi_{0*}\cF_1 \simeq 0, \pi_{0*}\cF_2 \simeq 0$.
Note that $\pi_0 r = \pi_0$, so that then
$ \pi_{0*}r_*\cF_1 \simeq 0,  \pi_{0*}r_*\cF_2 \simeq 0$ as well.
Thus by Lemma~\ref{lem orthog of codirs}, we have
$$
\xymatrix{
 \Hom_{\Sh(T_0)}(r_*\cF_1, r_*\cF_2) \simeq 0
}
$$
so we are left to show
$$
\xymatrix{
\Hom_{\Sh(T_0)}(\cF_1, \cF_2) \simeq 0
}
$$

We will now appeal to  the proof of Lemma~\ref{lem orthog of codirs}. It should be possible to directly apply the  proof to $\cF_1, \cF_2$, except the isotopies involved are less clearly non-characteristic. To take care of this, let us note the following inductive simplification.
Recall that $r = \Pi_0   \hat r_0$ where $\hat r_0 = \Pi_{\ul i_1} \cdots  \Pi_{\ul i_N}$. Thus by induction, 
%we know $\hat r_{0 * }$ is an equivalence, and in particular faithful.
%Thus 
it suffices to show $$
\xymatrix{
\Hom_{\Sh(T_0)}(\hat r_{0 * }\cF_1, \hat r_{0 * }\cF_2) \simeq 0
}
$$

Now for the sheaves $\hat r_{0 * }\cF_1, \hat r_{0 * }\cF_2$, we can simply repeat the proof of Lemma~\ref{lem orthog of codirs}
to move $\hat r_{0 * }\cF_2$ through a non-characteristic isotopy to a position
where the vanishing is evident.
%
%To this end, it is convenient to choose a diffeomorphism $\psi:B\to \R^{n+1}$ such that $\psi(x) = 0$ to make our constructions more concrete.
%
%\medskip
%
%{\em (Step 1)} If $\lambda_2 = -\lambda_1$, then proceed to (Step 2) below. 
%Else $\lambda_1, \lambda_2$ are linearly independent so span a two-dimensional plane $P \subset \R^{n+1}$. For $\theta\in [0,1]$, let $R_\theta:\R^{n+1} \to \R^{n+1}$ be the rotation  around $P\subset \R^{n+1}$ fixing $P^\perp\subset \R^{n+1}$, such that $R_0 = \id$, $R_1(\lambda_2) = -\lambda_1$,
%and $R_\theta(\lambda_2)$,  for $\theta\in [0,1]$,  traverses the  short arc of directions in $P$  from $\lambda_2$ to $-\lambda_1$.
%
%Viewing $R_\theta:\R^{n+1}\to \R^{n+1}$ as an isotopy, observe that it satisfies $R_\theta(\Lambda_2) \cap \Lambda_1 = \emptyset$, for $\theta\in [0,1]$.
%Thus $\Hom_{\Sh(B)} (\cF_1, R_{\theta*}(\cF_2))$ is independent of $\theta\in [0,1]$.
%
%\medskip
%
%{\em (Step 2)} By (Step 1), we may assume   $\lambda_2 = -\lambda_1$. We may further assume $\lambda_1 = dx_{0}$
%so $\lambda_2 = -dx_{0}$.
%For $t\in \R$, let $T_t:\R^{n+1} \to \R^{n+1}$ be the translation $T_t(x_0, x_1, \ldots, x_n) = (x_0+t, x_1, \ldots, x_n)$.
%Viewing  $T_t:\R^{n+1} \to \R^{n+1}$ as an isotopy, observe that it satisfies $T_t(\Lambda_2) \cap \Lambda_1 = \emptyset$, for $t\in \R$.
%Thus $\Hom_{\Sh(B)} (\cF_1, T_{t*}(\cF_2))$ is independent of $t\in \R$.
%
%Finally, for $t\gg 0$, observe by the vanishing~\eqref{eqn vanish} that the supports of $\cF^0_1, \cF^0_2$ are disjoint. 
%Hence $\Hom_{\Sh(B)} (\cF_1, T_{t*}(\cF_2)) \simeq 0$ and we are done.
%
This concludes the proof of the theorem.
\end{proof}

%%
%%
%%\begin{corollary}\label{cor main result}
%%The functor
%%$$
%%\xymatrix{
%%r_*:\Sh_{\Lambda^{\expand}}(M) \ar[r]^-\sim & \Sh_{\Lambda}(M)
%%}
%%$$
%%is an equivalence.
%%\end{corollary}
%
%
%\begin{proof}
%All of the dg categories are global sections of sheaves of dg categories.
%Thus this follows immediately from the local statement of the theorem.
%\end{proof}

%%%%%%%%%%%%%%%%%%%%%%%%%%%%%%%%%%%%%%%%%%%%%%%%%%%%

\subsection{Microlocal sheaves}

Let us apply the preceding constructions to microlocal sheaves. 

Let $H\subset M$ be a direct hypersurface with positive coray bundle $\Lambda\subset S^*M$.
Let $\mu\Sh_{\Lambda}$ denote the dg category of microlocal sheaves supported along $\Lambda$.
It is the global sections of a sheaf of dg categories on $\Lambda$.

To understand $\mu\Sh_{\Lambda}$ concretely, let  $p\in \Lambda$ be a point, and
let $\mu\Sh_{\Lambda}|_p$ be the stalk of $\mu\Sh_{\Lambda}$.
Let
 $x\in H$ be the image of $p\in \Lambda$, assume $\Lambda|_x \subset \Lambda$ consists only of $p$,
 and that $M$ is itself a small ball around $x$. 
Then there is
the concrete realization as a quotient category 
$$
\xymatrix{
\mu\Sh_{\Lambda}|_p \simeq \Sh_\Lambda(M)/\Loc(M)
}
$$
There is also
the concrete realization as the full subcategory 
$$
\xymatrix{
\mu\Sh_{\Lambda}|_p \simeq \Sh_\Lambda(M)^0_!\subset\Sh_\Lambda(M)
}
$$
of objects with $\Gamma_c(M, \cF)\simeq 0$.

Now let $\sE\subset M$ be the directed expansion of $H\subset M$ with positive coray bundle $\Lambda_\sE$.
Let $\mu\Sh_{\Lambda_\sE}$ be the dg category of microlocal sheaves along $\Lambda_\sE$.

\begin{prop}\label{prop exp microlocalization}
 Let  $p\in \Lambda$ be a point with image
 $x\in H$. Assume $\Lambda|_x \subset \Lambda$ consists only of $p$,
 and that $M$ is itself a small ball around $x$. 

Then the natural functor is an equivalence
$$
\xymatrix{
\Sh_{\Lambda_\sE}(M)/\Loc(M) \ar[r]^-\sim &  \mu\Sh_{\Lambda_\sE}
}
$$
\end{prop}

\begin{proof}
Regard $x\in H$ as a closed stratum. Let $E_0\subset \sE$ be its expanded stratum.

Choose an open cover $\{B_\kappa\}_{\kappa \in K}$ of a neighborhood of $E_0\subset M$ by a finite collection of small  balls $B_\kappa\subset M$ centered at points of $E_0$. Arrange so that their intersections $B_J = \cap_{\kappa \in J} B_\kappa$, for $J\subset K$, are also small  balls or empty.
Since $\sE$ deformation retracts to $E_0$, we have the  identification
$$
\xymatrix{
 \mu\Sh_{\Lambda_\sE} \simeq \lim_{J\subset K} \Sh_{\Lambda_\sE}(B_J)/\Loc(B_J)
}$$

Any object $\cF\in  \Sh_{\Lambda_\sE(B_J)}/\Loc(B_J)$ admits a canonical representative $\cF^0_![-1] \in  
  \Sh_{\Lambda_\sE}(B_J)^0_!$ defined by  the triangle 
$$
\xymatrix{
\cF\ar[r] & k_{B_J}\otimes \Gamma_c(B_J, \cF)  \ar[r] & \cF^0_!
}
$$
Observe that $\cF^0_! [-1]\in   \Sh_{\Lambda_\sE}(B_J)^0_!$ admits the alternative characterization as the canonical representative vanishing below $E_0$. Such representatives  are compatible and yield an identification
$$
\xymatrix{
 \mu\Sh_{\Lambda_\sE} \simeq \lim_{J\subset K} \Sh_{\Lambda_\sE}(B_J)^0_!
}$$

Similarly, we have parallel equivalences
$$
\xymatrix{
\Sh_{\Lambda_\sE}(M)/\Loc(M)\simeq  \Sh_{\Lambda_\sE}(M)^0_!
\simeq \lim_{J\subset K} \Sh_{\Lambda_\sE}(B_J)^0_!
}
$$

\end{proof}

\begin{corollary}
Pushforward along the almost retraction induces an equivalence
$$
\xymatrix{
r_*:\mu\Sh_{\Lambda_\sE} \ar[r]^-\sim & \mu\Sh_{\Lambda}
}
$$
\end{corollary}

\begin{proof}
Let  $p\in \Lambda$ be a point with image
 $x\in H$. It suffices to prove the assertion when $\Lambda|_x \subset \Lambda$ consists only of $p$,
  and $M$ is itself a small ball around $x$.
Then we have a commutative diagram of equivalences
$$
\xymatrix{
\ar[d]^-\sim \Sh_{\Lambda_\sE}(M)/\Loc(M) \ar[r]^-\sim &  \mu\Sh_{\Lambda_\sE} \ar[d]^-\sim \\
 \Sh_{\Lambda}(M)/\Loc(M) \ar[r]^-\sim &  \mu\Sh_{\Lambda}  \\
}
$$
where the bottom horizontal arrow is the usual quotient presentation.
The top horizontal arrow is an equivalence by Prop.~\ref{prop exp microlocalization}.
The left vertical arrow is an equivalence by Thm.~\ref{thm main result}.
Thus the right vertical arrow is an equivalence.
\end{proof}

%%%%%%%%%%%%%%%%%%%%%%%%%%%%%%%%%%%%%%%%%%%%%%%%%%%%%%%
%%%%%%%%%%%%%%%%%%%%%%%%%%%%%%%%%%%%%%%%%%%%%%%%%%%%%%%
%%%%%%%%%%%%%%%%%%%%%%%%%%%%%%%%%%%%%%%%%%%%%%%%%%%%%%%

%%%%%%%%%%%%%%%%%%%%%%%%%%%%%%%%%%%%%%%%%%%%%%%%%%%%

\section{Appendix: expansion data}\label{s app}

We collect here for convenient reference the hierarchy of constructions and sequentially small constants that enter into the expansion algorithm
of Sect.~\ref{s exp}.

Let $H\subset M$ be a directed hypersurface with positive coray bundle $\Lambda\subset S^*M$.
Fix a Whitney stratification $\{H_{\ul i}\}_{\ul i \in \ul I}$ of the hypersurface $H\subset M$ satisfying the setup of 
Sect.~\ref{ss: strat}. One obtains a compatible decomposition $\{\Lambda_{i}\}_{i\in I}$ 
of the positive coray bundle $\Lambda\subset S^*M$  over the map $I\to \ul I$.
Fix a compatible system of control data $\{(T_{\ul i}, \rho_{\ul i}, \pi_{\ul i}\}_{\ul i\in \ul I}$.

Choose a small $\epsilon>0$.
Fix a compatible family of lines, and construct the almost retraction
$r:M\to M$. % as reviewed in Sect.~\ref{s almost ret}.

Choose a small displacement $d_i>0$,
for each $i\in I$, without  concern for the poset structure of $I$.
These will not be used until the construction of the expanded strata, but should be chosen before the radii chosen immediately below.

Choose a small radius $r_i>0$, for each $i\in I$, following the poset structure on $I$ from minima to maxima.
Construct the truncated cylinders $C_i\subset M$, for $i\in I$, and total cylinder $C\subset M$.

Choose a small value $s_i$,  for each $i\in I$, following the poset structure on $I$ from minima to maxima.
Construct the expanded strata $E_i\subset M$, for $i\in I$, and total expansion $E\subset M$.

Choose a smoothing constant $\delta>0$.
Construct the smoothing homeomorphism $\Psi:M\to M$,  the directed cylinder $\sC= \Psi(C)\subset M$
with positive coray bundle $\Lambda_\sC \subset S^*X$,
and the directed expansion $\sE = \Psi(E) \subset M$
with positive coray bundle $\Lambda_\sE \subset S^*X$.

%
% 
%$\epsilon$: retraction $r$.
%
%displacements $d_i$, chosen independently for each $i\in I$.
%
%radii $r_i$, chosen for each  $i\in I$ following the poset structure on $I$ from minima to maxima.
%
%values $s_i$, chosen for each  $i\in I$ following the poset structure on $I$ from minima to maxima.
%
%smoothing constants $\delta_i$, chosen for each  $i\in I$ following the poset structure on $I$ from minima to maxima.
%
%
%overall $\delta$...

%%%%%%%%%%%%%%%%%%%%%%%%%%%%%%%%%%%%%%%%%%%%%%%%%%%%%%%
%%%%%%%%%%%%%%%%%%%%%%%%%%%%%%%%%%%%%%%%%%%%%%%%%%%%%%%
%%%%%%%%%%%%%%%%%%%%%%%%%%%%%%%%%%%%%%%%%%%%%%%%%%%%%%%

%%%%%%%%%%%%%%%%%%%%%%%%%%%%%%%%%%%%%%%%%%%%%%%%%%%%%%%
%%%%%%%%%%%%%%%%%%%%%%%%%%%%%%%%%%%%%%%%%%%%%%%%%%%%%%%
%%%%%%%%%%%%%%%%%%%%%%%%%%%%%%%%%%%%%%%%%%%%%%%%%%%%%%%
%%%%%%%%%%%%%%%%%%%%%%%%%%%%%%%%%%%%%%%%%%%%%%%%%%%%%%%%
%%%%%%%%%%%%%%%%%%%%%%%%%%%%%%%%%%%%%%%%%%%%%%%%%%%%%%%%
%%%%%%%%%%%%%%%%%%%%%%%%%%%%%%%%%%%%%%%%%%%%%%%%%%%%%%%%

%%%%%%%%%%%%%%%%%%%%%%%%%%%%%%%%%%%%%%%%%%%%%%%%%%%%%%%
%%%%%%%%%%%%%%%%%%%%%%%%%%%%%%%%%%%%%%%%%%%%%%%%%%%%%%%
%%%%%%%%%%%%%%%%%%%%%%%%%%%%%%%%%%%%%%%%%%%%%%%%%%%%%%%

%%%%%%%%%%%%%%%%%%%%%%%%%%%%%%%%%%%%%%%%%%%%%%%%%%%%%%%%
%%%%%%%%%%%%%%%%%%%%%%%%%%%%%%%%%%%%%%%%%%%%%%%%%%%%%%%%
%%%%%%%%%%%%%%%%%%%%%%%%%%%%%%%%%%%%%%%%%%%%%%%%%%%%%%%%

%%%%%%%%%%%%%%%%%%%%%%%%%%%%%%%%%%%%%%%%%%%%%%%%%%%%%%%
%%%%%%%%%%%%%%%%%%%%%%%%%%%%%%%%%%%%%%%%%%%%%%%%%%%%%%%
%%%%%%%%%%%%%%%%%%%%%%%%%%%%%%%%%%%%%%%%%%%%%%%%%%%%%%%

\end{document}